\documentclass[a4paper,10pt]{article}

\usepackage[utf8]{inputenc} 
\usepackage{color}
\usepackage[ruled, vlined]{algorithm2e}
\usepackage{relsize}
\pdfoutput=1
\makeatletter
\newcommand{\LoadPackagesNow}{}
\newcommand{\LoadPackageLater}[1]{%
   \g@addto@macro{\LoadPackagesNow}{%
      \usepackage{#1}%
   }%
}
\makeatother

\usepackage{xspace}
\usepackage{xargs}
\usepackage{ifthen}
\usepackage{etoolbox}

\usepackage[
	table 
]{xcolor}
\usepackage[%
]{graphicx}
\usepackage{float}
\usepackage{wrapfig}

\usepackage{subfigure}
\usepackage{caption}
\usepackage{multicol}
\usepackage[
   tbtags,    
   sumlimits,  
   nointlimits, 
   namelimits, 
   reqno,     
]{amsmath} %

\usepackage{amsfonts}
\usepackage{mathrsfs} 
\usepackage{dsfont}
\usepackage{amssymb}
\LoadPackageLater{amsthm}
\LoadPackageLater{thmtools}
\usepackage[fixamsmath,disallowspaces]{mathtools}
\mathtoolsset{showonlyrefs}
\mathtoolsset{centercolon=true}

\usepackage{geometry}
\usepackage[%
	square,	
	comma,	
	numbers,	
	sort,		
	sort&compress,    
]{natbib}
\let\chapter\section
\usepackage{algorithm2e}
\usepackage{framed}
\usepackage[normalem]{ulem}      
\usepackage{soul}		            
\usepackage{url} 
\usepackage{lipsum}
\usepackage[textsize=tiny,english,colorinlistoftodos,disable]{todonotes}

\usepackage{enumitem}
\definecolor{pdfurlcolor}{rgb}{0,0,0.6}
\definecolor{pdffilecolor}{rgb}{0.7,0,0}
\definecolor{pdflinkcolor}{rgb}{0,0,0.6}
\definecolor{pdfcitecolor}{rgb}{0,0,0.6}
\usepackage[
  colorlinks=true,         
  urlcolor=pdfurlcolor,    
  filecolor=pdffilecolor,  
  linkcolor=pdflinkcolor,  
  citecolor=pdfcitecolor,  %
  raiselinks=true,			 
  breaklinks,              
  verbose,
  hyperindex=true,         
  linktocpage=true,        
  hyperfootnotes=false,     
  bookmarks=true,          
  bookmarksopenlevel=1,    
  bookmarksopen=true,      
  bookmarksnumbered=true,  
  bookmarkstype=toc,       
  plainpages=false,        
  pageanchor=true,         
  pdftitle={Compressed Sensing for Finite-Valued Signals},             
  pdfauthor={Sandra Keiper},            
  pdfcreator={LaTeX, hyperref, KOMA-Script}, 
  pdfdisplaydoctitle=true, 
  pdfstartview=FitH,       
  pdfpagemode=UseOutlines, 
  pdfpagelabels=true,           
  pdfpagelayout=SinglePage, 
]{hyperref}

\LoadPackagesNow

\usepackage{bm}

\makeatletter
\g@addto@macro\bfseries{\boldmath}
\makeatother

\newcommand{\ifargdef}[3][{}]{\ifthenelse{\equal{#2}{}}{#1}{#3}}

\newtheoremstyle{claim}
	{\topsep}{\topsep}%
	{\itshape}
	{}
	{\bfseries\boldmath}
	{}
	{.5em}
	{\thmname{#1} \thmnumber{#2} \thmnote{(#3)}}
\newtheoremstyle{definition}
	{\topsep}{\topsep}%
	{}
	{}
	{}
	{}
	{.5em}
	{{\bfseries\thmname{#1} \thmnumber{#2}} \thmnote{(#3)}}
\newtheoremstyle{remark}
	{\topsep}{\topsep}%
	{}
	{}
	{}
	{}
	{.5em}
	{{\itshape\thmname{#1}:}}
\newtheoremstyle{example}
	{\topsep}{\topsep}%
	{}
	{}
	{}
	{}
	{.5em}
	{{\itshape\thmname{#1}:}}

\declaretheorem[style=claim,numberwithin=section]{theorem}
\declaretheorem[style=claim,sibling=theorem]{proposition}
\declaretheorem[style=claim,sibling=theorem]{lemma}

\declaretheorem[style=claim,sibling=theorem]{remark}
\declaretheorem[style=claim,sibling=theorem]{example}
\declaretheorem[style=definition,sibling=theorem]{definition}

\newcommand{\opleft}[1]{\mathopen{}\left#1}
\newcommand{\opright}[1]{\right#1\mathclose{}}
\newcommandx{\braces}[4]{%
\ifstrequal{#3}{normal}{#1#4#2}{%
\ifstrequal{#3}{auto}{\left#1#4\right#2}{%
\ifstrequal{#3}{opauto}{\opleft#1#4\opright#2}{%
#3#1#4#3#2}}}%
}
\newcommandx{\opannot}[3][3=\downarrow]{\stackrel{\mathclap{\substack{#1 \\ #3 \vspace{2pt}}}}{#2}}
\newcommandx{\lineannot}[3][3=\rightarrow]{\mathllap{\boxed{\text{\textsmaller{#1}}} #3} #2}
\newcommandx{\multilineannot}[4][4=\rightarrow]{\mathllap{\boxed{\parbox{#1}{\RaggedRight\textsmaller{#2}}} #4} #3}
 
\newcommand{\R}{\mathbb{R}} 

\newcommand{\suchthat}[1][normal]{\ifstrequal{#1}{normal}{\mid}{#1|}} 

\newcommand{\dist}[2]{\operatorname{dist}(#1, #2)} 
\newcommandx{\intvcl}[3][1=normal]{\braces{[}{]}{#1}{#2, #3}} 
\newcommandx{\intvop}[3][1=normal]{\braces{(}{)}{#1}{#2, #3}} 
\newcommandx{\intvclop}[3][1=normal]{\braces{[}{)}{#1}{#2, #3}} 
\newcommandx{\intvopcl}[3][1=normal]{\braces{(}{]}{#1}{#2, #3}} 

\DeclareMathOperator*{\argmin}{argmin} 
\DeclareMathOperator*{\argmax}{argmax} 
\DeclareMathOperator*{\convHull}{conv} 
\newcommandx{\abs}[2][1=normal]{\braces{\lvert}{\rvert}{#1}{#2}} 
\newcommandx{\ceil}[2][1=normal]{\braces{\lceil}{\rceil}{#1}{#2}} 
\newcommandx{\floor}[2][1=normal]{\braces{\lfloor}{\rfloor}{#1}{#2}} 
\newcommandx{\round}[2][1=normal]{\braces{[}{]}{#1}{#2}} 
\newcommandx{\der}[1]{D^{#1}} 
\newcommandx{\partder}[4][1={},4={}]{\frac{\partial^{#4} #2}{\partial #3^{#4}}\ifargdef{#1}{\Big|_{#1}}} 
\newcommandx{\integ}[4][1={},2={}]{\int_{#1}^{#2} #3 \, #4} 
\newcommandx{\asympffaster}[2][1=normal]{o\braces{(}{)}{#1}{#2}} 
\newcommandx{\asympfaster}[2][1=normal]{O\braces{(}{)}{#1}{#2}} 
\newcommandx{\asympeq}[2][1=normal]{\Theta\braces{(}{)}{#1}{#2}} 
\newcommandx{\asympsslower}[2][1=normal]{\omega\braces{(}{)}{#1}{#2}} 
\newcommandx{\asympslower}[2][1=normal]{\Omega\braces{(}{)}{#1}{#2}} 

\newcommandx{\norm}[2][1=normal]{\braces{\|}{\|}{#1}{#2}} 
\renewcommandx{\sp}[3][1=normal]{\braces{\langle}{\rangle}{#1}{#2, #3}} 
\newcommand{\adj}[1]{{#1}^\ast} 
\newcommandx{\End}[2][2={}]{\mathcal{L}\opleft( #1 \ifargdef{#2}{, #2} \opright)} 
\DeclareMathOperator{\spann}{\operatorname{span}} 
\newcommandx{\measure}[2][1=normal]{\operatorname{vol}\braces{(}{)}{#1}{#2}} 
\DeclareMathOperator{\supp}{supp} 
\newcommandx{\Leb}[3][1={},3=normal]{L^{#2}\ifargdef{#1}{\braces{(}{)}{#3}{#1}}{}} 
\newcommandx{\Lebnorm}[4][1=normal,3={2},4={}]{\norm[#1]{#2}_{\Leb[#4]{#3}}} 
\renewcommandx{\l}[3][1={},3=normal]{\ell^{#2}\ifargdef{#1}{\braces{(}{)}{#3}{#1}}} 
\newcommandx{\lnorm}[4][1=normal,3={2},4={}]{\norm[#1]{#2}_{\l[#4]{#3}}} 
\newcommandx{\Smooth}[4][1={},3={},4=normal]{C_{#3}^{#2}\ifargdef{#1}{\braces{(}{)}{#4}{#1}}} 
\newcommandx{\Schwartz}[2][1={},2=normal]{\mathscr{S}\ifargdef{#1}{\braces{(}{)}{#2}{#1}}} 
\newcommandx{\Schwartzpoly}[2][1=normal]{\braces{\langle}{\rangle}{#1}{\abs[#1]{#2}} } 
\newcommandx{\Tempdistr}[2][1={},2=normal]{\mathscr{S}'\ifargdef{#1}{\braces{(}{)}{#2}{#1}}} 
\newcommandx{\distrinp}[3][1=normal]{\braces{\langle}{\rangle}{#1}{#2, #3}} 
\newcommand{\Linedistr}[1][]{\mathfrak{L}\ifargdef{#1}{_{#1}}{}} 
\newcommandx{\ft}[3][1=default,2=auto]{
\ifstrequal{#1}{default}{\widehat{#3}}{
\ifstrequal{#1}{long}{{\braces{(}{)}{#2}{#3}}^{\wedge}}{}}} 
\newcommandx{\ift}[3][1=default,2=auto]{
\ifstrequal{#1}{default}{\check{#3}}{
\ifstrequal{#1}{long}{{\braces{(}{)}{#2}{#3}}^{\vee}}{}}} 

\newcommand{\rnd}{\operatorname{round}}
\graphicspath{{images/}} 

\usepackage[normalem]{ulem}

\title{\bfseries Compressed Sensing for Finite-Valued Signals}
\author{\hspace*{-1cm}
  Sandra Keiper$^1$, Gitta Kutyniok$^1$, Dae Gwan Lee$^2$, and G\"otz E. Pfander$^2$\\[.5em]
  \small{\textsc{\hspace*{-1cm}$^1$Technische Universit\"at Berlin, $^2$Philipps-Universit\"at Marburg }}\\[.5em]
}

\begin{document}

\listoftodos
\newcommand{\OpAnalysis}[1]{T_{#1}} 
\newcommand{\OpSynthesis}[1]{\adj T_{#1}} 
\newcommand{\OpFrame}[1]{S_{#1}} 
\newcommand{\defsf}{\varphi} 
\newcommand{\InpSp}{\mathcal{H}} 
\newcommand{\InpSpK}{\InpSp_K} 
\newcommand{\ProK}{P_K} 
\newcommand{\InpSpM}{\InpSp_M} 
\newcommand{\ProM}{P_M} 
\newcommand{\sig}{x^0} 
\newcommand{\sigrec}{x^\star} 
\newcommand{\PF}{\Phi} 
\newcommand{\pf}{\phi} 
\newcommand{\cluster}{\Lambda} 
\newcommand{\concentr}[2]{\kappa\ifargdef{#1}{\opleft( #1, #2 \opright)}} 
\newcommand{\clustercoh}[2]{\mu_c \ifargdef{#1}{( #1 , #2)}} 
\newcommand{\anorm}[2]{\norm{#1}_{1,#2}} 
\newcommand{\ver}{\mathrm{v}} 
\newcommand{\hor}{\mathrm{h}} 
\newcommand{\dir}{\imath} 
\newcommand{\meyerscal}{\phi} 
\newcommand{\Scalfunc}{\Phi} 
\newcommand{\Corofunc}{W} 
\newcommand{\Coro}{\mathscr{K}} 
\newcommand{\conefunc}{v} 
\newcommand{\Conefunc}[1]{V_{(#1)}} 
\newcommand{\Cone}[1]{\mathscr{C}_{(#1)}} 
\newcommand{\pscal}{A} 
\newcommand{\pshear}{S} 
\newcommand{\pscalcone}[2]{A_{#1,(#2)}} 
\newcommand{\shearcone}[1]{S_{(#1)}} 
\newcommand{\unishplain}{\psi} 
\newcommand{\unishplainft}{\ft{\unishplain}} 
\newcommand{\unish}[5][{}]{\unishplain_{#3,#4,#5}^{#2\ifargdef{#1}{,(#1)}}} 
\newcommand{\unishft}[5][{}]{\unishplainft_{#3,#4,#5}^{#2\ifargdef{#1}{,(#1)}}} 
\newcommand{\Scalparamdomain}{\mathsf{A}} 
\newcommand{\aj}{\alpha_j} 
\newcommandx{\Unish}[3][1=\meyerscal,2=\conefunc,3=(\aj)_j]{\operatorname{SH}(#1, #2, #3)} 
\newcommandx{\UnishLow}[1][1=\meyerscal]{\operatorname{SH}_{\mathrm{Low}}(#1)} 
\newcommandx{\UnishInt}[3][1=\meyerscal,2=\conefunc,3=(\aj)_j]{\operatorname{SH}_{\mathrm{Int}}(#1, #2, #3)} 
\newcommandx{\UnishBound}[3][1=\meyerscal,2=\conefunc,3=(\aj)_j]{\operatorname{SH}_{\mathrm{Bound}}(#1, #2, #3)} 
\newcommand{\Unishgroup}{\Gamma} 
\newcommand{\Unishind}{\gamma} 
\newcommand{\lmax}{l_j} 
\newcommand{\weight}{w} 
\newcommand{\wlen}{\rho} 
\newcommand{\model}{{\weight\Linedistr}} 
\newcommand{\modelrec}{\model^\star} 
\newcommand{\Corofilter}{F} 
\newcommand{\mdiam}{h} 
\newcommand{\mask}[1]{\mathscr{M}_{#1}} 
\newcommand{\Unishshort}{\Psi} 
\newcommand{\scalpm}[1]{#1^{\pm1}} 
\newcommand{\translind}[2]{#1^{(#2)}} 
\maketitle
\begin{abstract}
The need of reconstructing discrete-valued sparse signals from few measurements, that is solving  an undetermined system of linear equations, appears frequently in science and engineering. Whereas classical compressed sensing algorithms do not incorporate the additional knowledge of the discrete nature of the signal, classical lattice decoding approaches such as the sphere decoder do not utilize sparsity constraints.

In this work, we present an approach that incorporates a discrete values prior into basis pursuit. In particular, we address unipolar binary and bipolar ternary sparse signals, i.e., sparse signals with entries in $\{0,1\}$, respectively in $\{-1,0,1\}$. We will show that phase transition takes place earlier than when using the classical basis pursuit approach and that, independently of the sparsity of the signal, at most $N/2$, respectively $3N/4$, measurements are necessary to recover a unipolar binary, and a bipolar ternary signal uniquely, where $N$ is the dimension of the ambient space. We will further discuss robustness of the algorithm and generalizations to signals with entries in larger alphabets.
\end{abstract}

\vspace{.1in}
\noindent\textbf{Keywords.}
Compressed Sensing, Sparse Recovery, Null Space Property, Finite Alphabet, Statistical Dimension, Phase Transitions\\
\textbf{AMS classification.} 15A12, 15A60, 15B52, 42A61, 60B20, 90C05, 94A12, 94A20
\vspace{.1in}
%
%
%
%
%


\section{Introduction}
About ten years ago compressed sensing was introduced as an effective tool to acquire signals from an underdetermined system of linear equations, under the additional, by applications motivated, constraint, that the signal possesses a sparse or nearly sparse representation. More precisely, the main objective is to solve the underdetermined system
\begin{align}
Ax=b,
\end{align}
with $A\in \R^{m\times N}$ ($m<N$) and $b\in \R^m$ by using the a priori information that $x$ is $k$-sparse, i.e., at most $k$ entries of $x$ are non-zero.
In this situation, necessary and sufficient conditions, for instance, null space and incoherence properties of the measurement matrix $A$ for the exact recovery of the signal $x$, even when $b$ is contaminated with noise, are known. For a survey, we refer to \cite{DDEK}.

In many applications we can assume a secondary structure constraint besides sparsity, namely that the nonzero entries of $x$ come from a finite or discrete alphabet. Those signals appear, for example, in error correcting codes \cite{CRTV} as well as massive Multiple-Input Multiple-Output (MIMO) channel \cite {RHE} and wideband spectrum sensing \cite{ALLP}. A particular example is given by wireless communications, where the transmitted signals are sequences of bits, i.e., with entries in $\{0,1\}^N$. In this regime, sensor networks have gained some interest over the last years. In those, one observes the behavior that a large number of sensor nodes are either silent or transmit data (often $\pm 1$) to a receiver. If the receiver is equipped with a small number of antennas, the detection of the active sensors and their transmitted data results in an underdetermined inverse problem. However, there also exist several examples of applications, where the transmitted data originate from a general finite set $\mathcal{A}\subset \R$ such as in source decoding or radar.

In the following we will focus on signals with entries from a bounded lattice and show that compressed sensing recovery guarantees for those signals can be improved significantly in some cases.

\subsection{Finite-Valued Sparse Signals}

In this paper we derive recovery guarantees for structured sparse signals from an underdetermined system of linear equations. We will focus on the structural assumption that the entries of the original sparse signal $x\in \R^N$ stem from a finite alphabet $\mathcal{A}$, more precisely, $\mathcal{A} \subset \R$ is a finite set of real numbers. We first consider the cases $\mathcal{A} = \{0,1\}$ and $\mathcal{A}=\{-1,0,1\}$, since those will already illustrate the main concepts and arguments of our approach. Surprisingly, it will turn out that, in particular, the alphabets $\mathcal{A} = \{-1,0,1\}$ and $\{0,1,2\}$ exhibit quite different phenomena due to the positioning of the zero within the set. Certainly, all results hold true for any general finite alphabet, and we will then discuss this situation in the second part of the paper.

For notational purposes, we will call $x$ a \emph{unipolar binary signal}, if $x\in \{0,1\}^N$, a \emph{bipolar ternary signal}, if $x\in \{-1,0,1\}^N$, a \emph{unipolar finite-valued signal}, if, for $L\in \mathbb{Z}$, $x\in \{0,\dots, L\}^N$ and a \emph{bipolar finite-valued signal}, if, for $L_1,L_2\in \mathbb{N}$, $x\in\{-L_1,\dots, L_2\}$. Moreover, throughout our paper, sparsity will be imposed directly on the signal $x$ with respect to the canonical basis of $\mathbb{R}^N$. We may alternatively assume that $x$ possesses a sparse representation, that is
\begin{align}
x=Gv \quad \text{for some} \;\; G\in \R^{N \times N} \;\; \text{and} \;\; \text{a $k$-sparse vector} \;\; v \in \mathcal{A}^N,
\end{align}
in which case we simply replace the measurement matrix $A$ with $AG$. We remark that unipolar binary signals are also considered in the framework of \emph{1-bit compressed sensing} \cite{BB}. However, in this problem complex, the quantized measurement vector $b$ is unipolar binary rather than the signal $x$ itself.


\subsection{Recovery of Finite-Valued Signals using Basis Pursuit}\label{sec:BasPur}

A natural approach to recover sparse signals from an underdetermined linear system is to use $\ell_0$-minimization, i.e., to solve the problem
\begin{align}
 \min\|x\|_0 \text{ subject to } Ax=b  \label{P0}\tag{$P_0$}.
\end{align}
This problem, however, is known to be NP-hard \cite{Natarajan}.

A popular and by now well-understood approach is to relax \eqref{P0} to
\begin{align}
  \min\|x\|_1 \text{ subject to } Ax=b  \label{P1}\tag{$P_1$},
\end{align}
which is known as \emph{basis pursuit} \cite{CDS}. As this problem is convex, it can be solved easily with the help of convex
optimization methods. A necessary and sufficient condition under which $x_0$ is uniquely recovered by basis pursuit is given as follows:
The set of all feasible solutions, $x_0+\ker(A)$, intersects with the set $\{x:\|x\|_1\le \|x_0\|_1\}$ exactly at $x_0$ (cf.~Figure \ref{fig:NSP}).
This condition provides a useful geometric intuition about properties of measurement matrices to ensure uniqueness of the solution.
One of those properties is the so-called \emph{null space property (NSP)} given by
\noeqref{NSP}%
\begin{align}\label{NSP}\tag{NSP}
 \ker(A)\cap \{w\in \R^N: \|w_K\|_1\ge \|w_{K^C}\|_1\}=\{0\}.
\end{align}
It is well-known that if $A$ fulfills the NSP with respect to some subset $K\subset [N]$, where $[N]=\{1,\dots,N\}$, then every signal $x_0$
supported on $K$ is the unique minimizer of \eqref{P1} with $b=Ax_0$ (cf. \cite{FR}).

\begin{figure}[ht]
\begin{center}
\includegraphics[scale=0.5]{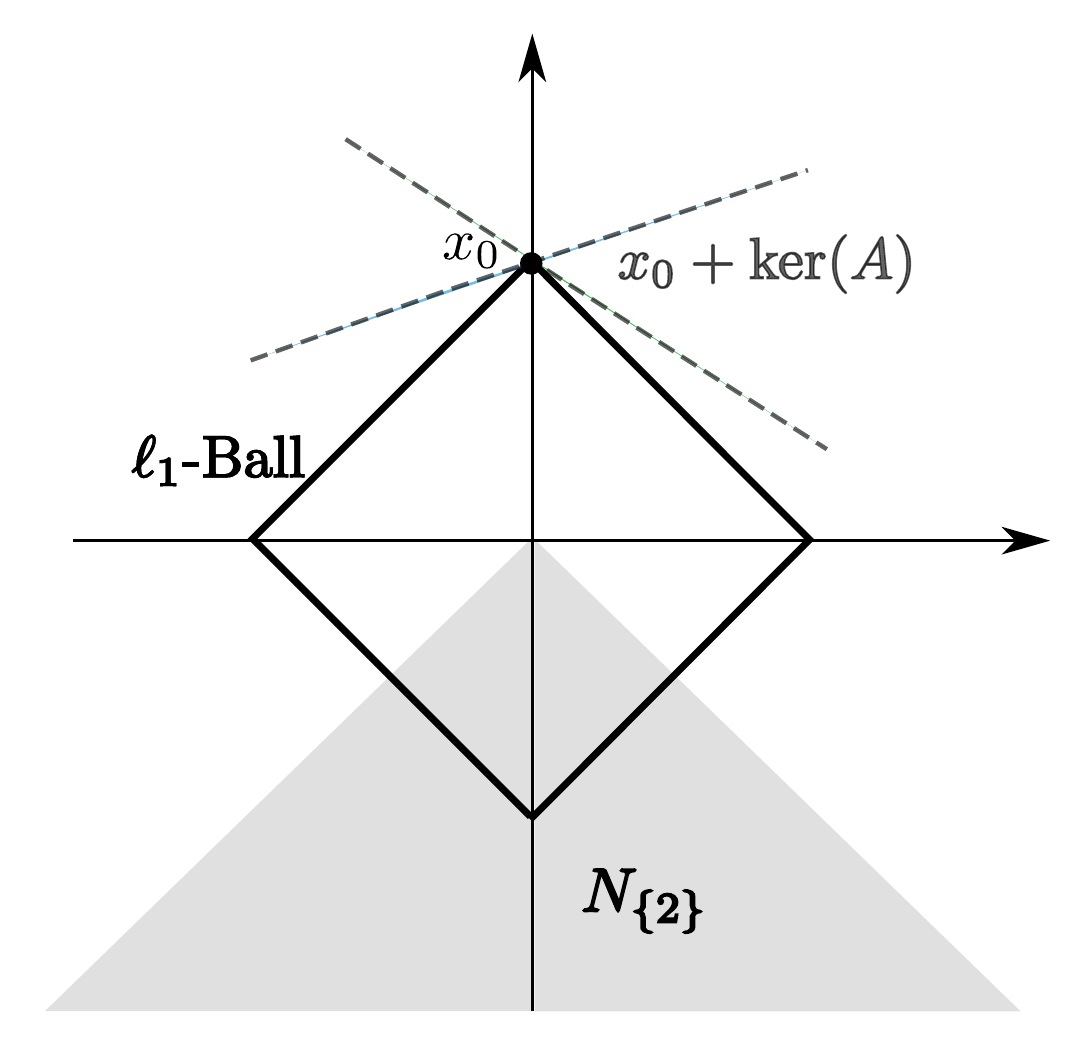}
\caption{Illustration of the NSP condition. If the set of feasible solutions $x_0+\ker(A)$ does not intersect the $\ell_1$-ball, then $x_0$ is the solution with the smallest $\ell_1$-norm and therefore the unique minimizer of \protect\eqref{P1}. This is equivalent to the condition that the kernel of $A$ does not intersect the set $N_{\{2\}}$, which is the descent cone corresponding to basis pursuit (cf.~Section \protect\ref{sec:IntPhaseTrans}).}\label{fig:NSP}
\end{center}
\end{figure}

By using random matrices $A$ such as a matrix with Gaussian iid entries, it is possible to achieve a very high probability of $A$ having NSP and therefore of \eqref{P1} to succeed given that the number of measurements satisfies $m\ge C k\log(N)$, where $k$ is the sparsity of the signal $x_0$ and $C$ some positive constant not depending on $k$ and $N$ \cite{CRT}. In the following, we aim to decrease the number of measurements $m$ further using additional structural assumptions. For finite-valued signals, we ask the following two questions:
\begin{enumerate}
 \item Can performance guarantees of basis pursuit be improved when a signal is finite-valued, i.e., are less measurements or a weaker NSP condition needed
 to recover an $\mathcal{A}$-valued sparse signal via \eqref{P1}?
 \item Can quantization of the output of basis pursuit help to improve the recovery performance?
\end{enumerate}

Unfortunately, the answers to both questions are not affirmative. The second question was already studied in \cite{FK}, where the authors showed that
basis pursuit followed by a quantization as post-processing does not help to recover the exact solution; one intuition behind this result being that
the finite nature of the signal is not incorporated in the reconstruction algorithm.

The first question is answered in part by the following result. For this, we consider an adaption of basis pursuit, which performs on positive signals 
much better than classical basis pursuit \cite{Stojnic+}, namely,
\begin{align}
 \min\|x\|_1 \text{ subject to } Ax=b \text{ and } x\in \R_+^N  \label{P+}\tag{$P_{+}$}.
\end{align}
For this recovery strategy, we can prove that provided \eqref{P+} yields the unique solution for unipolar binary signals with support $K\subset [N]=\{1,\dots,N\}$,
then it will also recover every other positive-valued signal supported on $K$. Thus, to hope for unique recovery by basis pursuit using less measurements
than for positive-valued signals is not reasonable. A short proof of the following result can be found in the Appendix \ref{sec:app0}.

\begin{proposition}\label{prop:classicL1}
 Let $A\in\R^{m \times N}$ be an arbitrary measurement matrix and $K\subset [N]$. Then the following conditions are equivalent:
 \begin{enumerate}
  \item[(i)] Any vector $x_0\in \R^N_+$ with $\supp x_0\subseteq K$ is the unique solution of \eqref{P+}.
  \item[(ii)] $\mathds{1}_K$ is the unique solution of \eqref{P+} with $b = A x_0$, where $\mathds{1}_K$ denotes the unipolar binary $k$-sparse vector in $\R^N$ 
  whose nonzero entries are supported exactly on $K$.
 \end{enumerate}
\end{proposition}

These considerations imply that a better performance for unipolar binary signals can only be achieved if the finite nature of the unipolar binary signals
is incorporated into basis pursuit. One first idea could be to solve the problem given by
\begin{align}
   \min\|x\|_1 \text{ subject to } Ax=b \quad \text{and} \quad x\in \mathcal{A}^N.
\end{align}
Unfortunately, this is a very hard problem due to the non-convexity of $\mathcal{A}$. To resolve the problem of
nonconvexity of the set $\mathcal{A}^N$, one can consider the same problem with a convexified version of this set which is
$\convHull{\mathcal{A}^N}$ as the convex hull of $\mathcal{A}^N$. Then the minimization problem reads
\begin{align}
 \min\|x\|_1 \text{ subject to } Ax=b \text{ and } x\in \convHull{\mathcal{A}}^N. \label{PA}\tag{$P_{\mathcal{A}}$}
\end{align}

We wish to mention that this reformulation was first considered by Stojnic \cite{Stojnic_Binary} with $\mathcal{A}=\{0,1\}$, 
in which case basis pursuit is adapted to
\begin{align}
 \min\|x\|_1 \text{ subject to } Ax=b \text{ and } x\in [0,1]^N  \label{Pbin}\tag{$P_{\text{bin}}$}.
\end{align}
Stojnic introduced a new NSP condition, which guarantees exact recovery of unipolar binary signals with support being a fixed 
set $K\subset [N]$ by using \eqref{Pbin} \cite{Stojnic_Binary}. In this paper, the author also derives a sufficient condition on 
the sparsity to ensure successful recovery through \eqref{Pbin}, by computing the Gaussian width of the so-called descent cone. 
His approach is however not only quite complicated and does not easily allow an interpretation, e.g., in terms of the number of
necessary measurements; it is also entirely restricted to the binary case.

\subsection{Phase Transition in Convex Programs} \label{sec:IntPhaseTrans}

As mentioned in Subsection \ref{sec:BasPur}, basis pursuit recovers a vector $x_0$ uniquely if and only if the set of feasible solutions $\{x : Ax=b\}$
intersects with an $\ell_1$-ball precisely in $x_0$. The latter condition can also be described in terms of the descent cone of the
$\ell_1$ norm as illustrated in Figure \ref{fig:NSP}. The \emph{descent cone} of a convex function $f:\R^N\rightarrow \bar{\R}$ at a
point $x$ is the conic hull of the perturbations of $f$ which do not increase $f$ near $x$, i.e.,
\begin{align}
 \mathcal{D}(f,x)= \bigcup_{\tau >0}\{y\in \R^N: f(x+\tau y)\le f(x)\}.
\end{align}

To analyze the success of a general convex program of the form
\begin{align}\label{Pf}
 \min f(x) \quad \text{subject to} \quad Ax=b,
\end{align}
where $f:\R\rightarrow \bar{\R}$ is a convex function, the fact that a vector $x_0$ is the unique optimal solution of the convex program
\eqref{Pf} if and only if $\mathcal{D}(f,x_0)\cap \ker{A}=\{0\}$ can be employed \cite{ALMT}. Thus, if the measurement matrix $A$ is a
random matrix, the probability that $\mathcal{D}(f,x_0)\cap \ker{A}=\{0\}$ would need to be computed.

In this work we focus on \emph{Gaussian matrices} $A\in \R^{m \times N}$, i.e., on matrices whose entries are drawn from the standard normal
distribution:
\begin{align}
 A=m^{-1/2}\; [a_{i,j}]_{i,j=1}^{m,N} \quad \text{with} \quad  a_{i,j} \sim \mathcal{N}(0,1).
\end{align}
It was shown in \cite{ALMT} that the probability of $\mathcal{D}(f,x_0)\cap \ker{A}=\{0\}$ can be computed using the so-called \emph{statistical
dimension} of  $\mathcal{D}(f,x_0)$. Recall that the \emph{statistical dimension} $\delta(C)$ of a closed convex cone $C\in \R^N$ is defined by
\begin{align}
 \delta\left(C\right)=\mathbb{E}\left[\|\pi_C(g)\|_2^2\right],
\end{align}
where $g\in \R^N$ is a standard normal vector and $\pi_C$ denotes the Euclidean projection onto the cone $C$. 

More precisely, \cite[Thm. II]{ALMT}  states that the transition from failure to success of
\eqref{Pf} occurs when the number of measurements equals the statistical dimension $\delta\left(\mathcal{D}(f,x_0)\right)$ and that the shift
from failure to success takes place over a range of about $\mathcal{O}(\sqrt{N})$ measurements. It was also shown in \cite{ALMT} that an
upper bound for the statistical dimension of a descent cone can be derived as follows:
\begin{enumerate}
 \item Compute the subdifferential $\partial f(x_0)=\{s\in \R^N: f(y)- f(x_0)\ge\langle s,y-x_0 \rangle \text{ for all } y\in \R^N\}$ of $f$ at $x_0$.
 \item For each $\tau\ge 0$, compute $J(\tau)=\mathbb{E}[\dist{\mathbf{g}}{\tau \, \partial f(x_0)}^2]$, where $\mathbf{g}$ is a standard normal vector.
 \item Then $\inf_{\tau\ge0}J(\tau)$ is an upper bound for $\delta(\mathcal{D}(f,x_0))$.
\end{enumerate}

We will see that for unipolar binary and bipolar ternary signals, interestingly, the bound $\Delta=\inf_{\tau\ge0}J(\tau)$ depends only on the size of
the support of $x_0$. Therefore, in the sequel, we denote it as $\Delta_{\text{bin}}(k)$ or $\Delta_{\pm\text{ter}}(k)$, where $k$ is the
cardinality of the support of $x_0$.

\subsection{Previous Work}

Until today sparsity and finiteness have been considered mostly separately. Compressed sensing focusses almost entirely on sparsity without considering
finiteness \cite{Donoho,FR}, whereas approaches such as lattice decoding  \cite{AEVZ,YW} utilizes the finite nature of a signal without taking sparsity into
account.

On the theoretical side, besides the previously mentioned 1-bit compressed sensing, one approach which deals with signals having entries
from a finite alphabet $\mathcal{A}=\{0,1,\dots, p\}$, is considered, for example, in the papers \cite{DM,TLL}. The authors assume that
$\mathcal{A}$ is a field, i.e., $p$ is assumed to be prime, which is very different from our assumption. More closely related is the
publication \cite{MR}, in which signals with entries in $\mathcal{A}=\{-1,1\}$ have been considered as so-called \emph{saturated vectors}.
For those, it was shown that $\ell_{\infty}$ minimization, i.e.,
\begin{align}
 \min\|x\|_{\infty} \text{ subject to } Ax=b   \label{PInf}\tag{$P_{\infty}$}
\end{align}
is perfectly suited \cite{JJF}. The authors in \cite{MR} have shown -- similar to our result in the situation of unipolar binary signals -- that 
at most $\frac{N}{2}$ measurements will be sufficient
to almost certainly recover any signal in $\{\pm1\}$. In our work, we will consider, among others, signals with entries in $\mathcal{A}=\{-1,0,1\}$.
Hence, in the case $|\supp{x_0}|=N$, i.e., $x_0$ having full support, it will necessarily only consist of entries in $\{-1,1\}$.
In this case, our results then reduce to the statement that $\frac{N}{2}$ measurements are sufficient with high probability.

A slightly different approach, which will though be very useful to us is based on the geometrical view of basis pursuit, as, e.g., 
carried out by Donoho et al. \cite{DT}. In \cite{DTHC} a geometrical analysis of basis pursuit with so-called
box constraint was performed, which will be the tool of choice in the underlying work. In this work the uniform recovery of
positive-valued sparse signals and so-called $k$-simple signals have been considered. A signal $x\in \R^N$ is called $k$-simple,
if $x\in [0,a]^N$ and at most $k$ of its entries are not in $\{0,a\}$, for some $a\in \R_+$. However, those results do not provide
phase transitions for special classes of matrices, they rather give criteria that matrices need to fulfill.

In \cite{LPST} conditions for the unique recoverability of integer-valued signals have been studied. The therein considered problems 
are in general NP-hard. However, for binary variables, medium-sized problems are shown to be solvable in reasonable time.

On the algorithmic side, the \emph{sphere decoder} \cite{AEVZ} is a useful algorithm to generally recover integer signals from noisy
linear measurements. There have been some attempts to incorporate sparsity constraints into the sphere decoder \cite{TLL, ZG}. However,
underdetermined systems still remain very problematic, even if there do exist some papers dealing with this situation: In \cite{WL},
the authors suggest to artificially add extra equations to the system, which probably cannot achieve a good performance for the sparse
setting we are considering. Another idea, presented in \cite{DAB}, is to combine sphere decoding on a part of the signal of length $m$
with a brute force search on the rest of the signal. However, this forces an immense growth of the complexity of computations. In
\cite{ZLH} another way of determining the remaining part of the signal has been suggested. However, the methods remain very heuristically
and theoretical performance analysis are rare.

There already exist some few cases in which compressed sensing algorithms were adapted to the finite or rather discrete situation. One case is orthogonal
matching pursuit (OMP), which has been considered in connection with quantization, soft feedback \cite{SFSD} and the sphere decoder
\cite{SFSF}. Additionally, in \cite{FK} the knowledge of the discrete nature has been used to initialize the support set for the OMP
algorithm.  This approach is able to slightly beat conventional compressed sensing algorithms. The first mentioned approaches show
improvements of the symbol error rates by incorporating the finite or discrete structure in OMP, however, they do not consider the reduction of
the number of measurements.

\subsection{Our Contribution}

Our work analyzes the recovery of finite-valued $k$-sparse signals using basis pursuit with box constraints in full generality. The 
related alphabets belong to either the unipolar ($\mathcal{A}=\{0,\dots, L\}, L\in \mathbb{Z}$) or the bipolar ($\mathcal{A}=\{-L_1,\dots, L_2\}^N$, 
$L_1,L_2\in\mathbb{N}$) situation.

Our main contributions can be summarized as follows.
\begin{itemize}
\item {\em Null Space Properties.} We provide null space properties for the recovery of finite-valued $k$-sparse signals, which allow equivalent
conditions for unique recoverability of such signals (see Theorems \ref{thm:BinNec}, \ref{thm:BTNSP} for the unipolar binary and bipolar ternary cases, and Theorems \ref{thm:NSPUF}, \ref{thm:NSPBF} for the general situation). The sufficiency of the null space property by Stojnic \cite{Stojnic_Binary} in the unipolar
binary case is a special case of our results. 
\item {\em Stability.} All versions of basis pursuit with box constraints --~adapted to the specific alphabet considered~-- are shown to be stable
under noisy measurements with precise error bounds.
\item {\em Phase Transitions.} We analytically compute the phase transitions of all versions of basis pursuit with box constraints using the statistical
dimension of the associated descent cone as the main methodology (see Theorems \ref{mainthm1}, \ref{mainthm2} for the unipolar binary and bipolar ternary cases, and Theorems \ref{mainthm3}, \ref{mainthm4} for the general situation). Those are
then verified numerically in Section \ref{sec:numerics}.
\item {\em Comparison of Different Alphabets.} Our analysis will surprisingly show that the unipolar situation is very different from the bipolar one. 
One of our findings is that the positioning of the zero -- i.e., whether it is a boundary element or not -- is crucial. A second key observation is
the fact that mainly the boundary elements play a role in the sense of $-L_1$ and $L_2$ in the case of  bipolar finite-valued signals.
\end{itemize}

\section{Binary Sparse Signals}

\subsection{Adapted NSP Condition}

We start our considerations with unipolar binary sparse signals, i.e.,  $x_0\in \{0,1\}^N$. The key objective is to solve the 
underdetermined system of linear equations
\begin{align}
 Ax_0=b,
\end{align}
with $A\in \R^{m\times N}$ and $b\in \R^m$, $m<N$, under the additional assumption that $x$ is sparse and unipolar binary, i.e.,
$x_0\in \mathcal{A}^N$, $\mathcal{A}=\{0,1\}$, and $\|x_0\|_0\le k$. As already described in the introduction, a natural
approach is to exploit basis pursuit under the additional constraint that $x_0\in \mathcal{A}^N$. However, to derive a
convex program we need to 'convexify' $\mathcal{A}$, which yields \emph{binary basis pursuit} as introduced in
\cite{Stojnic_Binary}, namely,
\begin{align}
 \min\|x\|_1 \text{ subject to } Ax=b \text{ and } x\in [0,1]^N \tag{$P_{\text{bin}}$}.
\end{align}
Notice that this minimization does not solely require that $\mathcal{A}$ is of the form $\{0,1\}$, but is in general also 
applicable to non-integer alphabets such as $\{0,\frac12, 1\}$; certainly, with a very different analysis of recoverability.

To address the question under which conditions this program uniquely recovers a given unipolar binary sparse signal, we take a
closer look at the null space property (NSP). In \cite{Stojnic_Binary} the following definition for a weakened NSP has
been introduced.

\begin{definition}\label{def:BNSP}
A matrix $A\in \R^{m\times N}$ is said to satisfy the \emph{binary NSP} with respect to some set $K\subset [N]$, if
\noeqref{BNSP}
\begin{align}
 \ker(A)\cap N_K\cap H_K =\{0\} \label{BNSP} \tag{B-NSP},
\end{align}
where $N_K=\{w\in \R^N:\|w_K\|_1\ge\|w_{K^C}\|_1\}$ and $H_K=\{w\in \R^N: w_i\le 0\text{ for } i\in K, \text{ and }
w_i\ge 0 \text{ for } i\in K^C\}$.
\end{definition}

Observe that the NSP can be rewritten as $\ker(A)\cap N_K=\{0\}$. Thus, B-NSP is indeed weaker than NSP. Further note
that we could have alternatively formulated B-NSP by substituting $N_K$ by $N^+=\{w\in\R^N:0\ge\sum_{i=1}^Nw_i\}$.
However, the formulation of NSP conditions for other specially structured signals requires the use of $N_K$ (e.g.,
bipolar ternary) and sometimes of $N^+$ (e.g., unipolar finite-valued signals).

\begin{figure}[ht]
\begin{center}
 \includegraphics[scale=0.5]{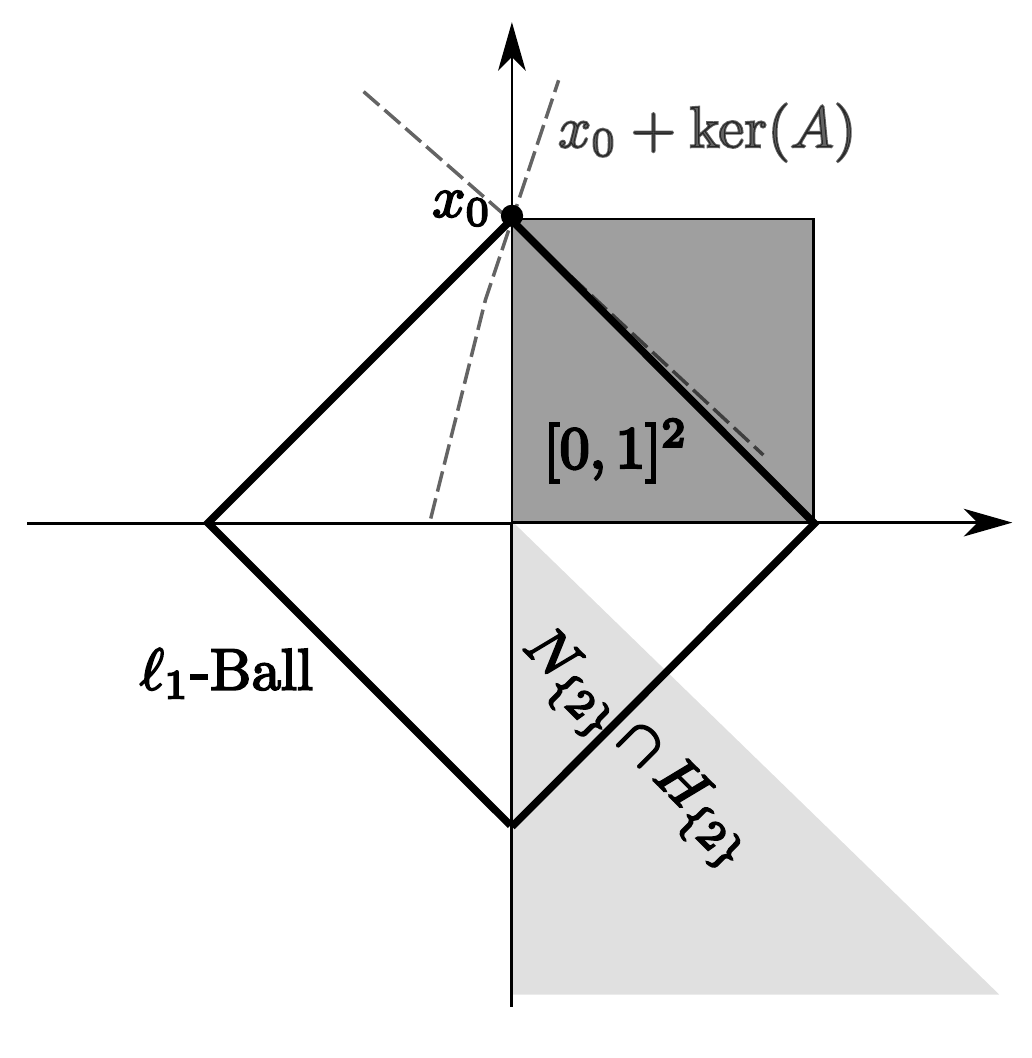}
 \caption{Geometrical Interpretation of the B-NSP. If the set  $x_0+\ker(A)$ does not intersect the $\ell_1$-ball and the constraint set $[0,1]^2$ simultaneously, $x_0$ is the unique solution of \protect\eqref{Pbin}. This is equivalent to the condition that the kernel of $A$ does not intersect the descent cone $N_{\{2\}}\cap H_{\{2\}}$ corresponding to \protect\eqref{Pbin}.}\label{fig:NSPBin}
 \end{center}
\end{figure}

The following theorem states that B-NSP is indeed sufficient to recover unipolar binary sparse signals. Its proof is very easy
and omitted in \cite{Stojnic_Binary}. We will though still provide a proof, which will then motivate our further considerations.

\begin{theorem}{\cite{Stojnic_Binary}}\label{StojBin}
 Let $A\in \R^{m\times N}$ be a measurement matrix which fulfills B-NSP with respect to a set $K\subset [N]$.
 Then $\mathds{1}_K$ is the unique solution of \eqref{Pbin} with $b=A\mathds{1}_K$ and also the unique unipolar binary solution of \eqref{P0}.
\end{theorem}

\begin{proof}
Assume that $A$ satisfies B-NSP with respect to $K$. Further, let $x$ be a solution of \eqref{Pbin} with $b = A\mathds{1}_K$.
Then it follows that $x-\mathds{1}_K \in \ker{A}$ as well as
\begin{enumerate}
 \item[(a)] for $i\in K^C$ it holds $(x-\mathds{1}_K)_i=x_i\ge0$, \label{prop1}
 \item[(b)] for $i\in K$ it holds $(x-\mathds{1}_K)_i=x_i-1\le 0$, \label{prop2}
\end{enumerate}
since $x$ is the solution of \eqref{Pbin} and therefore all its entries lie between zero and one. This shows that $x-\mathds{1}_K \in \ker{A}
\cap H_K$. Together with B-NSP, we derive $x-\mathds{1}_K \notin N_K$, that is,
\begin{align}
 \sum_{i\in K}|(x-\mathds{1}_K)_i| < \sum_{i\in K^C}|(x-\mathds{1}_K)_i|.
\end{align}
Using (a) and (b), this is equivalent to
\begin{align}
-\sum_{i\in K}(x-\mathds{1}_K)_i< \sum_{i\in K^C}x_i,
\end{align}
which implies
\begin{align}
\|\mathds{1}_K\|_1=k=&\sum_{i\in K}(\mathds{1}_K)_i< \sum_{i=1}^N x_i = \|x\|_1.
\end{align}
This shows that $\mathds{1}_K$ is the unique solution of \eqref{Pbin} with $b = A\mathds{1}_K$. Similarly, if $x$ is a unipolar binary solution of
\eqref{Pbin} with $b = A\mathds{1}_K$, we obtain
\begin{align}
\|\mathds{1}_K\|_0 = k < \|x\|_1 \le \|x\|_0.
\end{align}
Therefore $\mathds{1}_K$ is the unique unipolar binary solution of \eqref{P0} with $b = A\mathds{1}_K$.
\end{proof}

Notice that condition (b) in the proof hinges on the fact that $\mathds{1}_K$ is a unipolar binary vector. This argument does not apply to signals
which have entries in a set of cardinality larger than two. Further note that it is essential that the signal is exactly supported on $K$.
If the signal happens to be supported on a proper subset of $K$, we cannot deduce that this signal is the unique solution of \eqref{Pbin}
(also compare with Theorem \ref{thm:Uniform})

In fact, B-NSP is optimal in the sense that unique recovery of a unipolar binary $k$-sparse signal supported on $K$ also implies B-NSP with respect
to $K$. Thus, we can show the following equivalence.

\begin{theorem}\label{thm:BinNec}
Let $A\in \R^{m\times N}$. Then the following conditions are equivalent:
\begin{enumerate}
\item[(i)] The measurement matrix $A$ satisfies B-NSP with respect to a set $K\subset [N]$.
\item[(ii)] The vector $\mathds{1}_K$ is the unique solution of \eqref{Pbin} with $b=A\mathds{1}_K$.
\end{enumerate}
If one of those cases is true, then $\mathds{1}_K$ is also the unique unipolar binary solution of \eqref{P0}.
\end{theorem}

\begin{proof}
(i) $\Rightarrow$ (ii). This is Theorem \ref{StojBin}.

(ii) $\Rightarrow$ (i). Suppose to the contrary that $u \in \ker(A)\cap N_K\cap H_K$ is a nonzero element. Since $\ker(A)\cap N_K\cap H_K$ is
a linear subspace of $\R^N$, we may assume that $u \in [-1,1]^N$, i.e., $\|u\|_{\infty}\le 1$. As $u\in H_K$, we have $u_i\in [-1,0]$ for all
$i\in K$ and $u_i\in [0,1]$ for all $i\in K^C$. Thus $w=\mathds{1}_K + u \neq \mathds{1}_K$ is an element of $[0,1]^N$. Next, $u\in \ker(A)$
implies that $Aw= A\mathds{1}_K$, which implies that $w$ is a feasible solution of \eqref{Pbin}. The condition $u\in N_K$ then yields
\begin{align}
 \|w\|_1=\sum_{i\in K} 1+ \sum_{i\in K}u_i +\sum_{i\in K^C}u_i \le \sum_{i\in K} 1 =\|\mathds{1}_K\|_1,
\end{align}
contradicting the assumption.
\end{proof}

We have shown that B-NSP is optimal, in the sense that we have equivalence between the uniqueness of a unipolar binary $k$-sparse solution and B-NSP.
But it is not clear so far that B-NSP only ensures unique recovery for unipolar binary signals. It may also be possible that all signals with entries
between zero and one that are supported on $K$ can be uniquely recovered if $A$ satisfies B-NSP with respect to $K$. In the following we
provide a counterexample.

\begin{example}\rm
The B-NSP is indeed weaker than NSP and is perfectly suited for the recovery of unipolar binary signals. Theorem \ref{thm:BinNec} states that
unique recovery of a unipolar binary signal $\mathds{1}_K$ using \eqref{Pbin} is possible if and only if $A$ has B-NSP with respect to the set $K$. The following
toy example shows an even stronger result, namely that provided $A$ satisfies B-NSP with respect to some subset $K\subset [N]$, even unique
recovery of $\lambda \mathds{1}_K$ for every $0\le \lambda < 1$ is not guaranteed.

For this, let $A\in\R^{2\times4}$ with
 $$\ker(A)=\spann\left\{u=\begin{bmatrix}
                       -1\\-1\\0\\3
                      \end{bmatrix},v= \begin{bmatrix} 1\\-3\\0.5\\1.5\end{bmatrix}
\right\}.$$
Then $v$ does not fulfill NSP with respect to the set $K=\{1,2\}$, nor NSP$_+$, which is sufficient to uniquely recover every positive-valued
signal supported on $K$ (see Theorem \ref{thm:Uniform}~(iii) for a definition). However, $A$ fulfills B-NSP. 
To show this let $N_K$ and $H_K$ be defined as in Definition \ref{def:BNSP}. If $w\in \ker(A)\cap H_K$, there are $\alpha, \beta\in \R$ such that $w=\alpha u+\beta v$ and it holds $w_1, w_2\le 0$ and $w_3, w_4\ge 0$. From $0\le w_3=\beta v_3=0.5\beta$ and $0\ge w_1=\alpha u_1+\beta v_1=-\alpha+\beta$ it, then, follows immediately that $0\le \beta\le \alpha$. 

Now assume towards a contradiction that $w\in N_K$, i.e., that
$$-w_1-w_2\ge w_3+w_4,$$ which is equivalent to $$\alpha-\beta+\alpha+3\beta\ge \tfrac12\beta+3\alpha+\tfrac32\beta,$$ which yields $$0\ge \alpha.$$
This is a contradiction, since otherwise $\alpha=\beta=0$ and hence, $w=0$.

Now, consider the signal $x_0=\lambda \mathds{1}_K$ and let $w=\frac{\lambda}{3}v=\begin{bmatrix}\frac{\lambda}{3}&-\lambda&\frac{\lambda}{6}&\frac{\lambda}{2}\end{bmatrix}\in \ker(A)$. Then we have $x_0+w=\begin{bmatrix}4\frac{\lambda}{3}&0&\frac{\lambda}{6}&\frac{\lambda}{2}\end{bmatrix}$ ia a feasible solution for $\lambda\le3/4$. Hence,
$\|x+w\|_1=2\lambda=\|x\|_1$ and $Ax=A(x+w)$.
\end{example}

We conclude this section with some observations.

\begin{remark}\rm\

\begin{enumerate}
\item[(1)]  The previous results hold for all alphabets with cardinality 2. Thus, we can replace $1$ by every other value $\alpha\in \R$; statements and proofs remain the same. The statements and results remain also the same, if we replace $0$ by any other value, i.e., we could also consider alphabets of the form $\{-1,1\}$.
\item[(2)] If $A$ fulfills B-NSP with respect to some subset $K\subset[N]$ and $U$ is some orthogonal matrix, then $UA$ also fulfills B-NSP with respect to $K$.
\item[(3)] However, the product $AU$ does not necessarily fulfill B-NSP.
\end{enumerate}

\end{remark}

\subsection{Nonuniform and Uniform Recovery}

Theorem \ref{StojBin} is a nonuniform recovery result for unipolar binary sparse signals via \eqref{Pbin}. This raises the question of what can be shown
concerning uniform recovery. In fact, our next theorem shows a somewhat negative result.

\begin{theorem}\label{thm:Uniform}
 For $A\in\R^{m \times N}$ and $K\subset [N]$, the following statements are equivalent:
 \begin{enumerate}
  \item[(i)] Any unipolar binary vector $x_0$ with $\supp x_0 \subset K$ is the unique solution of \eqref{Pbin} with $b = A x_0$.
  \item[(ii)] Any vector $x_0 \in [0,1]^N$ with $\supp x_0 \subset K$ is the unique solution of \eqref{Pbin} with $b = A x_0$.
  \item[(iii)] The measurement matrix A satisfies NSP$_+$ with respect to $K$, i.e., $\ker(A)\cap N^+\cap H^+_K=\{0\}$, where
  $$H^+_K=\{w\in \R^N: w_i\ge 0 \text{ for } i\in K^C\} \quad \text{and} \quad N^+=\{w\in\R^N:0\ge \sum_{i=1}^N w_i\} . $$
  \end{enumerate}
\end{theorem}

Before proving this result let us discuss some implications. It unfortunately infers that even if we only wish to recover every
unipolar binary signal $x_0$ with $\supp x_0 \subset K$, the measurement matrix $A$ needs to fulfill a much stronger property. This property
is then sufficient to uniquely recover {\it every} positive signal supported on $K$ via \eqref{P+}. Thus, if we wish to show uniform
recovery results, additional assumptions on the signal to be unipolar binary are not beneficial. Therefore, in the following we will concentrate
on nonuniform recovery results.

\begin{proof}
The implication (ii) $\Rightarrow$ (i) is trivial. Moreover, the implication (iii) $\Rightarrow$ (ii) has been shown in \cite{Stojnic+}.

Thus, it remains to prove (i) $\Rightarrow$ (iii). For this, by Theorem \ref{thm:BinNec}, we can conclude that for $A$ to fulfill (i),
it needs satisfy B-NSP with respect to every subset $\tilde{K}\subset K$. To show that this implies (iii), let $v\in \ker(A)\cap H^+_K$,
i.e., $v_i\ge 0$ for all $i\in K^C$. Then there exists a subset $\hat{K}\subset K$ such that $v_i\le 0$ for all $i\in \hat{K}$ and
$v_i\ge 0$ for all $i\in \hat{K}^C$. Since $A$ satisfies B-NSP with respect to $\hat{K}$, this yields
 \begin{align}
  -\sum_{i\in \hat{K}}v_i < \sum_{i\in \hat{K}^C}v_i \Rightarrow  v\notin N^+.
 \end{align}
This being true for every $v\in \ker(A)\cap H^+_K$ implies that $\ker(A)\cap N^+\cap H^+_K=\{0\}$, i.e., $A$ satisfies NSP$_+$.
\end{proof}

\subsection{Phase Transition in Binary Basis Pursuit}\label{sec:PhaseTransBin}

We now aim to show that for $m$ large, a Gaussian matrix fulfills B-NSP with respect to some fixed but unknown support set $K\subset [N]$
with high probability, i.e., that the kernel of a Gaussian matrix does not intersect the convex cone $N_K\cap H_K$ with high probability.
In \cite{ALMT} it has been shown that this probability can be computed in terms of the statistical dimension of $N_K\cap H_K$. However, the
statistical dimension of this cone seems impossible to calculated directly. Therefore we use an approach first suggested in \cite{ALMT} to
obtain an upper bound instead. For this, we rely on the fact that $N_K\cap H_K$ can be recast into the form of a descent cone with
$f:\R^N\rightarrow \R_+$ defined by
\begin{equation} \label{eq:fx}
 f(x) =\begin{cases}
        \|x\|_1 & \text{ if } x\in [0,1]^N,\\
        \infty & \text{ otherwise}.
       \end{cases}
\end{equation}

This can be seen as follows: Using the observation $f(\mathds{1}_K+\tau y)\le f(\mathds{1}_K)<\infty$ and hence, $\mathds{1}_K+\tau y\in [0,1]^N$,
we obtain
\begin{align}
\mathcal{D}(f,\mathds{1}_K):&=\cup_{\tau>0}\{y\in\R^N:f(\mathds{1}_K+\tau y)\le f(\mathds{1}_K)\}\\
&=\cup_{\tau>0}\{y\in\R^N: \|\mathds{1}_K+\tau y\|_1\le \|\mathds{1}_K\|_1 \text{ and } \mathds{1}_K+\tau y\in [0,1]^N\}\\
&=N_K\cap H_K.
\end{align}
The intuition behind this calculation is illustrated in Figure \ref{fig:NSPBin}.

Consequently, we obtain the number of measurements necessary to recover unipolar binary signals with high probability.

\begin{theorem}\label{mainthm1}
Fix a tolerance $\varepsilon >0$. Let $K\subset [N]$, $A\in \R^{m \times N}$ be Gaussian, and $b=A\mathds{1}_K$. Further set $k=|K|$.
Then \eqref{Pbin} will succeed to recover $\mathds{1}_K$ uniquely with probability larger than $1-\varepsilon$ provided that
\begin{align}\label{eqn:m}
 m\ge \Delta_{\text{bin}}(k)+\sqrt{8\log(4/\varepsilon)N},
\end{align}
where
\begin{align}\label{eqn:Deltak}
 \Delta_{\text{bin}}(k):=\inf_{\tau\ge0}\left\{J_k(\tau)\right\}
\end{align}
with
\begin{align}\label{eqn:Jktau}
J_k(\tau) := k\int_{-\infty}^{\tau}(u-\tau)^2\phi(u)du + (N-k)\int_{\tau}^{\infty}(u-\tau)^2\phi(u)du
\end{align}
and $\phi(u)\!=\!(2\pi)^{-1/2}e^{-u^2/2}$ being the probability density of the Gaussian distribution.
\end{theorem}

Before proving this theorem let us first discuss the behavior of the function $\Delta_{\text{bin}}$, which is plotted as the 
blue curve in Figure \ref{fig:StatBin}. 

\begin{proposition}\label{lemma2:StatBin}
Let $\mathbf{g}$ be a normally distributed random vector, and let $\Delta_{\text{bin}}(k)$, $k\ge0$, be defined as in Theorem \ref{mainthm1}. Then the following hold:
\begin{enumerate}
\item[(i)] $\Delta_{\text{bin}}(k)< N/2$ for all $0\le k< N/2$.
\item[(ii)] $\Delta_{\text{bin}}(k)=N/2$ for all $N/2\le k\le N$.
\end{enumerate}
\end{proposition}

\begin{proof}
Let $J_k(\tau)$ be defined as in Theorem \ref{mainthm1} and note that 
\[
J_k(0)=k\int_{-\infty}^0u^2\phi(u)du+(N-k)\int^{\infty}_0u^2\phi(u)du=k/2+(N-k)/2=N/2\quad \mbox{for all } k=0,\dots, N,
\]
which in turn implies $\inf_{\tau\ge0}J_k(\tau)/N\in [0,1/2]$ for $k=0,\dots,N$.

We next rewrite $J_k$ as
\begin{eqnarray}\nonumber
 J_k(\tau)&=&k\int_{-\infty}^{\infty}(u-\tau)^2\phi(u)du+(N-2k)\int^{\infty}_{\tau}(u-\tau)^2\phi(u)du\\ \nonumber
 &=& k(1+\sqrt{2\pi}\tau^2)+(N-2k)\int^{\infty}_{\tau}(u-\tau)^2\phi(u)du\\ \label{eq:Jk}
 &=:& f_k(\tau)+g_k(\tau).
\end{eqnarray}
The function $f_k$ and $g_k$ satisfy
\begin{align}
 f'_k(\tau)=\sqrt{8\pi}k\tau >0 \quad \text{for} \quad \tau>0, \label{eq:f'k}
\end{align}
and
\begin{align}
 g'_k(\tau)=-2(N-2k)\int_{\tau}^{\infty}(u-\tau)\phi(u)du\begin{cases}
                                                          > 0 & \text{for} \quad k>N/2\\
                                                          < 0 & \text{for} \quad k<N/2\label{eq:g'k}\\
                                                          = 0 & \text{for} \quad k=N/2.
                                                          \end{cases}
\end{align}
Thus, by \eqref{eq:f'k} and \eqref{eq:g'k}, for $k\ge N/2$ the function $J_k$ is monotonically increasing on $[0,\infty)$ and the infimum is attained in 
$\tau=0$. Application of the definition of $\Delta_{\text{bin}}(k)$ (see \eqref{eqn:Deltak}), proves (ii). 

To show (i), notice that, if $k<N/2$, $f_k$ increases and $g_k$ decreases monotonically. Moreover,
\begin{align}
 f_k'(0)=0< 2(N-2k) =-g_k'(0).
\end{align}
Since $f_k'$ and $g_k'$ are continuous, the intermediate value theorem implies the existence of some $\tau^*> 0$ satisfying
\begin{align}
 f_k'(\tau)< -g_k'(\tau) \quad \text{for all} \quad \tau \in (0,\tau^*).
\end{align}
Thus, the function $J_k$ decreases on $\tau\in (0,\tau^*)$. We can hence conclude that $J_k(\tau)<J(0)=N/2$,
which is (i).
\end{proof}

\begin{figure}[ht]
\begin{center}
\includegraphics[scale=0.5]{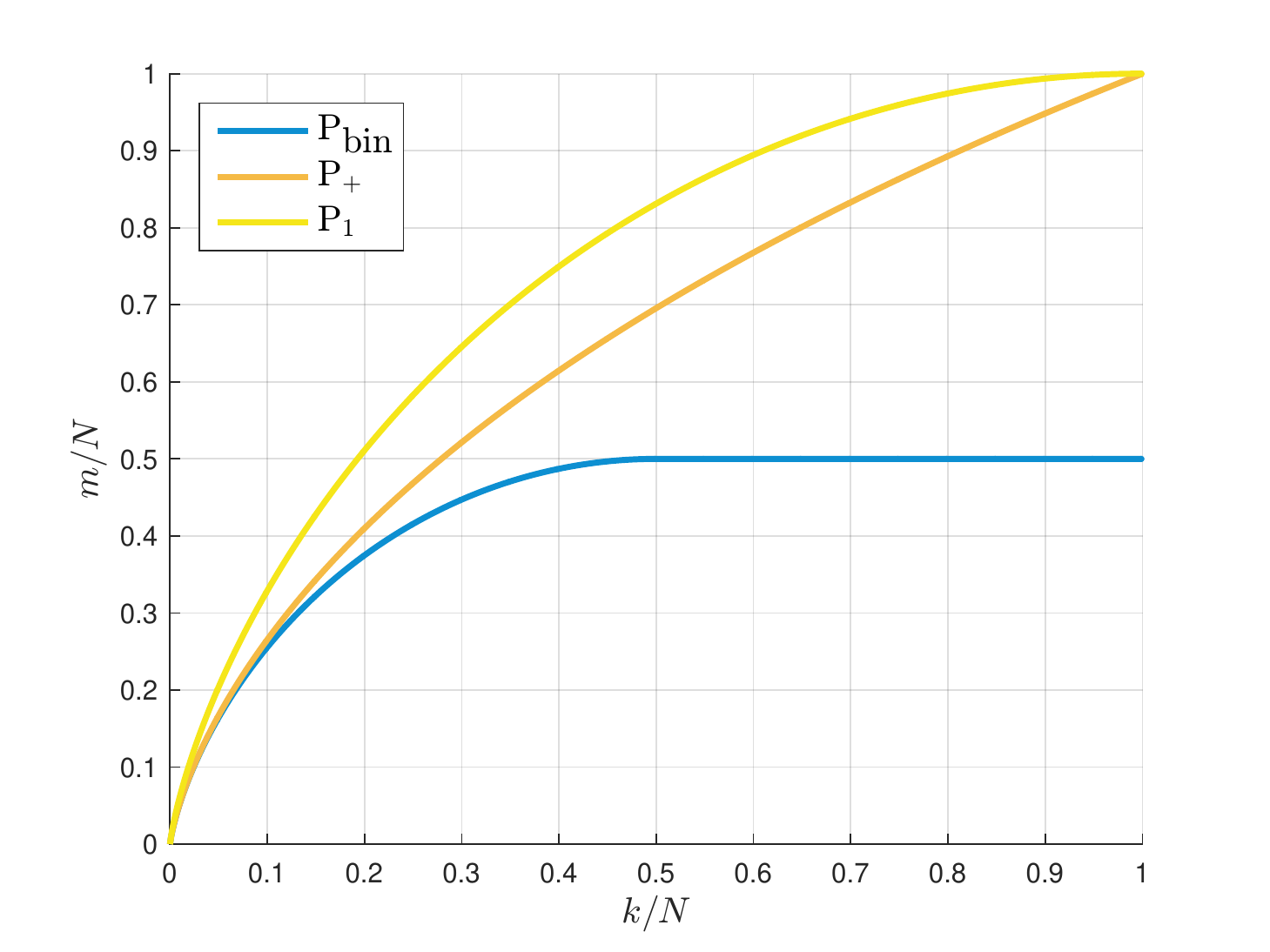}\includegraphics[scale=0.5]{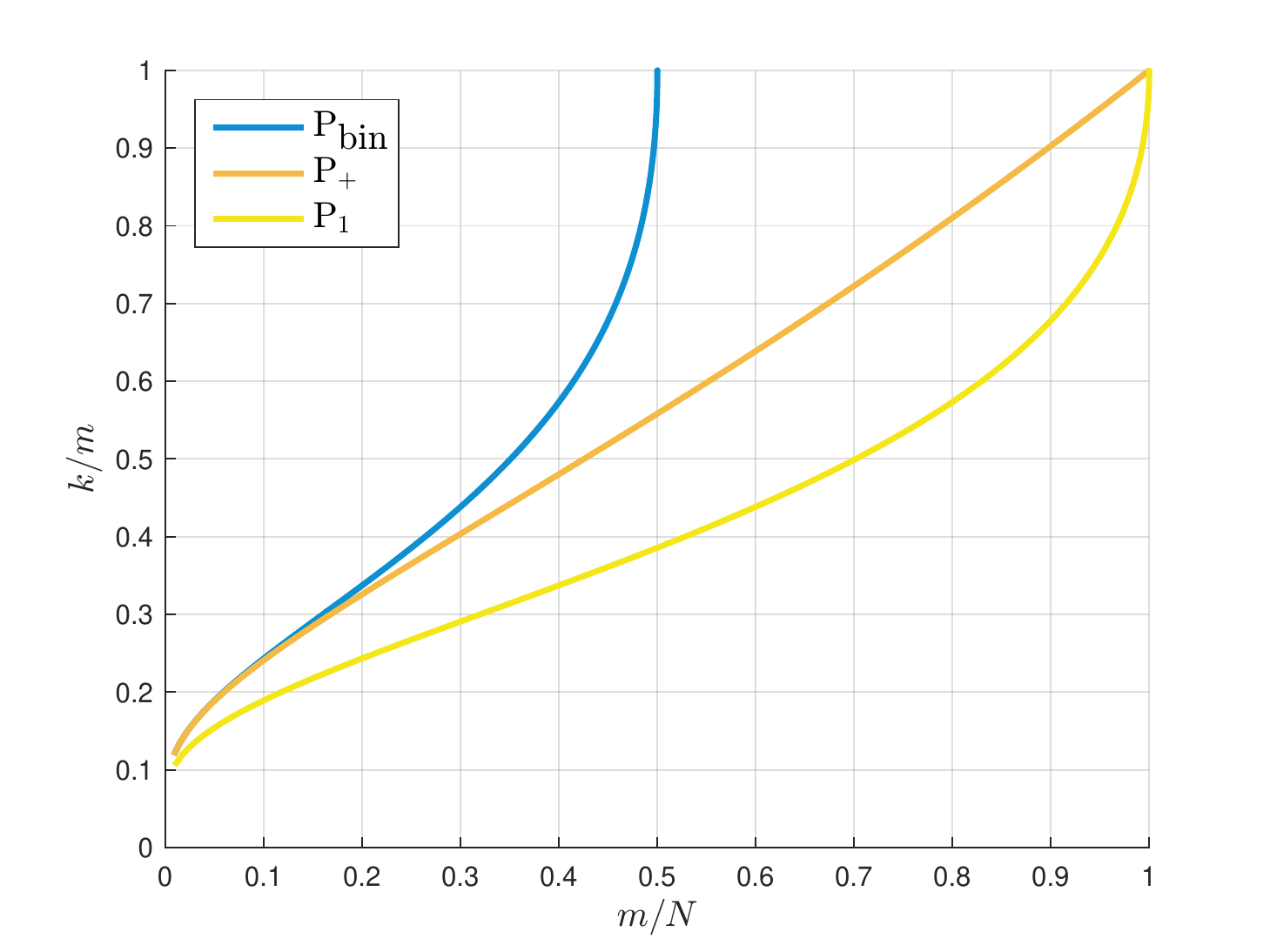}
\put(-320,-15){(a)}
\put(-108,-15){(b)}
\caption{Phase transition of the convex programs \protect\eqref{P1}, \protect\eqref{P+} and \protect\eqref{Pbin}. For the convenience of the reader, the following two illustrations are provided:
Successful recovery related to the area above the curves in (a), and below the curves in (b).}
\label{fig:StatBin}
\end{center}
\end{figure}

As a first step to prove Theorem \ref{mainthm1}, we describe the subdifferential $\partial f(\mathds{1}_K)$.

\begin{lemma}\label{lem:subdiff}
For $f$ defined by \eqref{eq:fx},
\begin{align}
\partial f(\mathds{1}_K) =\{s\in \R^N: s_i \ge 1 \text{ for } i\in K \text{ and } s_i\le 1 \text{ for } i\in K^C \}.
\end{align}
\end{lemma}

\begin{proof}
First notice that, for given $y\notin [0,1]^N$, the inequality $f(y)\ge f(\mathds{1}_K) + \langle s, y-\mathds{1}_K\rangle$ is vacuously true
for all $s\in\R^N$. This implies that the subdifferential of $f$ simplifies to
\begin{align}\label{subBin}
\partial f(\mathds{1}_K)=\{s\in \R^N: \|y\|_1\ge \|\mathds{1}_K\|_1 + \langle s, y-\mathds{1}_K \rangle \text{ for all } y\in [0,1]^N\}.
\end{align}
Thus
\begin{align}
\partial f(\mathds{1}_K) &=\{s\in \R^N: k\le \sum_{i=1}^N (1-s_i) y_i + \sum_{i\in K}s_i,  y\in [0,1]^N \}.
\end{align}
Next, without loss of generality, we assume that $K=\{1,\dots, k\}$, and, for the sake of brevity, define
\[
S := \{s\in \R^N: s_i \ge 1 \text{ for } i\in K \text{ and } s_i\le 1 \text{ for } i\in K^C \}.
\]
To prove $S \subseteq \partial f(\mathds{1}_K)$, let $s\in S$ and observe
\begin{align}
\sum_{i=1}^N (1-s_i) y_i + \sum_{i=1}^ks_i
&=\sum_{i=1}^k(1-s_i)y_i + \sum_{i=k+1}^N(1-s_i)y_i + \sum_{i=1}^ks_i \\
&\ge \sum_{i=1}^k(1-s_i)y_i + \sum_{i=1}^ks_i \\&\ge \sum_{i=1}^k(1-s_i) + \sum_{i=1}^ks_i=k,
\end{align}
where we used in the first inequality that $(1-s_i)y_i\ge 0$ for $i=k+1,\dots, N$, and in the last inequality that $(1-s_i)y_i\ge (1-s_i)$ for $i=1,\dots, k$.
Thus, indeed, $S \subseteq \partial f(\mathds{1}_K)$.

To prove  the other inclusion, i.e., $\partial f(\mathds{1}_K) \subseteq S$, we consider the following points $y^j\in [0,1]^N$, $j=1,\dots,N$:
\begin{align}
y^j= \left\{
   \begin{array}{ll} \sum_{i=1}^k e_i - e_j:  \quad &  j = 1, \cdots , k ,\\[3pt]
                     \sum_{i=1}^k e_i + e_j: \quad & j = k+1 , \cdots , N.
   \end{array}
   \right.
\end{align}
Since $y^j\in [0,1]^N$ for all $j$, every $s \in \partial f(\mathds{1}_K)$ must satisfy
\begin{align}\label{sub1}
k\le \sum_{i=1}^N (1-s_i) y^j_i + \sum_{i=1}^ks_i, \quad j=1,\dots, N.
\end{align}
For $j=1, \dots, k$, this inequality is equivalent to
\begin{align}
k\le k-1-\sum_{i=1, i\neq j}^ks_i+\sum_{i=1}^k s_i,
\end{align}
which yields equivalently $s_j \geq 1$. And for $j=k+1,\dots, N$, we have
\begin{align}
k\le k+1- \sum_{i=1}^ks_i -s_j +\sum_{i=1}^k s_i,
\end{align}
which is equivalent to $s_j \le 1$. This implies $\partial f(\mathds{1}_K) \subseteq S$. The claim is proven.
\end{proof}

The statistical dimension of $D(f,x_0)$ can now be bounded by the infimum of the function
\[
J_{k}(\tau):= \mathbb{E}[\dist{\mathbf{g}}{\tau\partial f(\mathds{1}_K)}^2] \quad \mbox{for } \tau \ge 0,
\] 
where $\mathbf{g}\sim\text{\footnotesize{NORMAL}}(0,\mathbf{I})$. The computation of $\dist{\mathbf{g}}{\tau\partial f(\mathds{1}_K)}^2$ 
is quite straightforward, wherefore we omit it, and leads to
\begin{equation} \label{eq:g}
\dist{\mathbf{g}}{\tau\partial f(\mathds{1}_K)}^2
=\|u\|_{2}^2, \quad \text{ with }
u_i = \left\{
   \begin{array}{ll} \max\{\tau-g_i,0\}: \quad &  i \in K ,\\
                     \max\{g_i-\tau,0\}: \quad & i \in K^c .
   \end{array}
   \right.
\end{equation}

This finally allows us to prove Theorem \ref{mainthm1}.
\begin{proof}[Proof of Theorem \ref{mainthm1}]
Equation \eqref{eqn:Jktau} follows immediately by taking the expected value of \eqref{eq:g}. As just discussed,
the statistical dimension of the cone can now be bounded by its infimum, which yields Equation \eqref{eqn:Deltak}. 
Finally, using the results from \cite{ALMT}, the sufficient number of measurements to uniquely recover a unipolar binary sparse signals (Equation \eqref{eqn:m}), follows.
\end{proof}

For the recovery of unipolar binary signals, we observed a phase transition which increases monotonically in the sparsity $k$ 
and becomes constant in the range $k\in [N/2,N]$. However, we have not incorporated the symmetry of the signal in the recovery 
algorithm. The next remark provides an idea how to incorporate this.

\begin{remark}\rm
Note that if the sparsity of a unipolar binary
signal $x_0$ exceeds $N/2$ and if we consider the ``mirrored''  binary basis pursuit
\begin{align}
 \min\|x\|_1 \text{ subject to } Ax=A\mathds{1}_{[N]}-b \text{ and } x\in [0,1]^N \tag{P$_{\text{Mbin}}$} \label{PbinS},
\end{align}
the vector $\mathds{1}_{[N]}-x_0$ is a unipolar binary $N-k$-sparse vector which is feasible for $Ax=A\mathds{1}_N-Ax_0$. Thus, if the sparsity $k$ of $x_0$ is
larger than $1/2$ we can recover $x_0$ with high probability if the number of measurements $m$ is larger than $\Delta_{\text{bin}}(N-k)$.

This shows that running \eqref{Pbin} in the case $k\le N/2$ and \eqref{PbinS} in the case $k>N/2$, we
obtain a reverse parabola for the phase transition (cf. Figure \ref{fig:GespiegeltBin}~(a)). If we do not know the sparsity level in advance, we
propose to run both algorithms \eqref{Pbin} and \eqref{PbinS} to compute the solutions $x^1,x^2$, and choose the $x^i$ which is closest to be unipolar binary,
i.e., the solution of $\argmax_{i=1,2}\|x^i-\rnd{(x^i)}\|_1$. Numerically this approach seems to be very promising (cf. Figure \ref{fig:GespiegeltBin}~(b)).

\begin{figure}[ht]
\begin{center}
\includegraphics[scale=0.4]{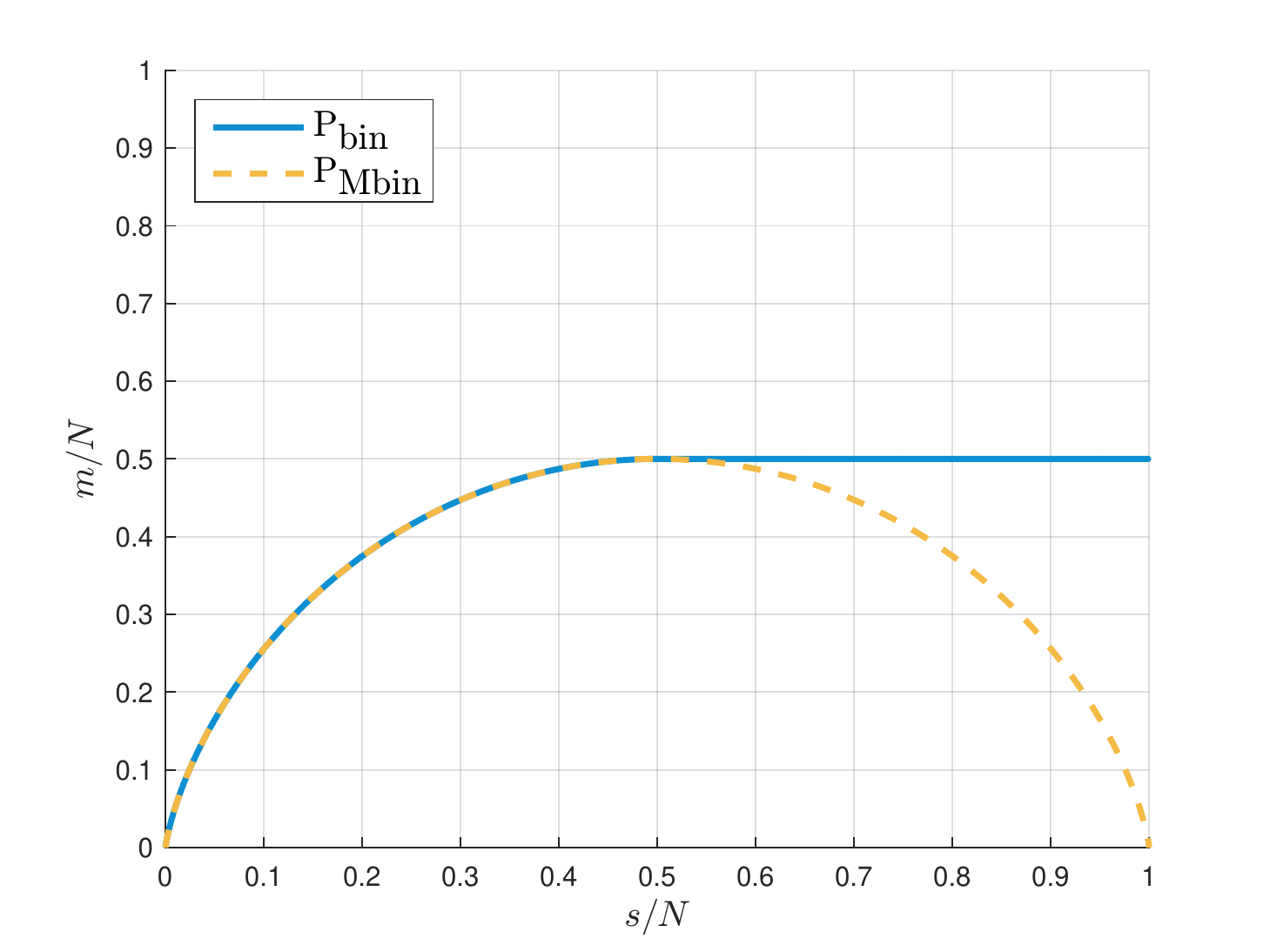}\hspace{1.5cm}
\includegraphics[scale=0.3]{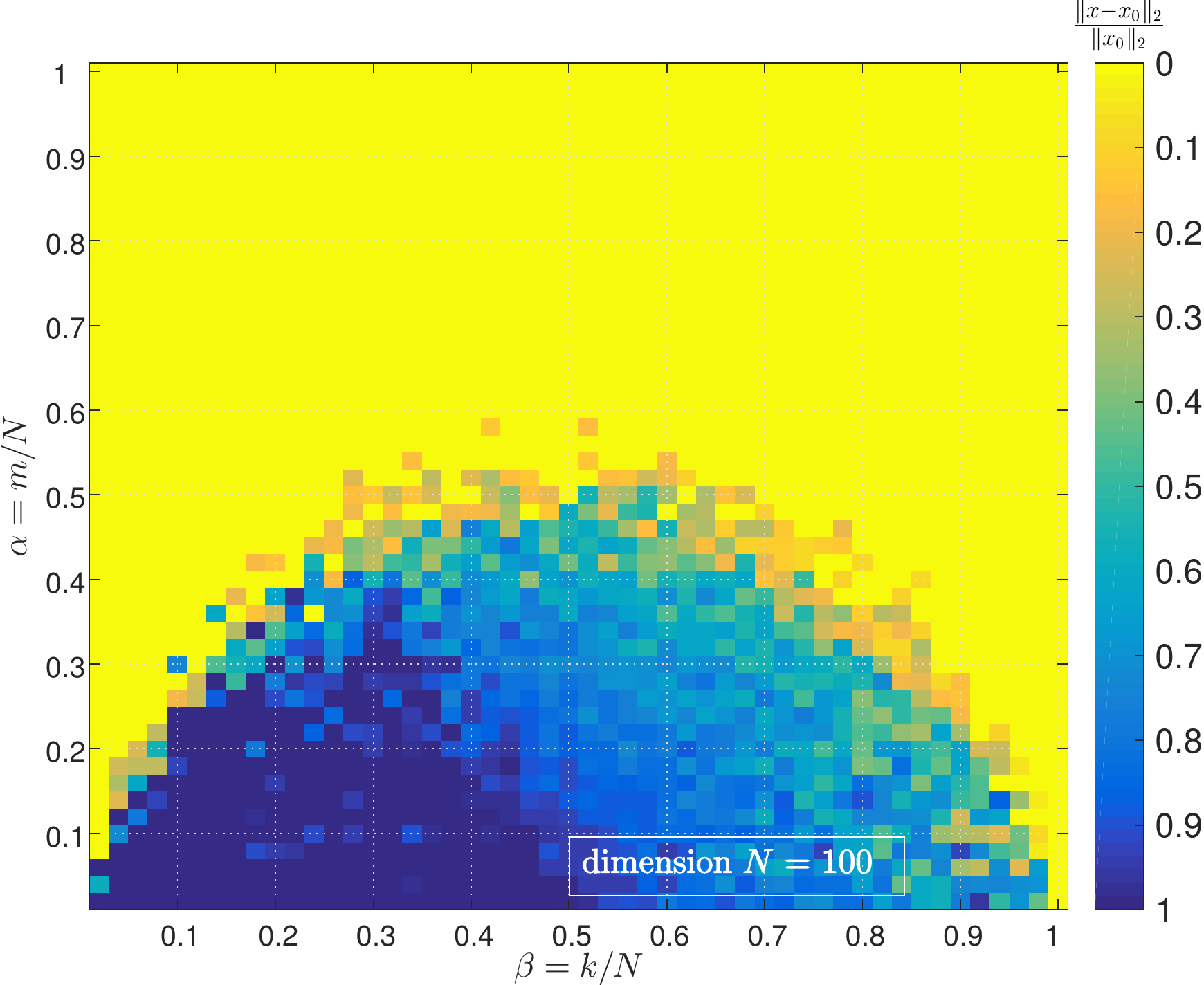}
\put(-288,-15){(a)}
\put(-84,-15){(b)}
\caption{Reconstruction error of the adapted version of \protect\eqref{Pbin}, which runs both \protect\eqref{Pbin} and \protect\eqref{PbinS} and chooses the solution which is closest to 
be unipolar binary. (a): Numerically determined phase transition of \protect\eqref{PbinS}; (b): Numerical experiments.} \label{fig:GespiegeltBin}
\end{center}
\end{figure}
\end{remark}

\subsection{Robust Binary NSP}

In realistic situations, signals cannot be measured with infinite precision. This raises the question of stability of a recovery algorithm with
respect to noisy measurements, i.e., in case the measurement vector $b$ is only an approximation of the vector $Ax$ with $\|Ax-b\|\le \eta$ for
some $\eta\ge 0$ and some norm $\|\cdot\|$. The canonical way to tackle this problem numerically, is to replace the constraint $Ax=b$ in \eqref{Pbin}
by $\|Ax-b\|\le \eta$ (cf. \cite{FR}). In our situation, this then yields the following algorithm, which we refer to as \emph{binary basis
pursuit with inequality constraints}:
\begin{align}
\label{BBPDenoising}
\min\|z\|_{1} \quad \text{subject to} \quad \|Az - y \|_{2} \le \eta \quad \text{and} \quad z\in [0,1]^N \tag{P$_{\text{bin}}^{\eta}$}.
\end{align}
Theorem \ref{thm:BBPDS} will show that indeed \eqref{BBPDenoising} is robust provided that the sensing matrix $A\in \R^{m \times N}$ fulfills the
following adapted B-NSP condition:

\begin{definition}
A matrix $A \in \R^{m \times N}$ is said to satisfy the \emph{robust binary null space property} with constants $0 < \rho < 1$ and $\tau > 0$
relative to a set $K \subset [N]$, if \noeqref{RBNSP}
\begin{align}
- \sum_{i \in K} v_i\le  \rho \sum_{i \in K^C} v_i + \tau \| Av\|_2 \label{RBNSP}\tag{RB-NSP}
\quad
\textnormal{for any } v \in H_K \cap \ker(A).
\end{align}
\end{definition}

In preparation of our robustness theorem (Theorem \ref{thm:BBPDS}), the following lemma provides some equivalent conditions for RB-NSP to hold.

\begin{lemma}
\label{thm_robust_NSP_z_between_0_alpha}
For $A \in \R^{m \times N}$, $K \subset [N]$, $0 < \rho < 1$, and $\tau > 0$, the following conditions are equivalent.
\begin{itemize}
\item[(i)]
$A$ satisfies RB-NSP with $\rho$ and $\tau$ relative to $K$.
\item[(ii)]
For any $x, z \in \R^{N}$ with $z-x\in H_K$,
\begin{align}
\sum_{i\in K}x_i - \sum_{i\in K} z_i\le \rho \|(z-x)_{K^C}\|_1+ \tau\|A(z-x)\|_2.
\end{align}
\item[(iii)]
For any $z \in \R^{N}$ with $0 \leq z_i \le 1$,
\begin{align} |K| - \|z_K\|_1\le \rho \|z_{K^C}\|_1+ \tau\|A(z - \mathds{1}_K) \|_2 .
\end{align}
\end{itemize}
\end{lemma}

\begin{proof}
The equivalence of (i) and (ii) is straightforward by taking $v = z - x$. Also, (ii) $\Rightarrow$ (iii) is immediate by choosing $x = \mathds{1}_K$.

Thus it remains to show that (iii) implies (i). The condition \eqref{RBNSP} holds trivially for $v = 0$; hence we assume that $v \in H_K \setminus \{ 0 \}$.
For sufficiently small $t > 0$, we have $z_t= \mathds{1}_K + tv \in [0,1]^N$. This shows that (iii) implies
\begin{align}
 |K|-\| (\mathds{1}_K + tv)_K  \|_1\le \rho\sum_{i \in K^C} t v_i+ \tau \| A (tv) \|_2 ,
\end{align}
which is equivalent to
\begin{align}
-t\sum_{i\in K}v_i \le t\rho \sum_{i\in K^C} v_i + t\tau \|Av\|_2.
\end{align}
But this in turn is equivalent to \eqref{RBNSP}.
\end{proof}

Now we are ready to prove that the RB-NSP is sufficient to ensure robust recovery of \eqref{BBPDenoising}. We remark that the following result quantifies
the $\ell_1$-error of the approximation as well.

\begin{theorem}
\label{thm:BBPDS}
Let $A \in \mathbb{R}^{m \times N}$ satisfy the RB-NSP with constants $0 < \rho < 1$ and $\tau > 0$ relative to a set $K \subset [N]$. Further,
assume the measurements satisfy $b = A (\mathds{1}_K) + e$, where the noise $e \in \mathbb{R}^N$ satisfies $\|e\|_{2} \le \eta$. Then a
solution $\hat{z}$ of \eqref{BBPDenoising} approximates $\mathds{1}_K$ with $\ell_1$-error
\begin{align}
\| \hat{z} -\mathds{1}_K \|_1 \leq \frac{4 \tau }{1 - \rho} \eta.
\end{align}
\end{theorem}

\begin{proof}
Let $z \in \R^{N}$ satisfy $0 \le z_i \le 1$. Then Lemma \ref{thm_robust_NSP_z_between_0_alpha} implies
\begin{align}
\label{robustness_thm_binary_pf_ineq0}
|K| - \|z_K \|_1 \le \rho \|z_K\|_1 + \tau \| A (z -\mathds{1}_K) \|_2,
\end{align}
which is
\begin{align}
\label{robustness_thm_binary_pf_ineq2}
\|(z - \mathds{1}_K)_{K} \|_1 \le \rho \| z_{K^C} \|_1 + \tau \| A (z - \mathbf{1}_K) \|_2.
\end{align}
Using \eqref{robustness_thm_binary_pf_ineq0}, we have also
\begin{align}
\|\mathds{1}_K \|_1 - \| z \|_1 &= |K| - \| z_K \|_1 - \| z_{K^C} \|_1 \\
&\le \Big( \rho \|z_{K^C} \|_1  + \tau \| A (z - \mathds{1}_K) \|_2 \Big) - \| z_{K^C} \|_1 \\
&\le (\rho - 1) \| z_{K^C} \|_1  + \tau \| A (z -  \mathds{1}_K) \|_2,
\end{align}
which is equivalent to
\begin{align}
\label{robustness_thm_binary_pf_ineq1}
\| z_{K^C} \|_1\le \frac{1}{ 1- \rho } \left( \|z \|_1 - \| \mathds{1}_K \|_1  + \tau \| Az - \mathds{1}_K)\|_2 \right) .
\end{align}
Combining \eqref{robustness_thm_binary_pf_ineq2} and \eqref{robustness_thm_binary_pf_ineq1} we obtain
\begin{align}
\|z - \mathds{1}_K \|_1 &= \| (z - \mathds{1}_K)_{K} \|_1 + \| z_{K^C} \|_1 \\
&\le \Big( \rho \| \mathds{z}_{K^C} \|_1 + \tau \| A (z - \mathds{1}_K) \|_2 \Big) +  \| z_{K^C} \|_1\\
&= (1 + \rho) \|z_{K^C} \|_1 + \tau \| A(z -  \mathds{1}_K) \|_2 \\
&\le \frac{1+ \rho}{ 1- \rho } \Big( \| z \|_1 - \| \mathds{1}_K \|_1  + \tau \| A (z - \mathds{1}_K) \|_2 \Big) + \tau \| A (z - \mathds{1}_K) \|_2\\
&\le \frac{1+ \rho}{ 1- \rho } \Big( \|z \|_1 - \| \mathds{1}_K \|_1 \Big) + \frac{2 \tau }{ 1- \rho } \| A (z - \alpha \mathds{1}_K) \|_2 .
\end{align}
Now if $z = \hat{z}$ is a minimizer of \eqref{BBPDenoising}, then $\| \hat{z} \|_1 \le \| \mathds{1}_K \|_1$ and
\begin{align}
\| \hat{z} -  \mathds{1}_K \|_1 \le \frac{2 \tau }{ 1- \rho } \| A (\hat{z} -  \mathds{1}_K) \|_2 .
\end{align}
Finally, the desired result follows from $\| A (\hat{z} - \mathds{1}_K) \|_2 \leq \| A \hat{z} - y \|_2 +  \| y - A ( \mathds{1}_K) \|_2 \le 2 \eta$.
\end{proof}

\bigskip

\begin{remark}\rm
In the finite-valued setting, `stability' (with respect to the sparsity defect of a finite-valued signal) is not crucial as its entries
are assumed to be from a finite alphabet.
\end{remark}

\subsection{Phase Transition under Noisy Measurements}

As seen in the previous subsection, \eqref{BBPDenoising} is robust provided that the sensing matrix $A$ satisfies RB-NSP. However, this property is 
hard to verify in general. In the case of Gaussian matrices we can though again rely on the statistical dimension and utilize the results 
from Section \ref{sec:PhaseTransBin}. We will show that Gaussian matrices are in terms of the NSP well-suited for robust recovery of unipolar 
binary signals provided that $m$ is sufficiently large. 

Since the outcome of the adapted basis pursuit \eqref{BBPDenoising} is not necessarily an integer, in addition we utilize the finite nature of the 
signal and post-project it on the integers in the spirit of \cite{FK}. Hence, we consider the following algorithm:
\begin{align}
 \tilde{x}=\rnd(\hat{x}) \quad \text{with}\quad \hat{x}=\argmin\|x\|_1 \quad \text{subject to} \quad \|Ax-b\|_2 \le \eta \quad \text{and} \quad x\in [0,1]^N. \label{PbinDR}\tag{$P_{\text{bin}}^{\eta,\text{r}}$}
\end{align}

In the next theorem we will prove that a certain number of measurements is sufficient for \eqref{BBPDenoising} to succeed. This result
then directly yields a number of measurements for \eqref{PbinDR} to succeed as a special case. Indeed, if the solution $\hat{x}$ of \eqref{BBPDenoising} satisfies $\|\mathds{1}_K-\hat{x}\|_{\infty} < 1/2$, then $|(\mathds{1}_K-\hat{x})_i|<1/2$ for all $i$ and, hence, rounding provides the exact solution $\mathds{1}_K$.

\begin{theorem}\label{thm:PhaseTransNoise}
Let $x=\mathds{1}_K$, $k=|K|$, and $A\in \R^{m \times N}$ be a Gaussian measurement matrix. Further, assume the the measurements satisfy $b=Ax+e$,
where $\|e\|_{2} \le \eta$, and let $0 < \varepsilon < 1$, $\tau >0$, and $m$ fulfill
 \begin{align}\label{mainthm2.1}
  \frac{m^2}{m+1}\ge \left(\sqrt{ln(\varepsilon^{-1})}+\sqrt{\Delta_{\text{bin}}(k)}+\tau\right)^2.
 \end{align}
Then, with probability at least $1-\varepsilon$, every minimizer $\hat{x}$ of \eqref{BBPDenoising} satisfies
\begin{align}
 \|x-\hat{x}\|_2\le \frac{2\eta}{\tau}.
\end{align}
In particular, if
 \begin{align}
  \frac{m^2}{m+1}> \left(\sqrt{ln(\varepsilon^{-1})}+\sqrt{\Delta_{\text{bin}}(k)}+4\eta\right)^2,
 \end{align}
then $x$ is the unique solution of \eqref{PbinDR}.
 \end{theorem}

\begin{proof}
The last part follows from \eqref{mainthm2.1} by choosing $\tau=4\eta$ to obtain $\|x-\hat{x}\|_2<1/2$. 

By \cite[Thm. 4.36]{FR}, the first statement is true, if $\inf_{v\in\mathcal{D}_K\cap \mathbb{S}^{N-1}}\|Av\|_2 \ge \tau$, where  $\mathcal{D}_K:=\mathcal{D}(f,\mathds{1}_K)$. 
To show this inequality we use \cite[Thm. 9.21]{FR}, which states that
\begin{align}
 \mathbb{P}\left(\inf_{x\in T}\|Ax\|_2\le E_m-\ell(T)-t\right)\le e^{-t^2/2},
\end{align}
where $E_m$ can be bounded from below by $m/\sqrt{m+1}$ and $\ell(T)=\mathbb{E}\sup_{v\in T}\langle \mathbf{g},v\rangle$ for some subset $T\subset \R^N$ and $\mathbf{g}$ a 
standard Gaussian random vector. Thus, if
\begin{align}
  \frac{m}{\sqrt{m+1}}\ge \tau + \ell(\mathcal{D}_K\cap \mathbb{S}^{N-1})+\sqrt{2\ln(\varepsilon^{-1})},
 \end{align}
 we can conclude that $\inf_{v\in\mathcal{D}_K\cap \mathbb{S}^{N-1}}\|Av\|_2 \ge \tau$.
 The result now follows from  the fact that, for $T\subset \R$, we have $\ell(T)\le \sqrt{\partial(T)}\le
\sqrt{\ell(T)^2+1}$, which was shown in \cite[Remark 3.5]{OT} and from
\begin{align}\label{eq:bin}
 \Delta_{\text{bin}}(k)=\partial(\mathcal{D}_K)=\partial(\sigma(\mathcal{D}_K))=\partial(\mathcal{D}_K\cap \mathbb{S}^{N-1}),
\end{align}
where $\sigma$ denotes the spherical retraction, which is defined as $\sigma(t)=t/\|t\|_2$ for $t\neq 0$ and $\sigma(0)=0$. Notice that the third equation in \eqref{eq:bin} uses the fact that $\mathcal{D}_K$ is a convex cone.
\end{proof}

\section{Bipolar Ternary Signals}

In the last section we analyzed recovery of unipolar binary sparse signals using \eqref{Pbin}. The question remains whether the results can be extended to general
finite-valued signals. In preparation for the general results stated in Section \ref{sec:larger},
this section is dedicated to the study of signals with entries in $\{0, \pm1\}$, which we refer to as \emph{bipolar ternary sparse signals}. Interestingly,
this small extension by one more possible value already leads to several necessary adaptions both theoretically and numerically; and it will turn out that
only slightly weaker results hold. 

In the spirit of \eqref{Pbin}, for bipolar ternary signals we impose the additional constraint $x\in [-1,1]^N$ on basis pursuit, i.e., we use the
following recovery procedure:
\begin{align}\label{PM}
\min \|x\|_1 \text{ subject to } Ax=b \text{ and } x\in [-1,1]^N. \tag{$P_{\pm\text{ter}}$}
\end{align}

\subsection{Bipolar Ternary NSP}

We start by adapting the NSP to the situation of bipolar ternary sparse signals. The bipolar ternary NSP (BT-NSP) stated below can indeed be
shown to provide a necessary and sufficient condition on the sensing matrix for \eqref{PM} to succeed.

\begin{definition}
 A matrix $A\in \R^{m\times N}$ is said to satisfy the \emph{bipolar ternary NSP} with respect to some disjoint subsets $K_1, K_{-1}\subset [N]$, if
 \noeqref{STNSP}%
 \begin{align}
  \ker(A)\cap N_K\cap H_{K_1,K_{-1}} =\{0\} \label{STNSP}\tag{BT-NSP},
 \end{align}
where $H_{K_1,K_{-1}}=\{w\in \R^N: w_i\le 0 \text{ for all } i\in K_1, \text{ and } w_i\ge 0 \text{ for all } i\in K_{-1}\}$ and $K=K_{-1}\cup K_1$.
\end{definition}

An illustration of BT-NSP is provided in Figure \ref{fig:NSPPMBin}. For this, we choose a non-sparse signal to illustrate the advantage of the additional,
bipolar ternary structure. Indeed, if we, for example, wish to recover the non-sparse signal $x_0=[-1\;1]^T\in \R^2$ from one measurement via the classical 
approach \eqref{P1}, we will fail, since the $\ell_1$-ball $\{x\in \R^N:\|x\|_1\le \|x_0\|_1\}$ always contains another feasible solution. 
The $\ell_1$-ball intersected with the constraint set, however, does not need to contain another feasible solution (cf. Figure \ref{fig:NSPPMBin}). 
This intuitively implies that restriction to box constraints yields a high probability of success.

\begin{figure}[ht]
\begin{center}
 \includegraphics[scale=0.5]{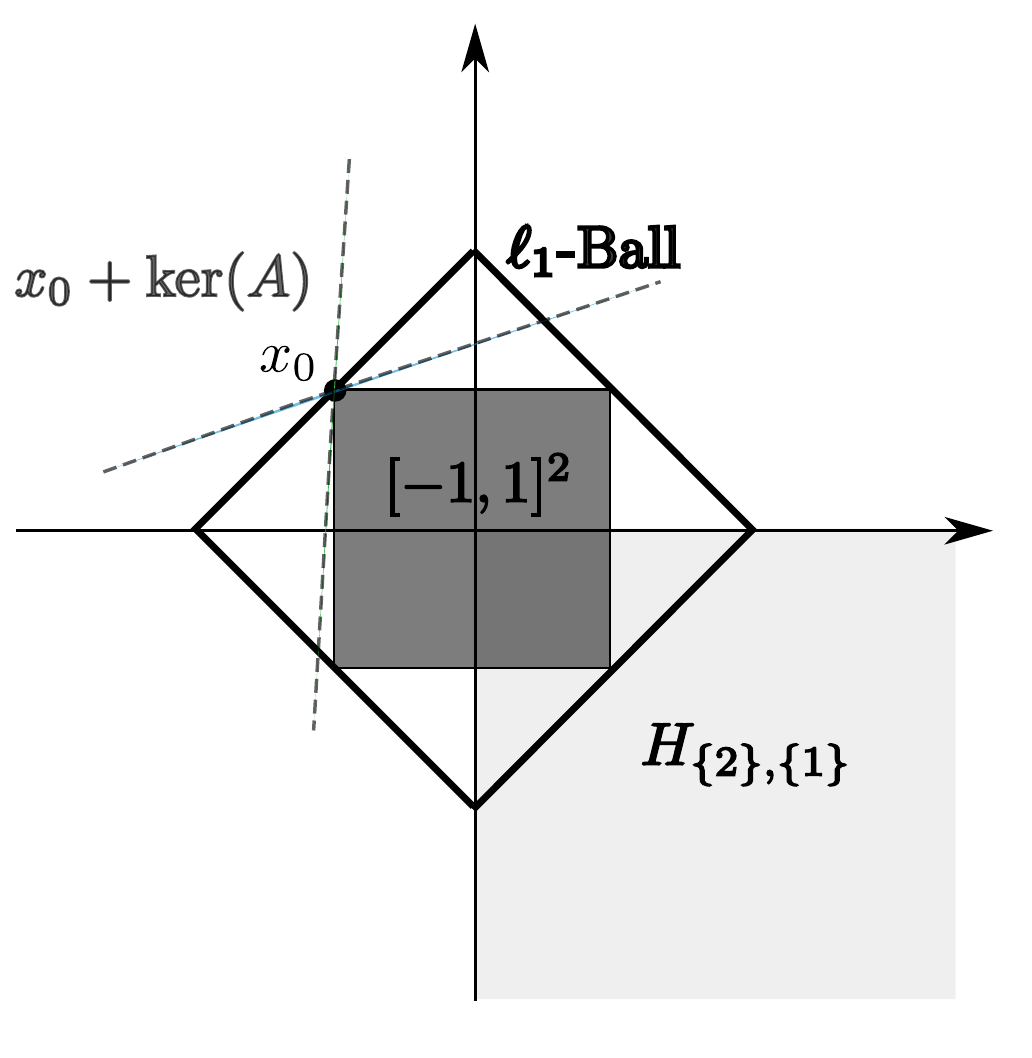}
 \caption{Geometrical interpretation of the BT-NSP. The NSP condition is equivalent to the condition that the kernel of $A$ does not intersect the descent cone $N_{\{2\}}\cap H_{\{2\},\{1\}}$ corresponding to \protect\eqref{Pbin}.}\label{fig:NSPPMBin}
 \end{center}
\end{figure}

The next theorem states that this condition is, indeed, equivalent to the success of \eqref{PM}.

\begin{theorem}\label{thm:BTNSP}
  Let $K_1,K_{-1}\subset [N]$ be two disjoint sets. Let further $x_{\pm1}=\mathds{1}_{K_1}-\mathds{1}_{K_{-1}}$ be a bipolar ternary $k$-sparse signal, 
  where $k=|K_1\cup K_{-1}|$. Then the following conditions are equivalent:
  \begin{enumerate}
  \item[(i)] $x_{\pm1}$ is the unique solution of \eqref{PM} with $b=A x_{\pm1}$.
  \item[(ii)] $A\in \R^{m\times N}$ satisfies BT-NSP with respect to sets $K_1$ and $K_{-1}$.
  \end{enumerate}
\end{theorem}

\begin{proof}
We begin by proving (ii) $\Rightarrow$ (i). For this, assume that $A$ satisfies BT-NSP, and let $x$ be the solution of \eqref{PM} with $Ax_{\pm1}=b$.
Then $x-x_{\pm1}\in \ker{A}$. Using the fact that $x$ is the solution of \eqref{PM} and therefore all its entries lie between $-1$ and $1$, we
can conclude that
\begin{enumerate}
 \item[(a)] $(x-x_{\pm1})_i=x_i-1\le 0$ for $i\in K_1$, and 
 \item[(b)] $(x-x_{\pm1})_i=x_i+1\ge 0$ for $i\in K_{-1}$. 
\end{enumerate}
This shows that $x-x_{\pm1} \in H_{K_1,K_{-1}}$, and therefore the assumption (a)  implies
\begin{align}
 \sum_{i\in K}|(x-x_{\pm1})_i|\le\sum_{i\in K^C}|(x-x_{\pm1})_i|.
\end{align}
Due to properties (a) and (b), this is equivalent to
\begin{align}
  -&\sum_{i\in K_1}(x-x_{\pm1})_i+\sum_{i\in K_{-1}}(x-x_{\pm1})_i< \sum_{i\in K^C}|x_i|,
  \end{align}
which in turn holds if and only if
\begin{align}
  \quad\quad&\sum_{i\in K_1}(x_{\pm1})_i-\sum_{i\in K_{-1}}(x_{\pm1})_i< \sum_{i\in K^C}|x_i|+\sum_{i\in K_1}x_i-\sum_{i\in K_{-1}}x_i\le \|x\|_1.
  \end{align}
 This is equivalent to
 \begin{align}
  & \|x_{\pm1}\|_0=k=\|x_{\pm1}\|_1< \|x\|_1\le\|x\|_0,
\end{align}
implying that $x_{\pm1}$ is the unique minimizer.

Next, we turn to proving (i) $\Rightarrow$ (ii). For this, let $u\in \ker(A)\cap H_{K_1,K_{-1}}$. Notice that the linearity of $\ker(A)\cap H_{K_1,K_{-1}}$ implies
that we may assume $\|u\|_{\infty}\le1$. Thus $u_i\in [-1,0]$ for $i\in K_1$ and $u_i\in [0,1]$ for $i\in K_{-1}$. Next, set
\[
w:= x_{\pm1}+ u.
\]
By the previous considerations, it follows that $w\in [-1,1]^N$. Moreover, since $Aw=A(x_{\pm1}+u)=A(x_{\pm1})$, the vector $w$ is feasible for
\eqref{PM}. By assumption, $k=\|x_{\pm1}\|_1<\|w\|_1$ follows. Therefore
\begin{align}
 k< \sum_{i\in K_1}w_i-\sum_{i\in K_{-1}}w_i+\sum_{i\in K^C}|w_i|= \sum_{i\in K_1} 1 +u_i-\sum_{i\in K_{-1}}(-1+u_i)+ \sum_{i\in K^C}|u_i|,
 \end{align}
 which is equivalent to
 \begin{align}
  -\sum_{i\in K_1}u_i+\sum_{i\in K_{-1}}u_i+<\sum_{i\in K^C}|u_i|, \quad \mbox{i.e.,} \;  \|u_K\|_1<\|u_{K^C}\|_1.
 \end{align}
This shows that $u\notin N_K$, which proves (ii).
\end{proof}

\subsection{Phase Transition}\label{sec:PhaseTransPM}

In the previous subsection, we gave a necessary and sufficient condition for a sensing matrix to ensure unique recoverbility
of a bipolar ternary signal via \eqref{PM}. Our next step will be to study Gaussian matrices $A\in \R^{m \times N}$
and show that those satisfy this condition with high probability provided that $m$ is sufficiently large.

\begin{theorem}\label{mainthm2}
Fix a tolerance $\varepsilon >0$. Let $K\subset [N]$, $A\in \R^{m \times N}$ be Gaussian, and $b=Ax_{\pm1}$. Further set $k=|K|$. With
\begin{align}
 \Delta_{\pm\text{ter}}(k)=\inf_{\tau \ge 0} \left[ k\int^{\tau}_{-\infty}(u-\tau)^2\phi(u)du+ 2(N-k)\int_{\tau}^{\infty}(u-\tau)^2\phi(u)du \right],
\end{align}
the program \eqref{PM} will succeed
to recover $x_{\pm1}$ uniquely with probability larger than $1-\varepsilon$ provided that
\begin{align}
m\ge \Delta_{\pm\text{ter}}(k)+\sqrt{8\log(4/\varepsilon)N}.
\end{align}
\end{theorem}

For an illustration of the statistical dimension $\Delta_{\pm\text{ter}}(k)$ and therefore of the phase transition of \eqref{PM}, 
we refer to Figure \ref{fig:StatPM}. It may come as a surprise that we require more measurements $Ax_{\pm1}$ to recover bipolar ternary 
signals than to recover positive-valued signals; the reason being that one could shift the bipolar ternary signal to the positive axis, 
in which situation it seems to be less complex to recover a bipolar ternary signal than to recover a positive-valued signal. However, 
though the shifted signal $x_\pm=\mathds{1}_{[N]}$ is positive-valued, it does not need to be sparse.

\begin{figure}[ht]
\begin{center}
\includegraphics[scale=0.5]{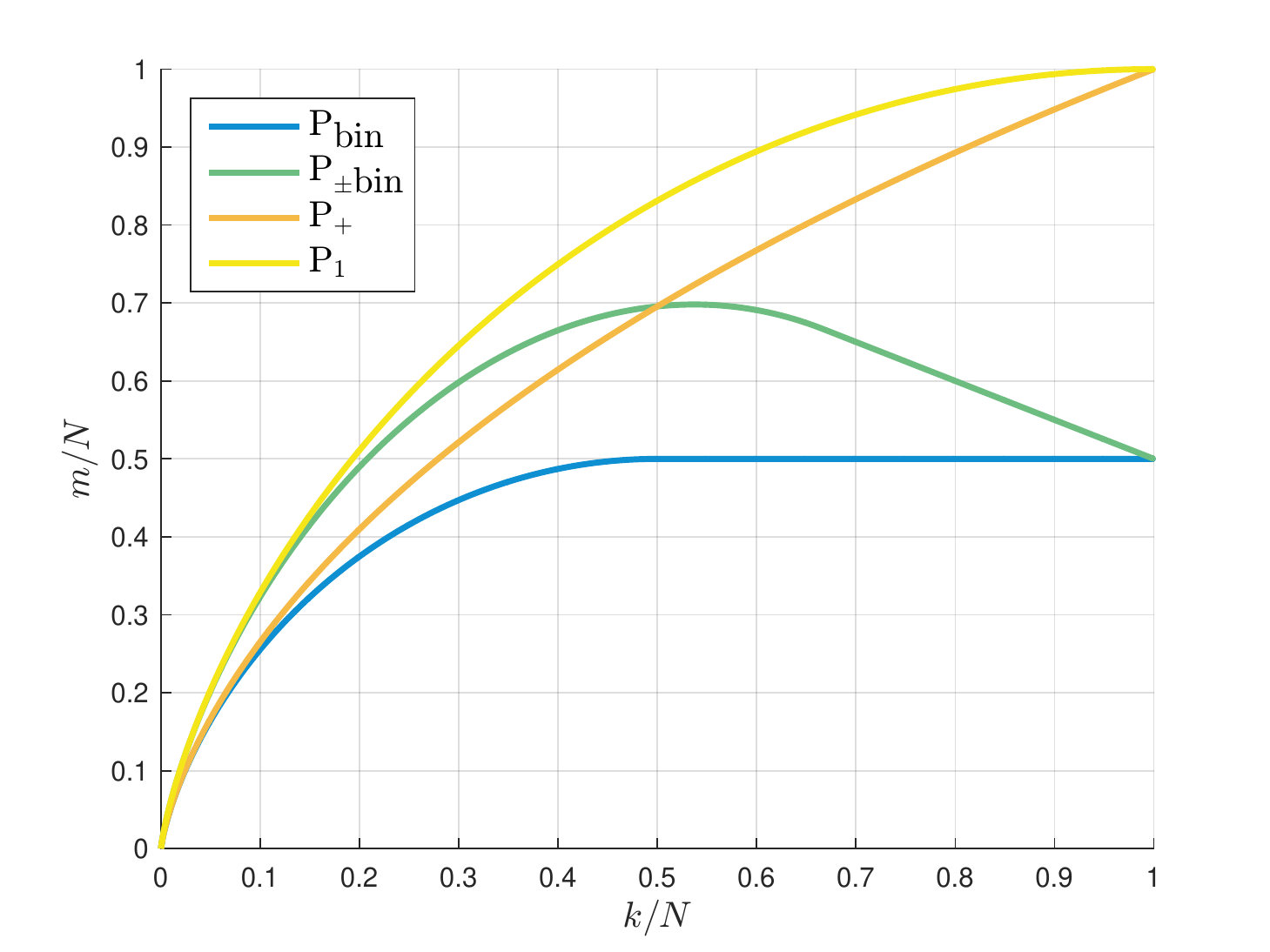}\includegraphics[scale=0.5]{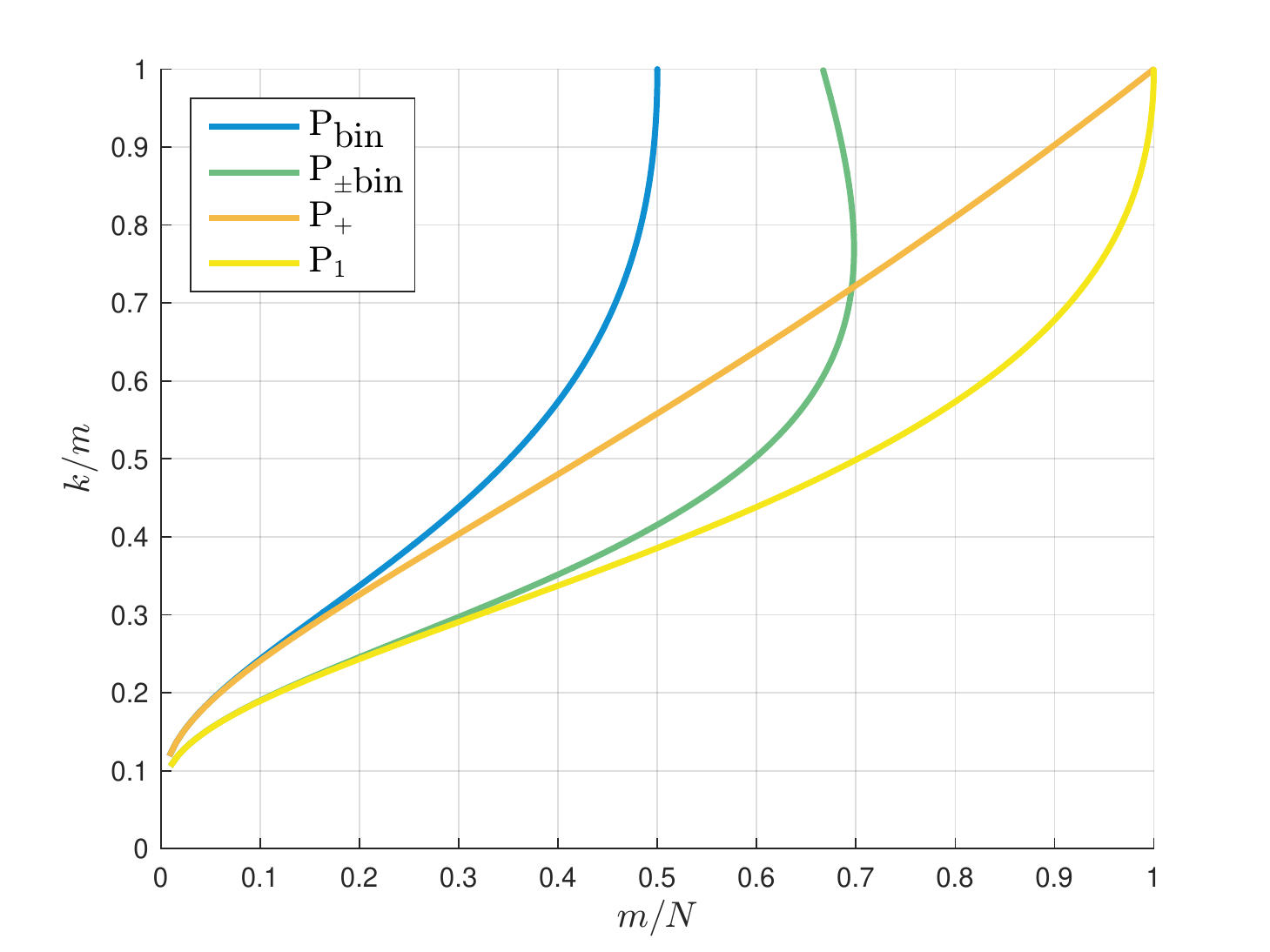}
\put(-320,-15){(a)}
\put(-108,-15){(b)}
\caption{Phase transition of the convex programs \protect\eqref{P1}, \protect\eqref{PM}, \protect\eqref{P+} and \protect\eqref{Pbin}. For the convenience of the reader, 
the following two illustrations are provided: Successful recovery related to the area above the curves in (a), and below the curves in (b).}
\label{fig:StatPM}
\end{center}
\end{figure}

To prove the previous result, we follow the same strategy as in the unipolar binary case and exploit properties of the statistical dimension. For this, 
we will consider the convex function $f$ given by
\begin{align} \label{eq:f_ter}
 f(x)=\begin{cases}
        \|x\|_1 & \text{ if } x\in [-1,1]^N,\\
        \infty & \text{ otherwise},
       \end{cases}
\end{align}
and describe its subdifferential in the next theorem.

\begin{theorem}
Let $f$ be defined as in \eqref{eq:f_ter}, and let $x_{\pm1} =\mathds{1}_{K_1}-\mathds{1}_{K_{-1}}$, where $K_1, K_{-1}$ are disjoint subsets of $[N]$. Then
\begin{align}\label{eqn:SubDiffTer}
\partial f(x_{\pm1})
=\{s\in \R^N: s_i\ge 1 \text{ for } i \in K_1 , \; s_i\le -1 \text{ for } i \in K_{-1}, \; |s_i|\le 1 \text{ for } i\in (K_{-1} \cup K_1)^C \}.
\end{align}
Furthermore, we obtain
\begin{align}
\delta(\mathcal{D}(f,x_{\pm1})) \le \Delta_{\pm\text{ter}}(k):=\inf_{\tau \ge 0} \left[ k\int^{\tau}_{-\infty}(u-\tau)^2\phi(u)du+ 2(N-k)\int_{\tau}^{\infty}(u-\tau)^2\phi(u)du \right],
\end{align}
where $k=|K_{-1}\cup K_{1}|$.
\end{theorem}

\begin{proof}
For $f$ defined as in \eqref{eq:f_ter}, we obtain
\begin{align}
\partial f(x_{\pm1}) &=\{s\in \R^N: \|y\|_1\ge \|x_{\pm1}\|_1+\langle s, y-x_{\pm1}\rangle \text{ for all } y\in [-1,1]^N\}\\
&=\{s\in \R^N: |K_1|+|K_{-1}|\le \|y\|_1-\langle s,y\rangle +\sum_{i\in {K_1}}s_i-\sum_{i\in K_{-1}}s_i \}.
\end{align}
First, we show that the set on the right hand side of \eqref{eqn:SubDiffTer}, say $S$, satisfies $S\subseteq \partial f(x_{\pm1})$. If $s\in S$, then
\begin{eqnarray*}
\lefteqn{\|y\|_1-\langle s,y\rangle +\sum_{i\in {K_1}}s_i-\sum_{i\in K_{-1}}s_i}\\
&=&\sum_{i\in (K_1 \cup K_{-1})^C}(|y_i|-s_iy_i)+\sum_{i\in K_1}\left(|y_i|+s_i(1-y_i)\right)+\sum_{i\in K_{-1}}\left(|y_i|-s_i(1+y_i)\right)\\
&=: I +II +III.
\end{eqnarray*}
To estimate $I$ we will use the fact that for $i\in (K_1 \cup K_2)^C$ it holds  $|s_i|\le 1$ and therefore
\begin{align}
 I= |y_i|-s_iy_i\ge0.
\end{align}
By distinguishing the cases $y_i\le 0$ and $y_i\ge 0$, one can easily estimate that
\begin{align}
 II=|y_i|+s_i(1-y_i)\ge s_i,
\end{align}
since $s_i\ge1$ for $i\in K_1$. And finally, because $s_i\le -1$ for  $i\in K_{-1}$, one can similarly estimate
\begin{align}
 III=|y_i|+s_i(1-y_i)\ge -s_i.
\end{align}
Together this yields
\begin{align}
 I+II+II\ge \sum_{i\in K_1}s_i-\sum_{i\in K_{-1}}s_i
\ge |K_1|+|K_{-1}|,
\end{align}
and therefore $s \in \partial f(x_{\pm1})$.

To show the reverse inclusion $\partial f(x_{\pm1})\subseteq S$, consider the vectors $y^j$ given by
\begin{align}
y^j: =x_{\pm1}+e_j\in[-1,1]^N,  \text{ for } j\in [N] \backslash K_1 \quad \text{and}\quad y^j:=x_{\pm1}-e_j\in[-1,1]^N, \text{ for } j\in [N] \backslash K_{-1}.
\end{align}
This proves Equation \eqref{eqn:SubDiffTer}, which in turn provides
\begin{align}
\mathbb{E}(\dist{\mathbf{g}}{\tau\partial f(x_{\pm1})}^2)=k_{-1}\int_{-\tau}^{\infty}(u-\tau)^2\phi(u)du +k_{1}\int^{\tau}_{-\infty}(u-\tau)^2\phi(u)du+ 2(1-k)\int_{\tau}^{\infty}(u-\tau)^2\phi(u)du,
\end{align}
with $k_i=|K_i|$. Exploring the symmetry of the integral kernel $\phi$ and applying the infimum, concludes the proof.
\end{proof}

\subsection{Robustness}

As in the unipolar binary case, one key aspect is robustness. In the bipolar ternary situation, the according recovery algorithm takes the form
\begin{align}
\label{SBBPDenoising}
\min\|z\|_{1} \quad \text{subject to} \quad \|Az - b \|_{2} \le \eta \quad \text{and} \quad z\in [-1,1]^N \tag{P$_{\pm\text{ter}}^\eta$}.
\end{align}
We start by defining a suitable form of the null space property as follows:

\begin{definition}
A matrix $A \in \R^{m \times N}$ is said to satisfy the \emph{robust bipolar ternary null space property} with constants $0 < \rho < 1$ and $\tau > 0$
relative to two disjoint sets $K_{-1}, K_1 \subset [N]$, if
\noeqref{RSBNSP}
\begin{align}
\sum_{i \in K_{-1}} v_i- \sum_{i \in K_1} v_i\le  \rho \sum_{i \in K^C} |v_i| + \tau \| Av\|_2
\quad
\textnormal{for any } v \in H_{K_1,K_{-1}} \cap \ker(A).\label{RSBNSP}\tag{RBT-NSP}
\end{align}

\end{definition}

Armed with this property, we derive a sufficient condition on the measurement matrix for robust recovery via \eqref{SBBPDenoising}. Since the proof
follows the same strategy as the one of Theorem \ref{thm:BBPDS}, we outsource it to Appendix \ref{sec:app1}.

\begin{theorem}
\label{thm:SBBPD1}
Let $A \in \mathbb{R}^{m \times N}$ satisfy the RBT-NSP with constants $0 < \rho < 1$ and $\tau > 0$ relative to the disjoint sets
$K_{-1},K_1 \subset [N]$. Further, assume the the measurements satisfy $b = A (x_\pm) + e$, where $\|e\|_{2} \le \eta$. Then a
solution $\hat{z}$ of \eqref{SBBPDenoising} approximates $x_{\pm1}$ with $\ell_1$-error
\begin{align}
\| \hat{z} -x_\pm \|_1 \leq \frac{4 \tau }{1 - \rho} \eta.
\end{align}
\end{theorem}

To conclude our analysis for bipolar ternary signals, we state the number of necessary measurements to ensure a robust approximation
of bipolar ternary signals. Since this proof is almost the same as the one of Theorem \ref{thm:PhaseTransNoise}, we omit it.

\begin{theorem}\label{thm:SBBPD2}
Let $x_\pm=\mathds{1}_{K_1}-\mathds{1}_{K_{-1}}$, and let $A\in \R^{m \times N}$ be randomly chosen from the Gaussian distribution.
Assume that noisy measurements $y=Ax_\pm+e$ are taken with $\|e\|\le \eta$. If, for $\varepsilon\in (0,1)$ and $\tau >0$,
 \begin{align}\label{mainthm3.1}
  \frac{m^2}{m+1}\ge \left(\sqrt{\ln(\varepsilon^{-1})}+\sqrt{\Delta_{\pm\text{ter}}(k)}+\tau\right)^2,
 \end{align}
then with probability at least $1-\varepsilon$, every minimizer $\hat{x}$ of \eqref{SBBPDenoising} satisfies
\begin{align}
 \|x-\hat{x}\|_2\le \frac{2\eta}{\tau}.
\end{align}
In particular, if
 \begin{align}
  \frac{m^2}{m+1}> \left(\sqrt{ln(\varepsilon^{-1})}+\sqrt{\partial_\pm(k)}+4\eta\right)^2,
 \end{align}
 then $x$ is the unique solution of \eqref{PbinDR}.
  \end{theorem}

\section{Generalizations to Larger Alphabets}\label{sec:larger}

We have seen that for unipolar binary signals (resp.~bipolar ternary signals), the performance of basis pursuit can be 
enhanced by imposing the additional constraint $x \in [0,1]^N$ (resp.~$x \in [-1,1]^N$). This raises the question of 
whether the improvements can also be carried over to signals with entries in larger finite alphabets such as, for
example, $\mathcal{A}=\{0,1,2\}$ or $\mathcal{A}=\{0,\pm1,\pm2\}$. The answer is yes and the results interestingly
depend only on the number of zero entries and the number of entries having largest amplitude. 

As indicated by the previous 
results, we need to to discuss unipolar finite-valued signals  ($\mathcal{A}=\{0,\dots, L\}$, $L\in \mathbb{Z}$) and 
bipolar finite-valued signals ($\mathcal{A}=\{-L_1,\dots, 0,\dots, L_2\}$, $L_1,L_2\in \mathbb{N}$), separately. 
Let us start with unipolar finite alphabets and note that in this case $L$ can be indeed also chosen to be negative.

\subsection{Unipolar Finite Alphabets}

As a straightforward adaptation of basis pursuit to signals $x_0$ with entries in the alphabet $\mathcal{A}=\{0,\dots,L\}$, 
to which we refer as \emph{unipolar finite-valued signals}, we obtain the minimization given by
\begin{align}
  \min\|x\|_1 \text{ subject to } Ax=b \text{ and } x\in [0,L]^N  \label{PUF}\tag{$P_{\mathcal{UF}}$}.
\end{align}
The solvability of this problem can be characterized by the following variant of NSP.

\begin{definition}
Let $K_L\subset K\subset [N]$. A matrix $A\in \R^{m\times N}$ is said to satisfy the \emph{unipolar finite NSP} with respect $K_L$ and $K$, if
 \noeqref{TNSP}
 \begin{align}
  \ker(A)\cap N^+\cap H_{K_L,K^C} =\{0\},\label{TNSP}\tag{$\mathcal{UF}$-NSP}
 \end{align}
where $H_{K_L,K^C}=\{w\in \R^N: w_i\le 0 \text{ for }i \in K_L \text{ and } w_i\ge 0 \text{ for } i \in K^C \}$.
\end{definition}

This enables us to prove that the $\mathcal{UF}$-NSP is indeed necessary and sufficient to recover a unipolar finite-valued signal via \eqref{PUF}. In the following we will always denote the unipolar binary vector we aim to recover by
\begin{align}\label{NotU:1}
x_0=\sum_{i=1}^Li\mathds{1}_{K_i},
\end{align}
with $K_1,\dots, K_L$ being disjoint subsets. We will also sometimes used
\begin{align}\label{NotU:2}
 K=\bigcup_{i=1}^L K_i \quad \text{and} \quad \hat{K}=K\setminus K_L.
\end{align}

\begin{theorem}\label{thm:NSPUF}
Let $x_0$ be defined as in \eqref{NotU:1} and $K,\hat{K}$ as in \eqref{NotU:2}. Further let $A\in \R^{m\times N}$. 
Then the following conditions are equivalent:
\begin{enumerate}
\item[(i)] The measurement matrix $A$ satisfies $\mathcal{UF}$-NSP with respect to sets $K_L$ and $K$.
\item[(ii)] The vector $x_0$ is the unique solution of \eqref{PUF} with $b=Ax_0$.
\end{enumerate}
\end{theorem}

The proof of Theorem \ref{thm:NSPUF} follows the same arguments as the proofs for Theorem \ref{thm:BinNec} and \ref{thm:BTNSP}, and is therefore 
omitted. At a glance ,this result may look promising for the recovery of {unipolar finite-valued} signals. It states that recovery guarantees can 
be highly improved as long as the set $\bigcup_{i=1}^{L-1}K_i$ is not too large in comparison to $K$ (see also Figure \ref{fig:StatTer}). However, 
the following theorem clarifies that it does not entirely use the discrete structure of unipolar finite-valued signals.

\begin{theorem}\label{thm:terKsimple}
Let $x_0$ be defined as in \eqref{NotU:1} and $K,\hat{K}$ as in \eqref{NotU:2}. Further let $A\in \R^{m\times N}$.
Then the following conditions are equivalent:
\begin{enumerate}
\item[(i)] The vector $x_0$ is the unique solution of \eqref{PUF} with $b=Ax_0$.
\item[(ii)] Every $\hat{x} \in \R^N$ of the form $\hat{x}=\tilde{x}_{\hat{K}}+L\mathds{1}_{K_L}$ with $\tilde{x}\in (0,L)^N$ 
is the unique solution of \eqref{PUF} with $b=Ax$.
\end{enumerate}
\end{theorem}

\begin{proof}
  It is clear that (i) follows from (ii), as $x_0$ particularly is of the form stated in ii). For the other direction let $x_0$ be the unique solution of \eqref{PUF} and let $\hat{x}=\tilde{x}_{\hat{K}}+L\mathds{1}_{K_L}$ with 
 $\tilde{x}\in (0,L)^N$. Respecting Theorem \ref{thm:NSPUF} the measurement matrix $A$ needs to fulfill the $\mathcal{UF}$-NSP condition.
 Let now $z\in \R^N$ feasible for $\eqref{PUF}$ with $b=Ax$, then $z-\hat{x}\in \ker{A}\cap H_{K_L,K^C}$ and hence, $z-\hat{x}\notin N^+$. This shows
 that $\|\hat{x}\|_1< \|z\|_1$ and therefore that $x$ is the unique solution.
\end{proof}

The last theorem states that recovery guarantees only depend on the entries of the signal which are equal to zero or having largest amplitude.

It remains to compute the number of necessary measurements for \eqref{PUF} to succeed in the case that $A\in \R^{m \times N}$ is a Gaussian matrix. With the same ideas as used in Subsections \ref{sec:PhaseTransBin} and \ref{sec:PhaseTransPM}, we will compute the statistical dimension of the cone $N^+\cap H_{K_L,K^C}$ for a unipolar finite valued signal $x_0=\sum_{i=1}^Li\mathds{1}_{K_i}$, with $K_1,\dots, K_L$ disjoint subsets in $[N]$. As indicated by Theorem \ref{thm:terKsimple} the descent cone does not depend on the fact that for $i\in [L-1]$ the entries of $x$ on $K_i$ are equal to $i$.

The phase transition depends highly on size of $K_L$ relative to $K=\bigcup_{i=1}^LK_i$. In the worst case, namely that $K_L=\emptyset$, the phase transition 
coincides with the one of \eqref{P+}; in the best case, namely that $\hat{K}=\emptyset$, it is as good as for \eqref{Pbin}.
For an illustration we refer to Figure \ref{fig:StatTer}.

\begin{figure}
\begin{center}
\includegraphics[scale=0.5]{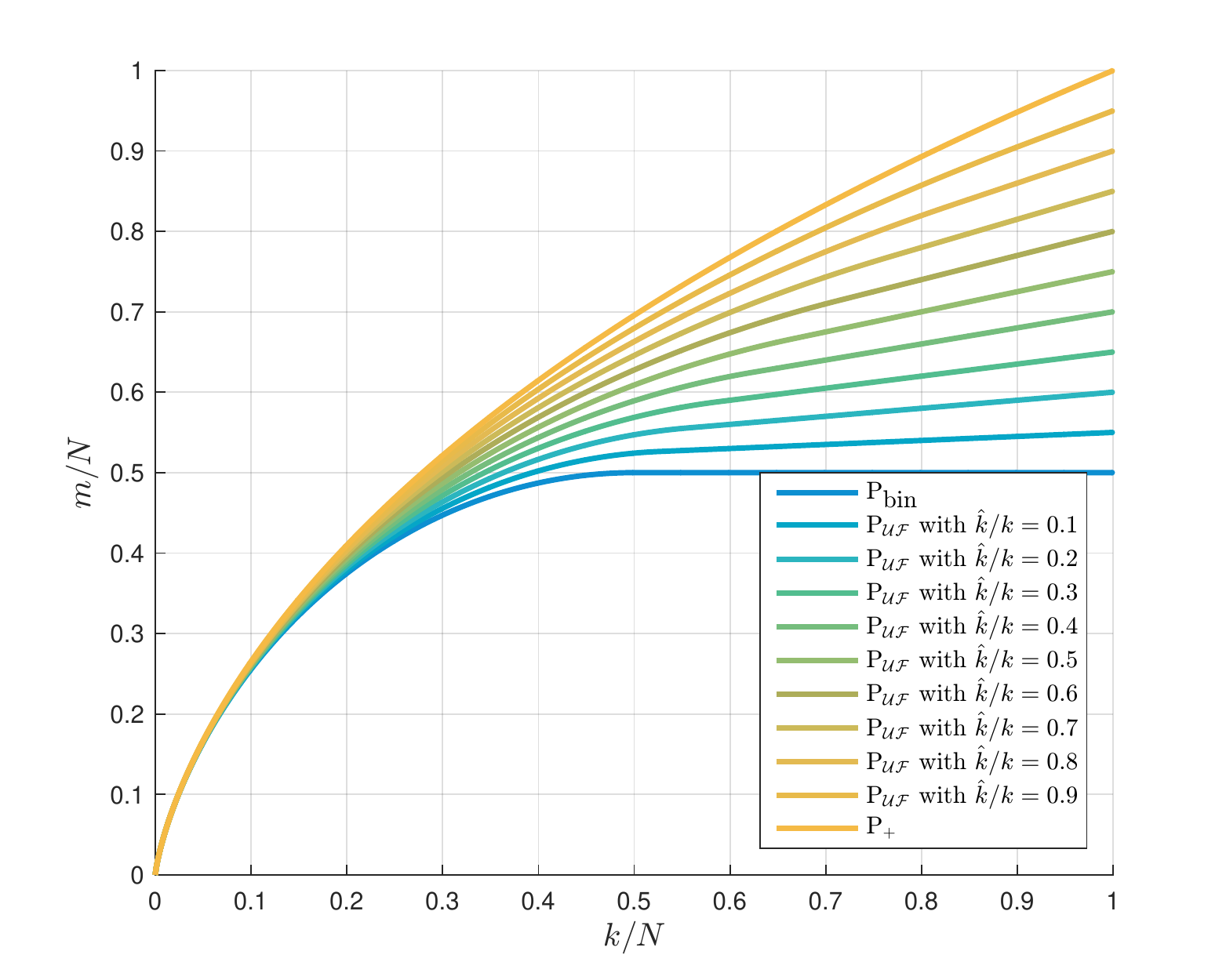}\includegraphics[scale=0.5]{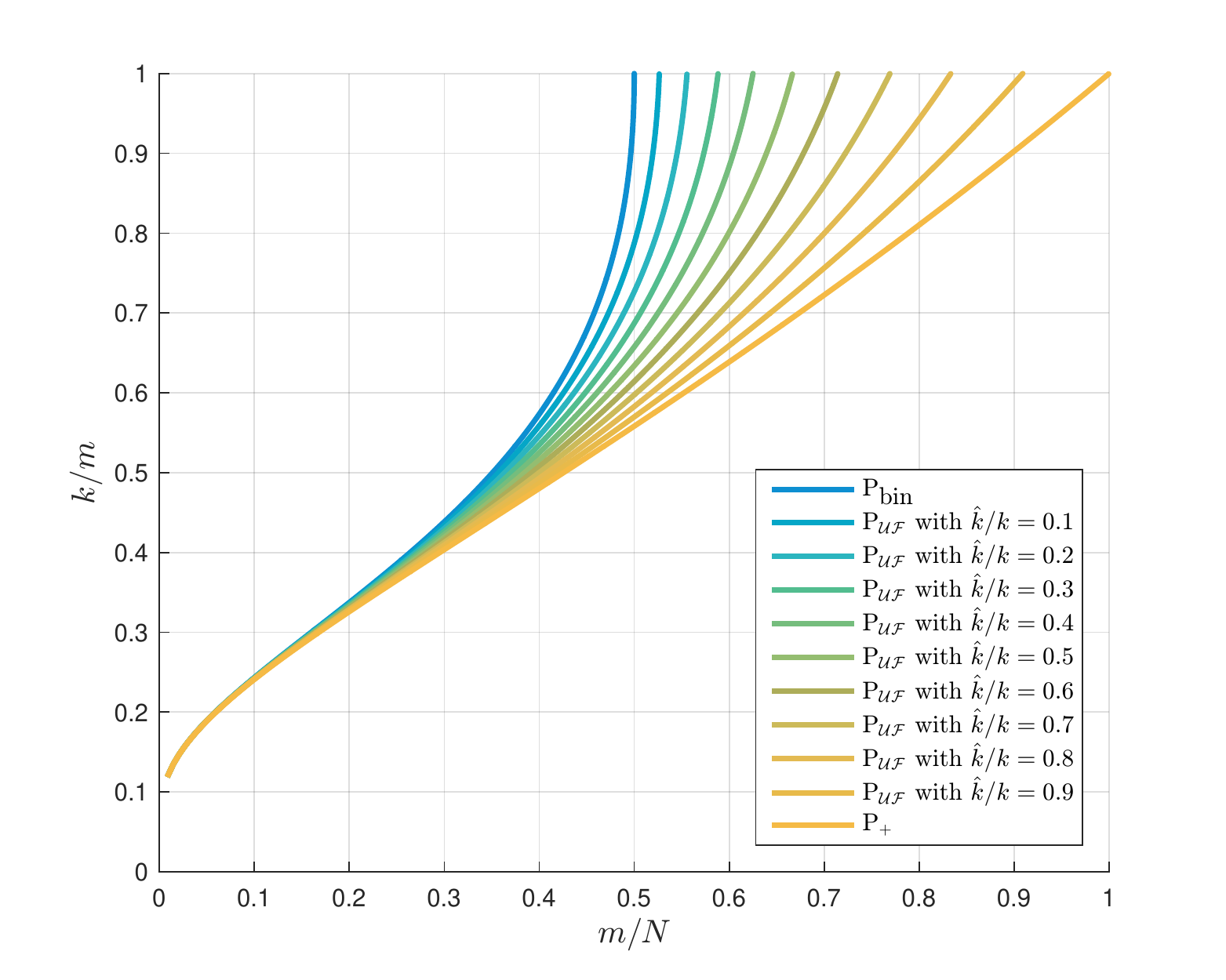}
\put(-343,-15){(a)}
\put(-115,-15){(b)}
\caption{Phase transition of the convex program \protect\eqref{PUF} according to the ratio of $\hat{k}$ to $k$, where $k$ is the size of the entire support of a 
unipolar finite-valued signal and $\hat{k}$ the number of entries in the signal not equal to zero or to the largest value of the given alphabet. For the convenience of the reader,
the following two illustrations are provided: Successful recovery related to the area above the curves in (a), and below the curves in (b).}
\label{fig:StatTer}
\end{center}
\end{figure}

\begin{theorem}\label{mainthm3}
Fix a tolerance $\varepsilon >0$. Let $x_0$ be defined as in \eqref{NotU:1} and $K,\hat{K}$ as in \eqref{NotU:2}. Further let $A\in \R^{m \times N}$ be Gaussian and $b=Ax_0$. 
With $\hat{k}=|\hat{K}|$, and $k=|K|$, provided
\begin{align}
 m\ge \inf_{\tau\ge0}\left\{\hat{k}(1+\tau^2)+k_L\int_{-\infty}^{\tau}(u-\tau)^2\phi(u)du + (N-k)\int_{\tau}^{\infty}(u-\tau)^2\phi(u)du,\right\}+\sqrt{8\log(4/\varepsilon)N},
\end{align}
the program \eqref{PUF} will succeed to recover $x$ uniquely with probability larger $1-\varepsilon$.
\end{theorem}

Note that Theorem \ref{mainthm3} is true for signals of the form $\hat{x}=\tilde{x}_{\hat{K}}+L\mathds{1}_{K_L}$, where $\tilde{x}\in (0,L)^N$. As for Theorem \ref{mainthm1}, the proof is straightforward once we have identified the subdifferential 
of $f$. In the following lemma will describe the subdifferential for the more general signal $\hat{x}$, to find that all $\hat{x}$, including $x_0$, have the same subdifferential. However, due to the similarity to the proof of Lemma \ref{lem:subdiff}, we will postpone the proof to the Appendix \ref{subsecAStat1}.

\begin{lemma}\label{lem:subUF}
Let $\hat{x}=\tilde{x}_{\hat{K}} + L\mathds{1}_{K_L}$ for some $\tilde{x}\in (0,L)^N$ and $K_L,\hat{K}$ be 
two disjoint subsets of $[N]$. The descent cone of
\begin{align}
f(x)=\begin{cases}\label{fUF}
        \|x\|_1 & \text{ if } x\in [0,L]^N,\\
        \infty & \text{ otherwise}
       \end{cases}
\end{align}
is given by $\mathcal{D}(f,\hat{x})=N_{\hat{K}\cup K_L} \cap H_{K_L,K^C}$, where $K=K_L\cup\hat{K}$, and the corresponding subdifferential is
\begin{align}\label{eq:subUF}
\partial f(\hat{x})=\{s\in \R^N: s_i \ge 1 \text{ for } i\in K_L, s_i=1 \text{ for } i\in \hat{K}, s_i\le 1 \text{ for } i\in K^C \}.
\end{align}
In particular,
\begin{align}J_{\hat{k},k_L}(\tau)=& \mathbb{E}[\dist{\mathbf{g}}{\tau \partial f(\hat{x})}^2 ] \\
=&\hat{k}(1+\tau^2)+k_L\int_{-\infty}^{\tau}(u-\tau)^2\phi(u)du + (N-k)\int_{\tau}^{\infty}(u-\tau)^2\phi(u)du,
\end{align}
where $\hat{k}=|\hat{K}|$, $k_L=|K_L|$ and $k=|K|$.
\end{lemma}

\subsection{Bipolar Finite Alphabets}

Finally we will analyze recovery guarantees for \emph{bipolar finite-valued} signals having entries in an alphabet of the 
form $\mathcal{A}=\{-L_1,\dots, L_2\}$, with $L_1$ and $L_2$ being positive integers. As before we will first introduce 
an NSP condition of the measurement matrix $A\in \R^{m\times N}$ which is necessary and sufficient for basis pursuit with box 
constraints to successfully recover bipolar finite-valued signals $x_0$ from the measurements $Ax_0$. We will then compute 
the phase transition for Gaussian matrices, i.e., the sufficient number of measurements such that basis pursuit with box 
constraints will succeed with high probability. In the following we will always denote the bipolar finite-valued signal we 
aim to recover by
\begin{align}\label{Not:1}
 x_0=\sum_{i=-L_1}^{L_2}i\mathds{1}_{K_i},
\end{align}
with $K_{-L_1},\dots, K_{L_2}\subset [N]$ being disjoint and $\bigcup_{i=-L_1}^{L_2}K_i=[N]$. We will also sometimes use
\begin{align}\label{Not:2}
 K=\bigcup_{i\in \{-L_1,\dots,L_2\}\setminus \{0\}}K_i \quad \text{ and } \quad \hat{K}=K\setminus(K_{-L_1}\cup K_{L_2})
\end{align}
and, hence, $K_0=K^C$.

The straightforward adaptation of basis pursuit to bipolar finite-valued signals is the program given by
\begin{align}
  \min\|x\|_1 \text{ subject to } Ax=b \text{ and } x\in [-L_1,L_2]^N  \label{PBFin}\tag{$P_{\mathcal{F}}$}.
\end{align}

The following variant of the NSP characterizes the solvability of this program. 

\begin{definition}
 Let $K_{-L_1},K_{L_2}\subset K \subset [N]$ with $K_{-L_1}\cap K_{L_2}=\emptyset $. A matrix $A\in \R^{m\times N}$ is said to satisfy 
 the \emph{finite NSP} with respect $K_{-L_1}$, $K_{L_2}$, and $K$, if
 \noeqref{FNSP}
 \begin{align}
  \ker(A)\cap N_K\cap H_{K_{L_2},K_{-L_1}} =\{0\},\label{FNSP}\tag{$\mathcal{F}$-NSP}
 \end{align}
where $H_{K_{L_2},K_{-L_1}}=\{w\in \R^N: w_i\le 0 \text{ for }i \in K_{L_2} \text{ and } w_i\ge 0 \text{ for } i \in K_{-L_1} \}$.
\end{definition}

 We will now see that the $\mathcal{F}$-NSP is indeed necessary and sufficient to recover a bipolar finite-valued signal via \eqref{PBFin}.

\begin{theorem}\label{thm:NSPBF}
Let $x_0 \in \R^N$ and $K_{-L_1},K_{L_2}\subset [N]$ be defined as in \eqref{Not:1} and \eqref{Not:2}, and let $A\in \R^{m\times N}$. 
Then the following conditions are equivalent:
\begin{enumerate}
\item[(i)] The vector $x_0$ is the unique solution of \eqref{PBFin} with $b=Ax_0$.
\item[(ii)] The matrix $A$ fulfills the $\mathcal{F}$-NSP with respect to the sets $K_{-L_1}, K_{L_2}$ and $K=\bigcup_{i\in\{-L_1,\dots, L_2\}\setminus \{0\}}K_i$.
\end{enumerate}
\end{theorem}

It remains to compute the number of measurements for \eqref{PBFin}, which are sufficient to succeed in the case that $A\in \R^{m \times N}$ 
is a Gaussian matrix. With the same ideas as used in Subsections \ref{sec:PhaseTransBin} and \ref{sec:PhaseTransPM}, we will compute the statistical dimension of the 
cone corresponding to \eqref{PBFin} for a bipolar finite signal $x_0$.

With the convex function
\begin{align}\label{fBF}
 f(x)=\begin{cases}
    \|x\|_1 & \text{for} \quad x\in [-L_1,L_2]^N\\
    \infty & \text{otherwise,}
   \end{cases}
\end{align}
the descent cone $\mathcal{D}(f,x_0)$ of $f$ at $x_0$ equals precisely the cone  $N_K\cap H_{K_{L_2},K_{-L_1}}$. Thus, we now describe the 
subdifferential of $f$ at $x_0$ in order to compute the statistical dimension of the cone $N_K\cap H_{K_{L_2},K_{-L_1}}$. The proof of the following lemma is again postponed to the Appendix \ref{subsecAStat2}.

\begin{lemma}\label{lem:subBF}
Let $x_0 \in \R^N$ and $K_{-L_1},K_{L_2}\subset [N]$ be defined as in \eqref{Not:1} and \eqref{Not:2}. Further let $\hat{K}^+=\bigcup_{i=1}^{L_2-1}K_i$ and $\hat{K}^-=\bigcup_{i=-L_1+1}^{-1}K_i$. Then the subdifferential takes the form
 \begin{align}
  \partial f(x_0)=\left\{s\in \R^N: s_i\ge 1 \text{ for }  i\in K_{L_2}, s_i\le -1 \text{ for }i\in K_{-L_1}, s_i=1\text{ for }  i\in \hat{K}^+,\right.\\ s_i=-1\text{ for }  i\in \hat{K}^-
  \left.\text{ and } |s_i|\le 1 \text{ for } i\in K_0\right\}.
 \end{align}
\end{lemma}

With this description we can prove the following theorem on the statistical dimension and therefore on the sufficient number of measurements. 
Due to the similarity to the proofs of the corresponding theorems in the last subsections, we again omit the proof. The
phase transition which this result determines is illustrated in Figure \ref{fig:StatBFin}.

\begin{theorem}\label{mainthm4}
Fix a tolerance $\varepsilon >0$. Let $A\in \R^{m \times N}$ be Gaussian, $b=Ax_0$ and let $x_0$, $K$, and $\hat{K}$ be defined as in 
\eqref{Not:1} and \eqref{Not:2}. Further set $k_i=|K_i|$, for $i\in \{-L_1,0,L_2\}$, and $\hat{k}=|\hat{K}|$. If
\begin{align}
 m\ge& \inf_{\tau\ge0}\left\{\hat{k}(1+\tau^2)+k_{-L_1}\int^{\infty}_{-\tau}(u-\tau)^2\phi(u)du+k_{L_2}\int_{-\infty}^{\tau}(u-\tau)^2\phi(u)du + k_0\int_{\tau}^{\infty}(u-\tau)^2\phi(u)du,\right\}\\&+\sqrt{8\log(4/\varepsilon)N},
\end{align}
then \eqref{PBFin} will succeed to recover $x$ uniquely with probability larger $1-\varepsilon$.
\end{theorem}

\begin{figure}
\begin{center}
\includegraphics[scale=0.5]{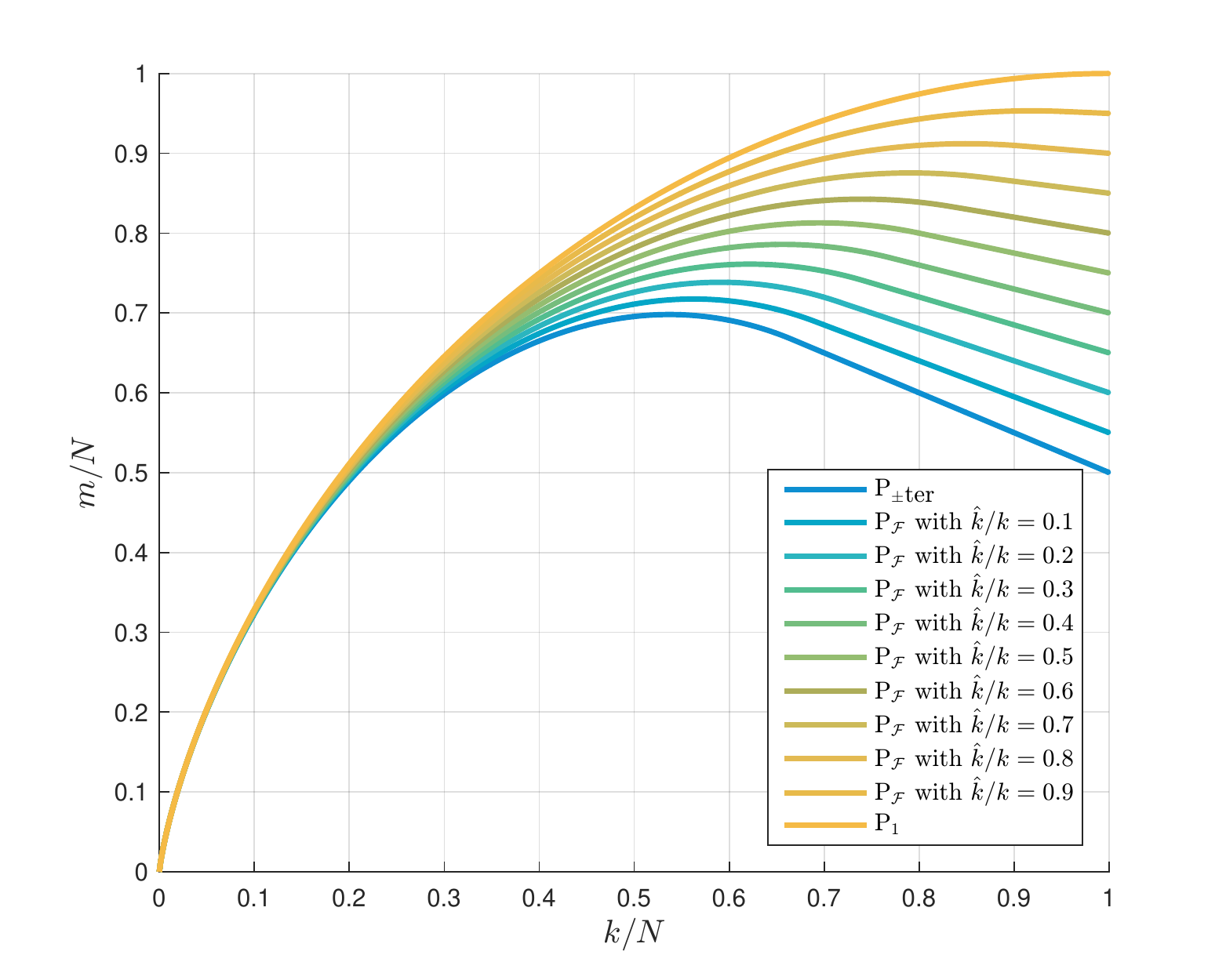}\includegraphics[scale=0.5]{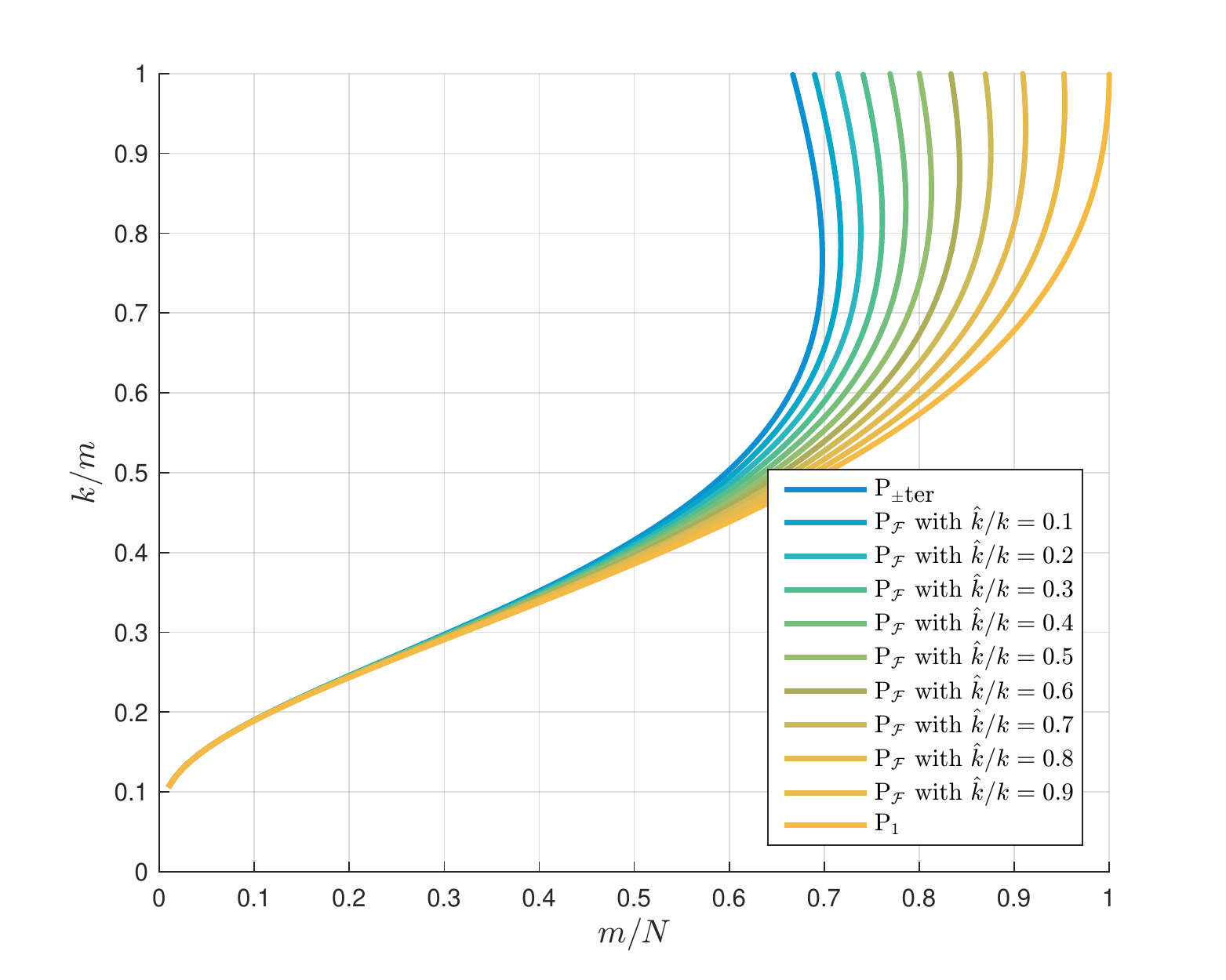}
\caption{Phase transition of the convex program \protect\eqref{PBFin} according to the ratio of $\hat{k}$ to $k$, where $k$ is the size of the whole support of a bipolar finite-valued signal and $\hat{k}$ the number of entries in the signal not equal to zero, to the smallest or to the largest value of the given alphabet. In the left illustration recovery is likely above the curves and in the right illustration below the curves.}\label{fig:StatBFin}
\end{center}
\end{figure}

We conclude the theoretical part with an overview of the different NSP conditions.

\begin{remark}\rm
Table \ref{table1} provides an overview of the NSP conditions introduced throughout this paper. 
 \begin{table}[H]
\begin{center}
\renewcommand{\arraystretch}{2}
\begin{tabular}{ | c | c | c | c | c | }
\hline
$\mathcal{A}$ & NSP & Condition
\\ \hline
$\{ 0, 1 \}$ & B-NSP & $\ker(A)\cap N_K\cap H_{K,K^C} =\{0\}$ or $\ker(A)\cap N^+\cap H_{K,K^C} =\{0\}$
\\ \hline

$\{ 0, \pm 1 \}$ & BT-NSP & $\ker(A)\cap N_K\cap H_{K_1,K_{-1}} =\{0\}$
\\ \hline
$\{ 0, \dots, L \}$ & $\mathcal{UF}$-NSP & $\ker(A)\cap N^+\cap H_{K_L,K^C} =\{0\}$
\\ \hline
$\{ -L_1, \dots, L_2 \}$ & $\mathcal{F}$-NSP & $ \ker(A)\cap N_K\cap H_{K_{L_2},K_{-L_1}} =\{0\}$
\\ \hline
\end{tabular}
\caption{A summary of all NSP conditions that appeared in this paper.}\label{table1}
\end{center}
\end{table}
\renewcommand{\arraystretch}{1}
\end{remark}

\section{Numerical Results}\label{sec:numerics}

In this last section we will empirically investigate our results. We first consider the noiseless case with the ambient dimension
$N=500, 1000$. The experiments are then conducted as follows. For each sparsity level $k\in \{0.02,0.04, \dots, 1\}\cdot N$ and
number of measurements $m\in \{0.02,0.04, \dots, 1\}\cdot N$, we draw a support set $K$ of size $k$ uniformly at random, as well
as a Gaussian matrix $A\in \R^{m \times N}$. To check the algorithms \eqref{Pbin} and \eqref{PM} and to compare it with \eqref{P1},
we chose the elements on the support $K$ equal to $1$ to form the vector $x_0$ and for \eqref{PM} we chose the elements on $K$
uniformly at random $\pm 1$.

Figure \ref{fig:Num} shows the reconstruction results of \eqref{Pbin} and \eqref{PM} in the ambient dimension $N=1000$. To illustrate
that the curves appear exactly the same in other dimensions and for
a comparison with \eqref{P+}, Figure \ref{fig:Comp} shows the reconstruction results, but also the runtime, of \eqref{Pbin} and
\eqref{P+} for $N=500$. In Figure \ref{fig:Comp}, we illustrate not only that the reconstruction becomes better, but also the
algorithms are faster. Figure \ref{fig:Num} also points out that the in Sections \ref{sec:PhaseTransBin} and \ref{sec:PhaseTransPM} theoretically derived phase transitions describe the numerical experiments very well. To see this compare Figure \ref{fig:Num}~(a) with Figure \ref{fig:StatBin} as well as Figure \ref{fig:Num}~(b) with Figure \ref{fig:StatPM}. Thus we see that to recover a unipolar binary signals at most around $N/2$ measurements are needed for \eqref{Pbin} to succeed (see Figure \ref{fig:Num}~(a)) and in the case of bipolar ternary signals at most around $0.7N$ measurements are needed for \eqref{PM} to succeed (see Figure \ref{fig:Num}~(b)).

\begin{figure}[ht]
\begin{center}
\includegraphics[scale=0.37]{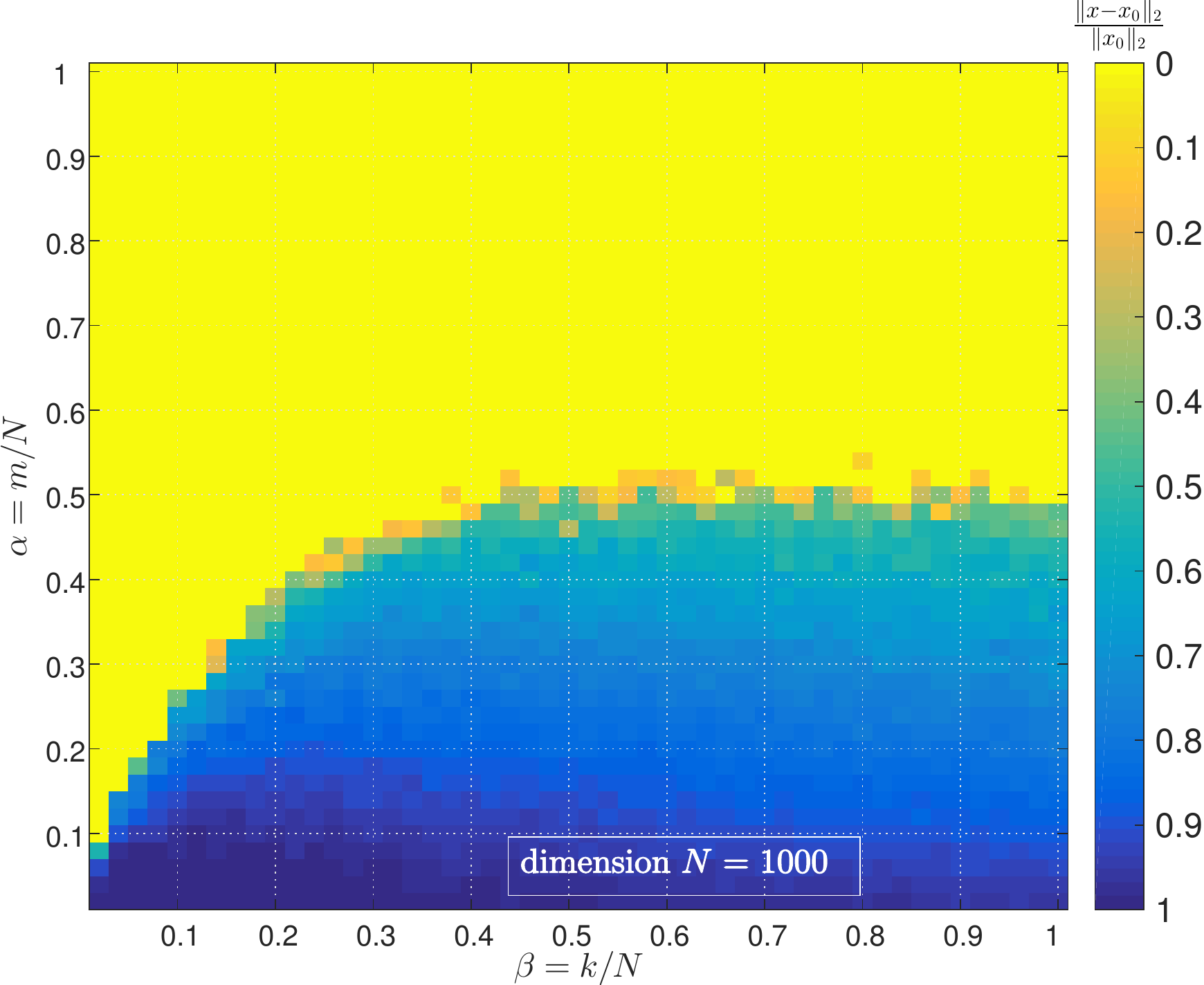}\hspace{1cm}\includegraphics[scale=0.37]{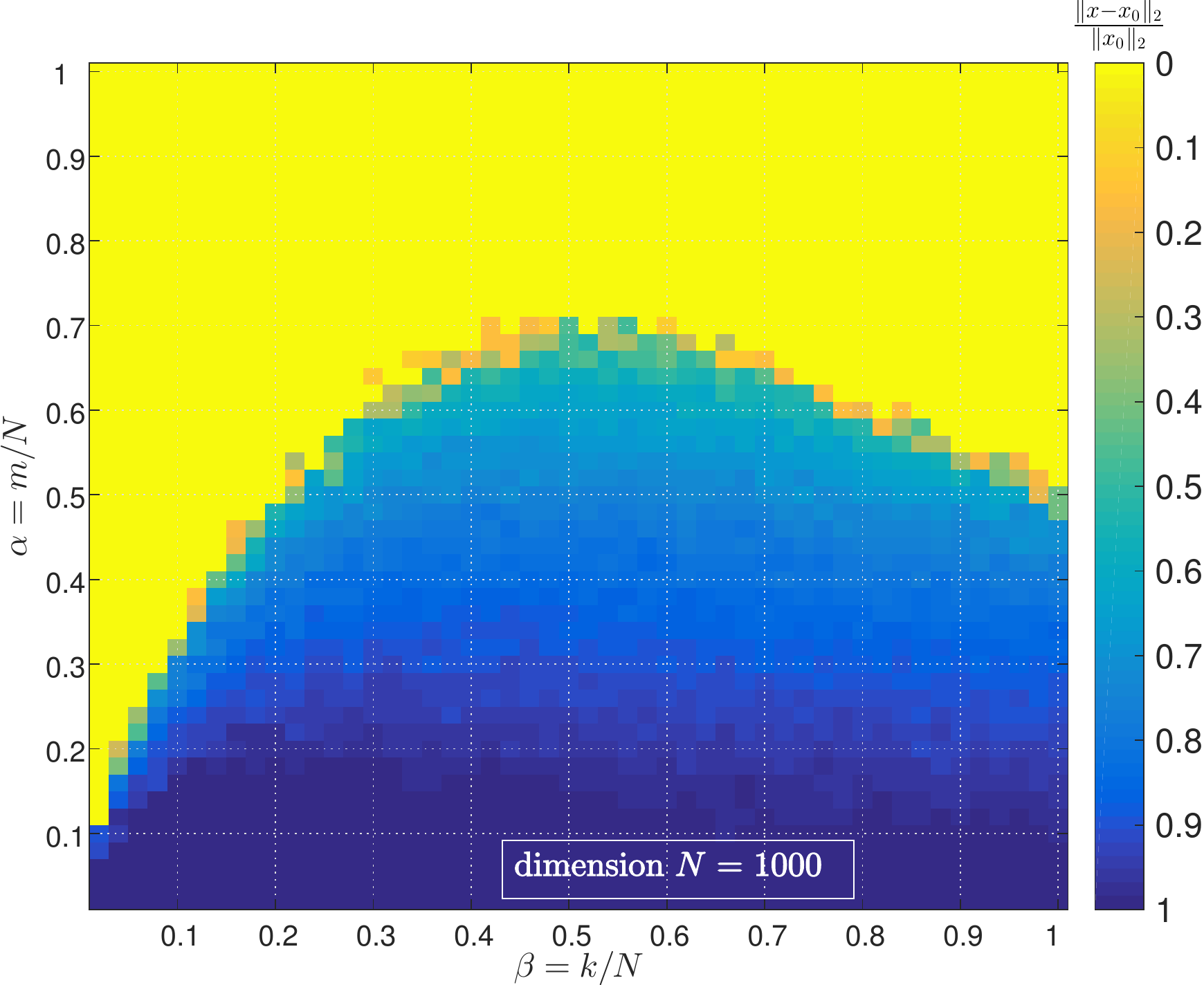}
\put(-310,-15){(a)}
\put(-96,-15){(b)}
\caption{Reconstruction error of the convex programs \protect\eqref{Pbin} for unipolar binary vectors (see (a)) and of \protect\eqref{PM} for bipolar ternary vectors (see (b)) 
depending on the number of measurements $m$ and the support size $k$ in the ambient dimension $N=1000$.}\label{fig:Num}
\end{center}
\end{figure}

\begin{figure}[ht]
\begin{center}
\includegraphics[scale=0.37]{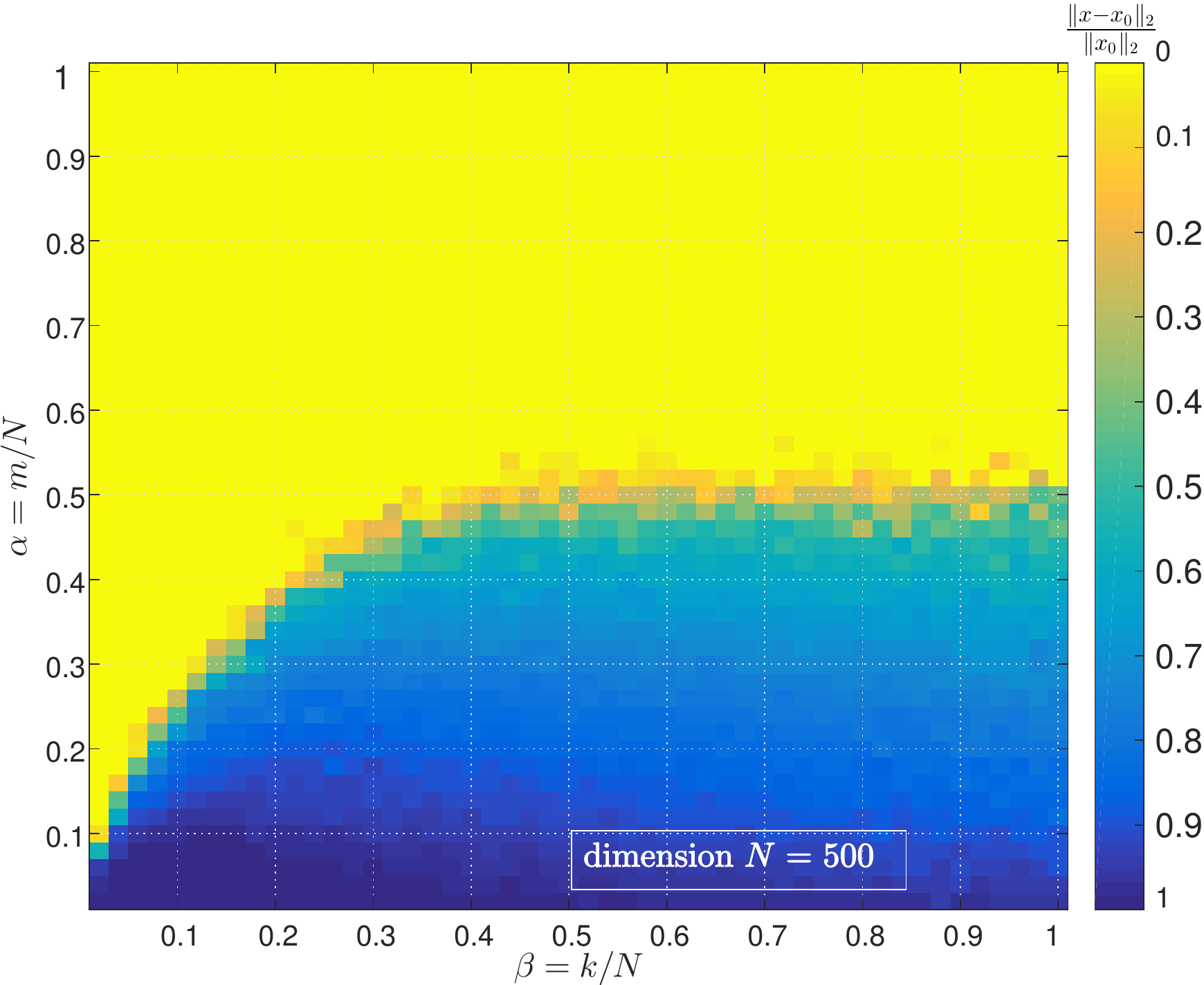}\hspace{1cm}\includegraphics[scale=0.37]{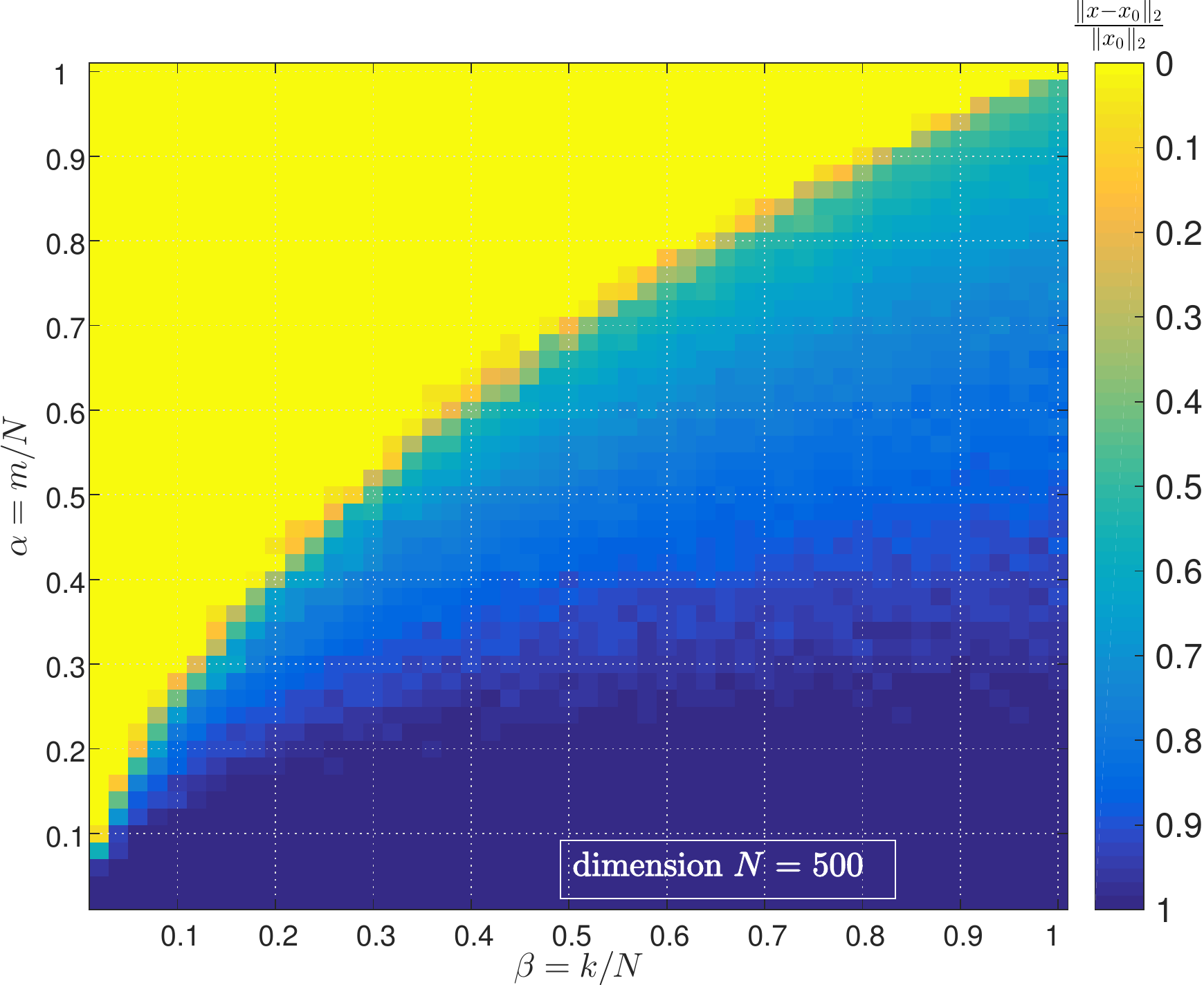}
\put(-310,-15){(a)}
\put(-83,-15){(b)}\\[2ex]
\includegraphics[scale=0.37]{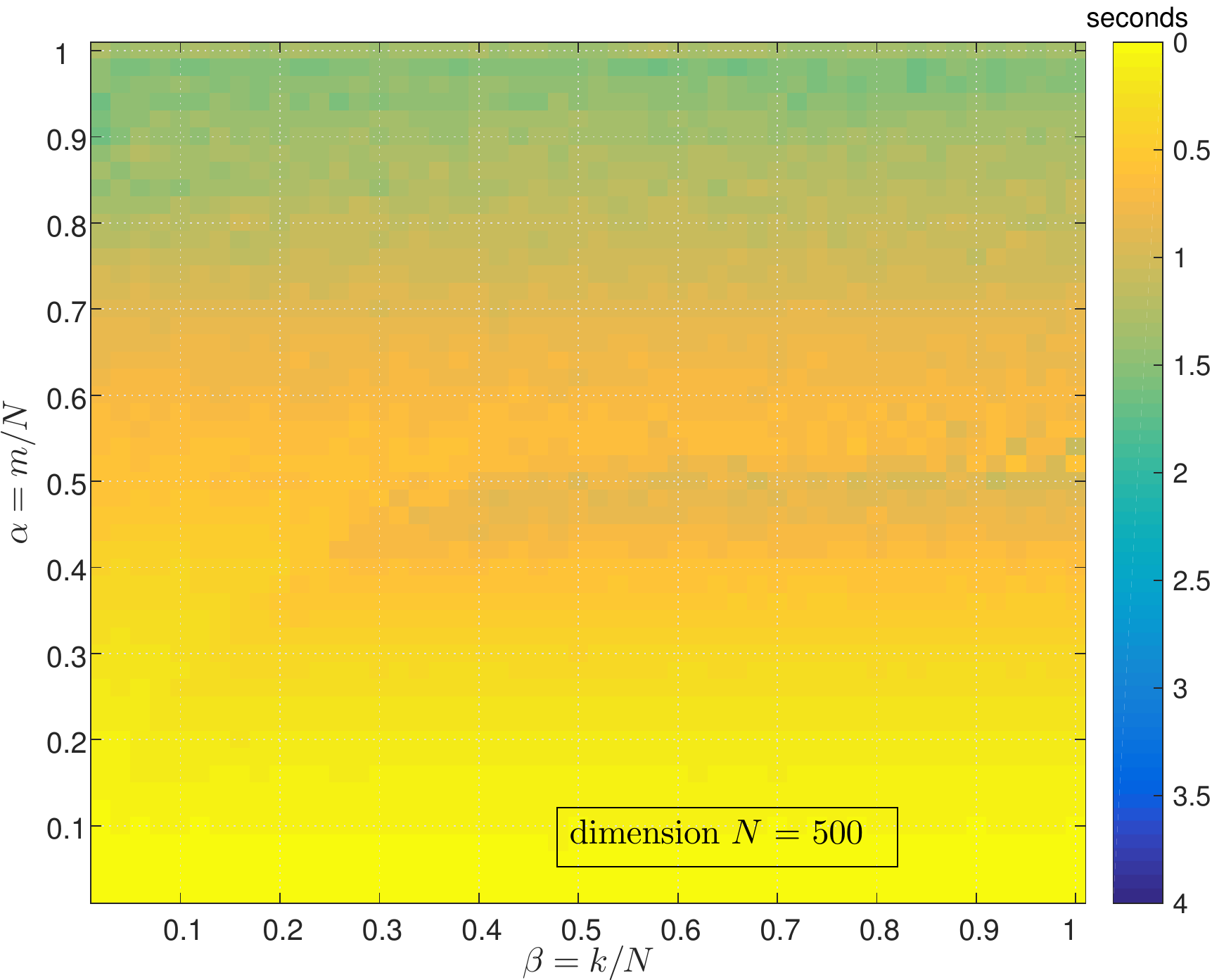}\hspace{1cm}\includegraphics[scale=0.37]{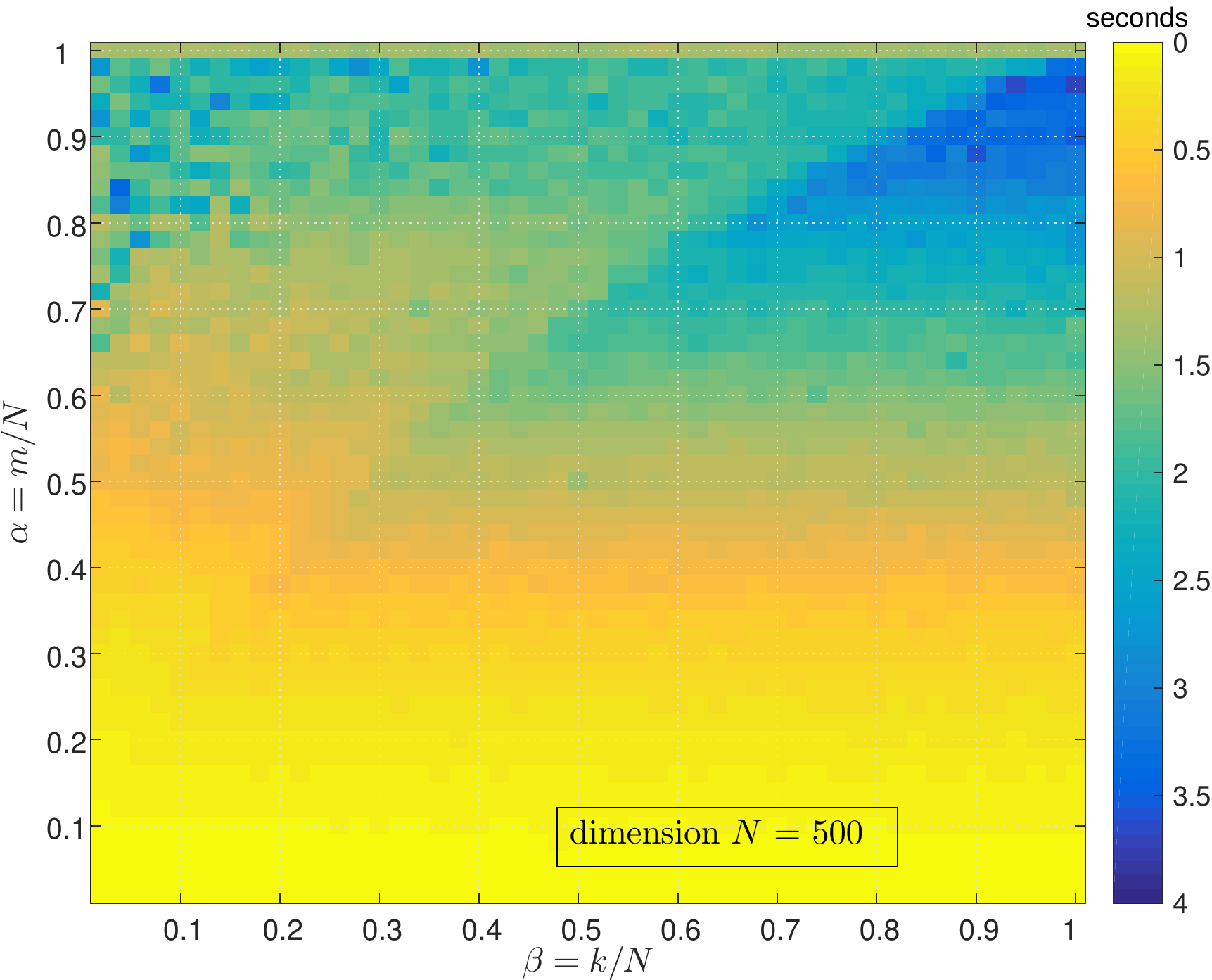}
\put(-310,-15){(c)}
\put(-83,-15){(d)}
\caption{Reconstruction error (see (a) and (b)) and runtime (see (c) and (d)) of the convex programs \protect\eqref{Pbin} (see (a) and (c)) and \protect\eqref{P+} 
(see (b) and (d)) for unipolar binary vectors depending on the number of measurements $m$ and the support size $k$ in the ambient dimension $N=500$.}\label{fig:Comp}
\end{center}
\end{figure}

To investigate the results for \eqref{PUF}, we considered unipolar ternary signals having entries in $\mathcal{A}=\{0,1,2\}$. We draw for 
each number of measurements $m$, as well as for each sparsity level $k$ and for
$k_1\in \{\frac{k}{10}, \frac{3k}{10}, \frac{5k}{10}, \frac{7k}{10}, \frac{9k}{10}\}$ two support sets $K_1$ of size $k_1$ and $K_2$ of size $k-k_1$. We then chose the elements on
$K_1$ equal to $1$ and on $K_2$ equal to $2$ to form the vector $x_0$. Thus, as done theoretically in Section \ref{sec:larger},  we compute the recovery error depending on the ratio of $|K_1|$ and $|K_2|$. Figure \ref{fig:NumTer} illustrates
the results for each $k_1=|K_1|$. If we compare the theoretically in Section \ref{sec:larger} derived phase transition  with the numerically derived one (compare also Figure \ref{fig:StatTer} with Figure \ref{fig:NumTer}~(a)-(e)), we again see that the theory describes the experiments very well. We also see that, as the relative size of $K_1$ to $K$ increases  the number of measurements for \eqref{PUF} to succeed increases. Thus, if $k_1/k$ is very small, the numerically derived phase transition is nearly the same as for \eqref{Pbin} (compare Figure \ref{fig:Num} with \ref{fig:NumTer}~(a)), and if $k_1/k$ is around $1$ it is nearly the same as for \eqref{P+} (compare Figure \ref{fig:Comp} with Figure \ref{fig:NumTer}~(e)).

\begin{figure}[H]
\begin{center}
\includegraphics[scale=0.3]{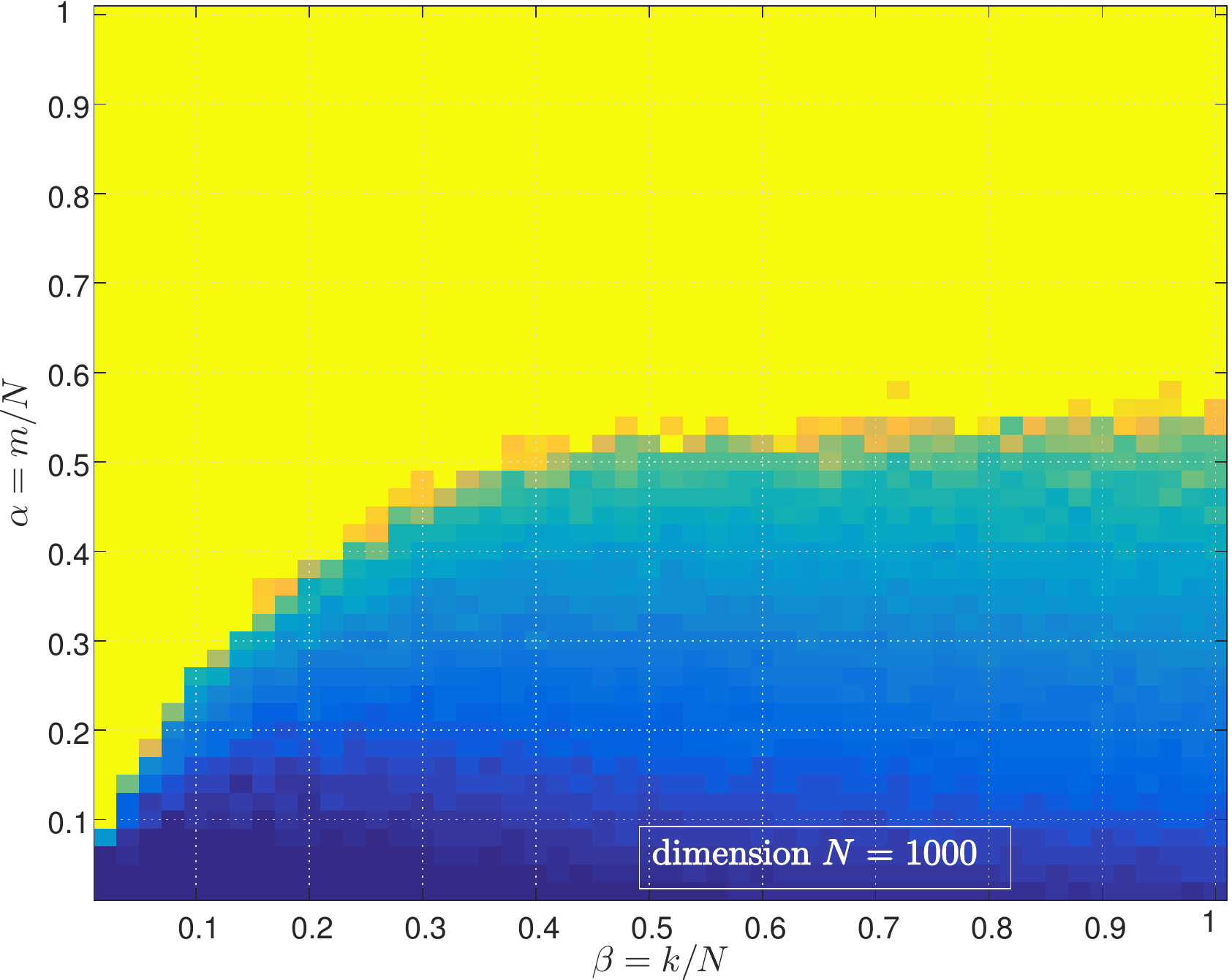}\includegraphics[scale=0.3]{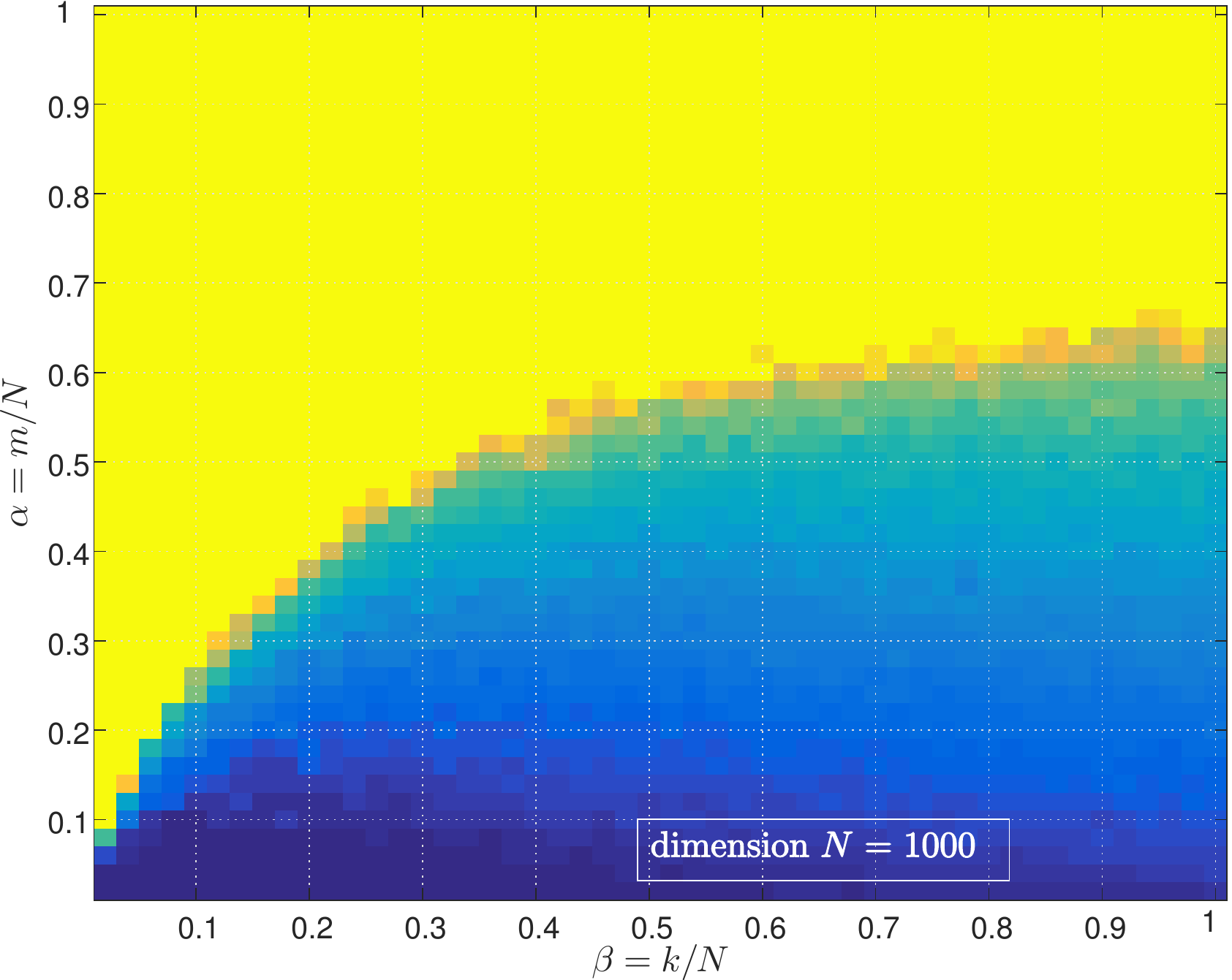}
\includegraphics[scale=0.3]{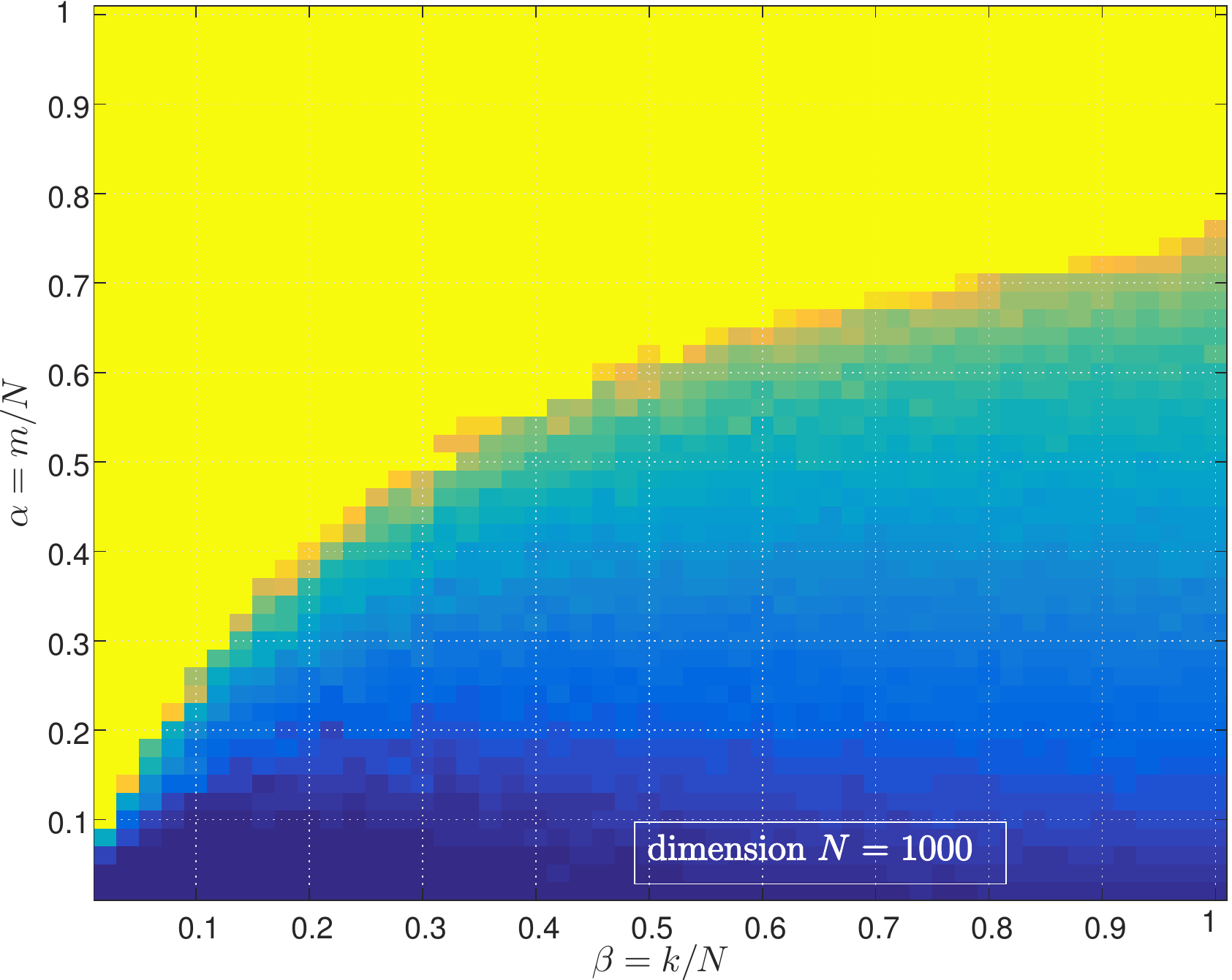}
\put(-369,-15){(a)}
\put(-226,-15){(b)}
\put(-76,-15){(c)}\\[2ex]
\includegraphics[scale=0.3]{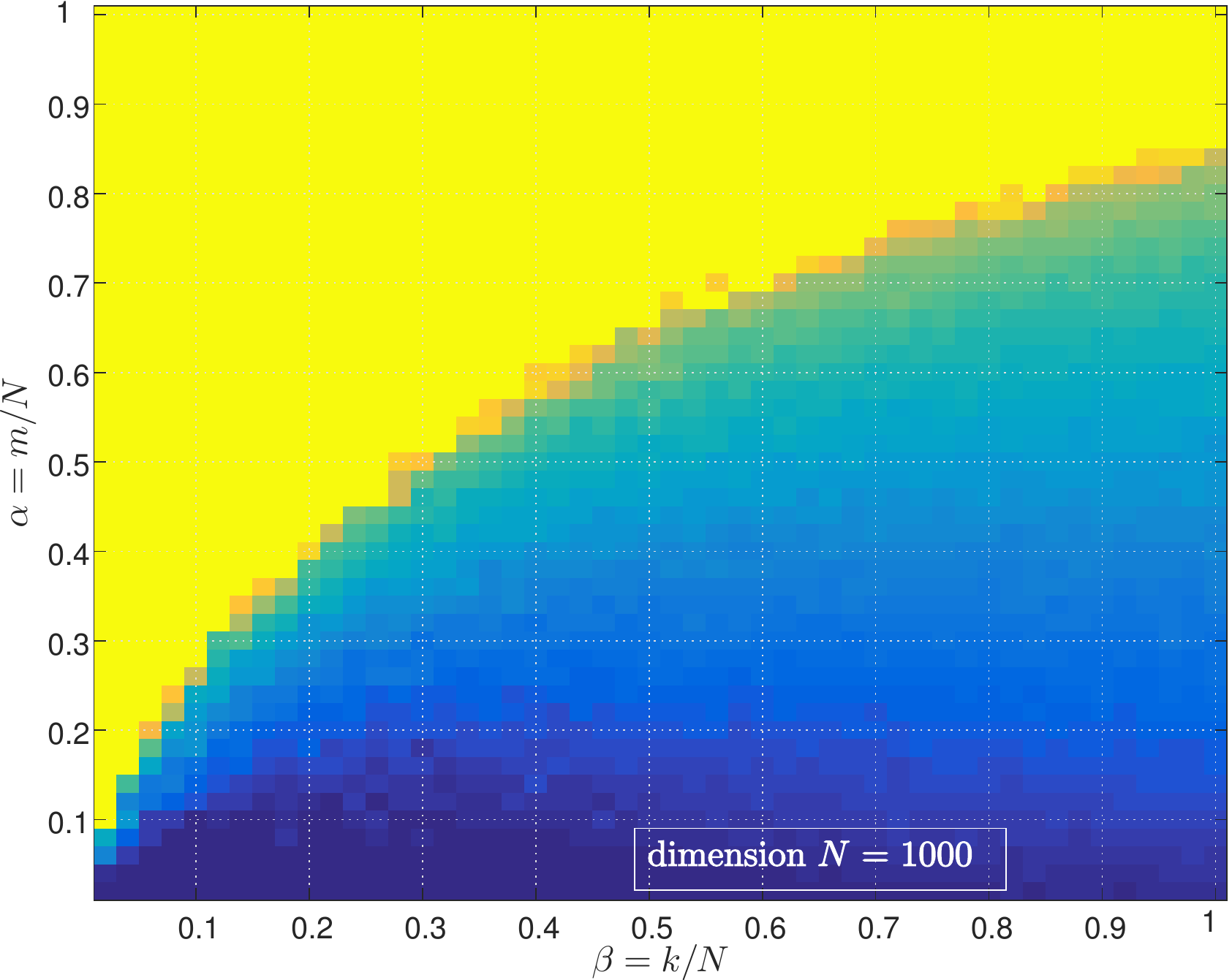}
\includegraphics[scale=0.3]{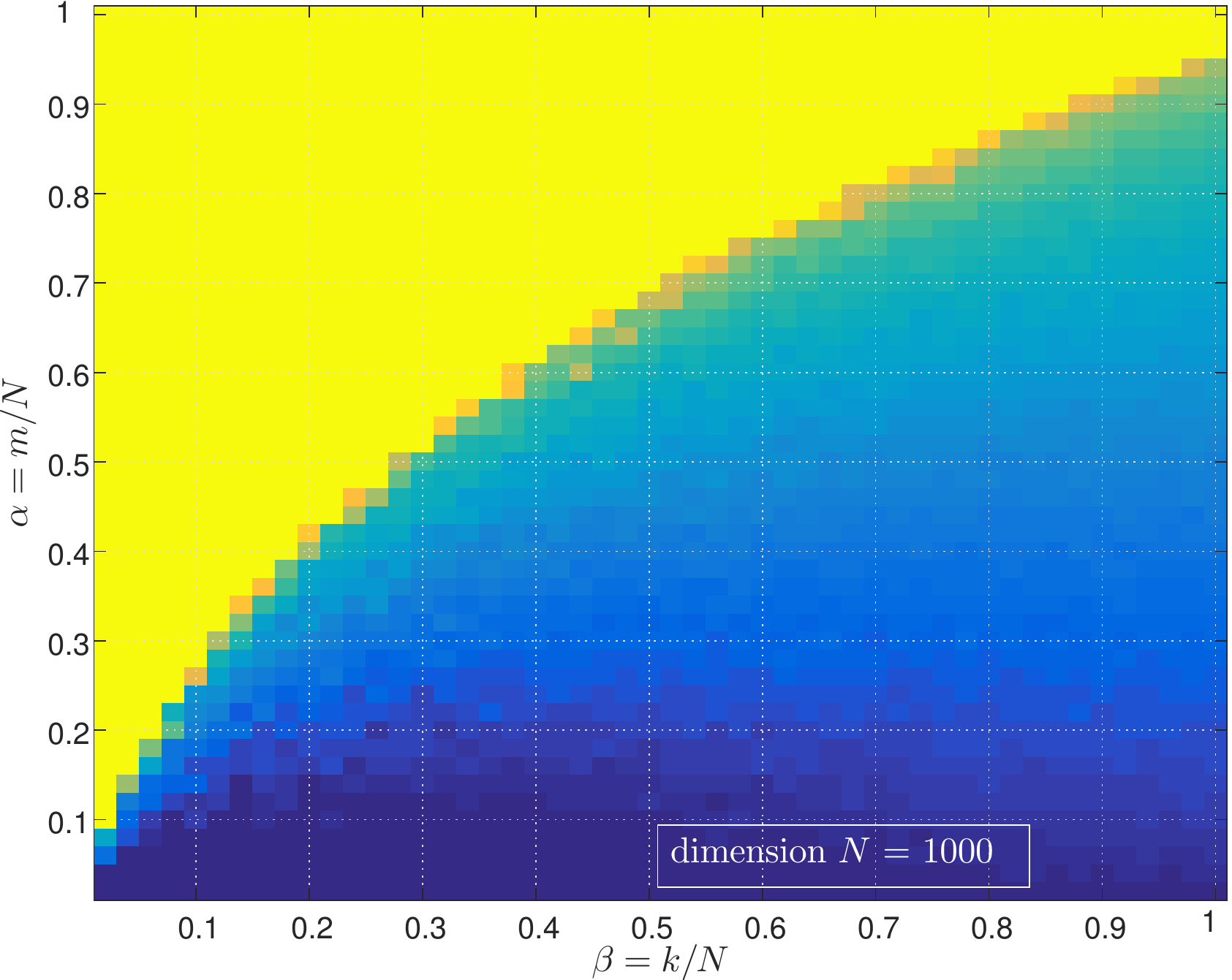}
\put(-226,-15){(d)}
\put(-73,-15){(e)}
\caption{Reconstruction error of the convex program \protect\eqref{PUF} for unipolar ternary vectors depending on the number of measurements $m$, the support size $k$ and the relative size of $k_1$ to $k$. We have chosen the ration $k_1/k$ in (a) as $k_1/k=0.1$, in (b) as $k_1/k=0.3$, in (c) as $k_1/k=0.5$, in (d) as $k_1/k=0.7$ and in (e) as $k_1/k=0.9$.}\label{fig:NumTer}
\end{center}
\end{figure}

We further tested the robustness of our algorithms numerically. For this purpose, we chose a noise level $\eta=0.3$, i.e., we performed
\eqref{PbinDR} with noisy measurements $b=Ax+e$ and $\|x\|_2\le0.3$ to check robustness. We chose as ambient dimension $N=500$. We then draw for each sparsity level
$k\in \{10,20, \dots, 500\}$ and number of measurements $m\in \{10,20, \dots, 500\}$, a set $K$ of size $k$ uniformly at random as well
as a Gaussian matrix $A\in \R^{m \times N}$, and constructed $x_0$ as in the noiseless case. Figure \ref{fig:NumDen} shows the results and particularly illustrates that basis pursuit with box constraints is robust to noise consider additionally inequality constraints. Note that as indicated by the theoretical results (c.f. Theorems \ref{thm:PhaseTransNoise} and \ref{thm:SBBPD1}), we might choose $\eta$ smaller in smaller ambient dimension to get comparable result. Thus, in dimension $N=1000$, we might have seen that less measurements are necessary for \eqref{PbinDR} and for \eqref{SBBPDenoising} to succeed in the case of the same noise level $\eta=0.3$. 
\begin{figure}[ht]
\begin{center}
\includegraphics[scale=0.37]{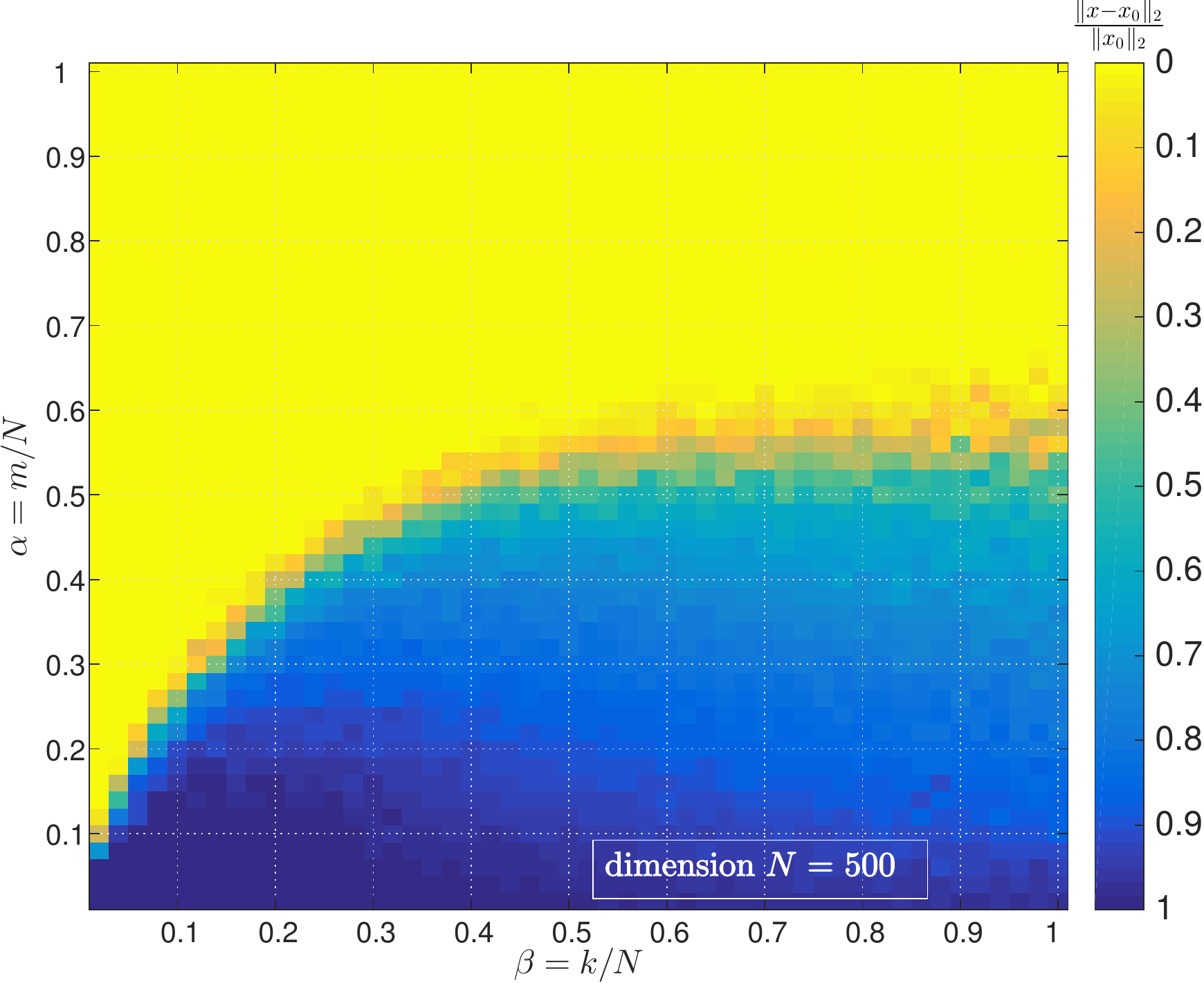}\hspace*{1cm}\includegraphics[scale=0.37]{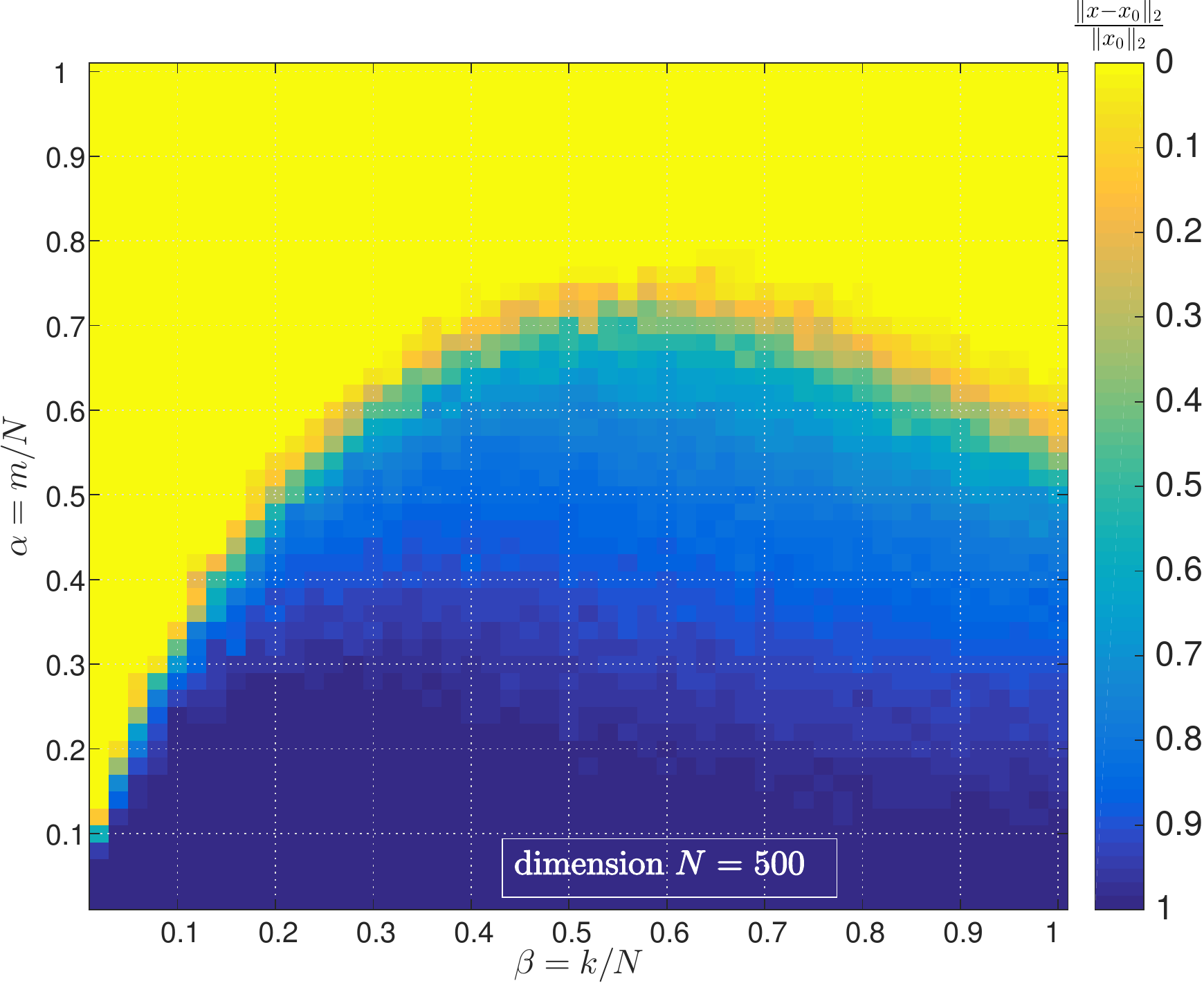}
\put(-310,-15){(a)}
\put(-96,-15){(b)}
\caption{Reconstruction error of the convex programs \protect\eqref{PbinDR} for unipolar binary vectors (see (a)) and of \protect\eqref{SBBPDenoising} for bipolar ternary
vectors (see (b)) depending on the number of measurements $m$ and the support size $k$ with an measurement error $e$ of $\|e\|_2\le 0.3$.}\label{fig:NumDen}
\end{center}
\end{figure}

\subsection*{Acknowledgements}

G.~Kutyniok acknowledges partial support by the Einstein Foundation Berlin, the Einstein Center for Mathematics Berlin (ECMath), the
European Commission-Project DEDALE (contract no. 665044) within the H2020 Framework Program, DFG Grant KU 1446/18, DFG-SPP
1798 Grants KU 1446/21 and KU 1446/23, the DFG Collaborative Research Center TRR 109 Discretization in Geometry and Dynamics,
and by the DFG Research Center {\sc Matheon} `Mathematics for Key Technologies' in Berlin.
S.~Keiper acknowledges support by the DFG Grant 1446/18.
G.E.~Pfander and D.G.~Lee acknowledge support by the DFG Grant PF 450/6 and would like to thank Tianlin Liu, Palina Salanevich and Weiqi Zhou for helpful discussions.

\begin{appendix}

\section{Proofs}

\subsection{Proof of Proposition \ref{prop:classicL1}} \label{sec:app0}

\begin{proof}That
(ii) follows from (i) is immediate.

To prove (ii) $\Rightarrow$ (i), let us assume that $\mathds{1}_K$ is the unique solution of \eqref{P+}. We will show that in this
case $A$ fulfills the NSP+ (cf. definition in Theorem \ref{thm:Uniform}~(iii)). Since it is known from \cite{Stojnic+} that this condition is sufficient to uniquely recover every positive-valued signal supported on $K$, the statement (i) follows. Thus we need to prove that for  $u\in \ker(A)\cap\{w\in \R^N:w_i\ge0 \quad \text{for} \quad i\in K^C\}$, we have $\sum_{i}^Nu_i>0$.

 The scaling invariance of this NSP condition allows us to assume $\|u\|_{\infty}\le 1$. It then follows that $\mathds{1}_K+u$ is feasible for \eqref{P+} and therefore $|K|=\|\mathds{1}_K\|_1< \|\mathds{1}_K+u\|_1=|K|+\sum_{i=1}^Nu_i$. Hence, $\sum_{i=1}^Nu_i> 0$, which concludes the proof.
\end{proof}

\subsection{Proof of Theorem \ref{thm:SBBPD1}} \label{sec:app1}

Before proving Theorem \ref{thm:SBBPD1}, we require the following lemma, which provides some equivalent properties
of the RBT-NSP. 

\begin{lemma}\label{lem:a}
Let $A \in \R^{m \times N}$, $K \subset [N]$, $0 < \rho < 1$, and $\tau > 0$. Then the following conditions are equivalent:
\begin{itemize}
\item[(i)]
$A$ satisfies the RBT-NSP with $0 < \rho < 1$ and $\tau > 0$ relative to $K$, that is, \eqref{RSBNSP} holds.
\item[(ii)]
For any $z \in \R^{N}$ with $z\in [-1,1]^N$,
\[  |K| -\sum_{i\in K_{1}}z_i+\sum_{i\in K_{-1}}z_i \le \rho \|z_{K^C}\|_1+ \tau\|A(z - x_\pm) \|_2 .
\]
\end{itemize}
\end{lemma}

\begin{proof}
(i) $\Rightarrow$ (ii): This follows by applying the condition for the (RBT-NSP) to $v = z - x_\pm$.

(ii) $\Rightarrow$ (i): Let $v \in H_{K_1,K_{-1}} \setminus \{ 0 \}$. (If $v = 0$, then \eqref{RSBNSP} holds trivially.)
For sufficiently small $t > 0$, $z^t= x_\pm + tv$ satisfies $z^t\in [-1,1]^N$. Thus, by (ii), we obtain that
\[
 |K|-\sum_{i\in K_{1}}z^t_i+\sum_{i\in K_{-1}}z^t_i\le \rho\|z^t_{K^C}\|+ \tau \| A (z^t-x_\pm) \|_2
 \]
is equivalent to
\[
 |K|-\sum_{i\in K_{1}}(1+tv_i)+\sum_{i\in K_{-1}}(-1+tv_i)\le \rho\|tv_{K^C}\|+ \tau \| A (tv) \|_2.
\]
But this holds if and only if
\begin{align}
t\sum_{i\in K_{-1}}v_i-t\sum_{i\in K_{1}}v_i \le \rho t \sum_{i\in K^C} |v_i|+t\tau \|Av\|_2,
\end{align}
which is equivalent to the desired inequality in (ii).
\end{proof}

\begin{proof}[Proof of Theorem \ref{thm:SBBPD1}]
Let $z \in \R^{N}$ be such that $z\in [-1,1]$. Then, by Lemma \ref{lem:a}, 
\begin{align}
|K| -\sum_{i\in K_{1}}z_i+\sum_{i\in K_{-1}}z_i \le \rho \|z_{K^C}\|_1+ \tau\|A(z - x_\pm) \|_2.
\end{align}
Hence, 
\begin{align}
\|(z - x_\pm)_K\|_1&= \sum_{i\in K_{1}}(1-z_i)-\sum_{i\in K_{-1}}(-1-z_i)=|K|-\sum_{i\in K_{1}}z_i+\sum_{i\in K_{-1}}z_i\\&\le \rho \| z_{K^C} \|_1 + \tau \| A (z - x_\pm) \|_2.\label{App1}
\end{align}
We also have
\begin{align}
|K| -\sum_{i\in K_{1}}z_i+\sum_{i\in K_{-1}}z_i- \|z_{K^C}\|_1 \le (\rho-1) \|z_{K^C}\|_1+ \tau\|A(z - x_\pm) \|_2,\end{align}
which yields
\begin{align}
\|x_\pm\|_1-\|z\|_1\le (\rho-1) \|z_{K^C}\|_1+ \tau\|A(z - x_\pm) \|_2.
\end{align}
But this is in turn equivalent to
\begin{align}\label{App2}
\| z_{K^C} \|_1\le \frac{1}{ 1- \rho } \left( \|z \|_1 -\|x_\pm \|_1  + \tau \| A(z -x_\pm )\|_2 \right) .
\end{align}
From \eqref{App1} and \eqref{App2}, we can deduce
\begin{align}
\|z - x_\pm \|_1 &= \| (z - x_\pm)_{K} \|_1 + \| z_{K^C} \|_1 \\
&\le \rho \| z_{K^C} \|_1 + \tau \| A (z - x_\pm) \|_2 +\| z_{K^C} \|_1\\
&=(1+\rho) \| z_{K^C} \|_1+\tau \| A (z - x_\pm)\\
&\le \frac{1+\rho}{1-\rho} \left( \|z \|_1 -\|x_\pm \|_1  + \tau \| A(z -x_\pm )\|_2 \right)+\tau \| A (z - x_\pm)\|_2\\
&=\frac{1+\rho}{1-\rho} \left(\|z \|_1 -\|x_\pm \|_1\right)+\frac{2\tau}{1-\rho}\| A (z - x_\pm)\|_2,
\end{align}
where \eqref{App1} was applied in the second step and \eqref{App2} in fourth step. Now if $z = \hat{z}$ is a minimizer of 
\eqref{SBBPDenoising}, then $\| \hat{z} \|_1 \le \| \mathds{1}_K \|_1$, which implies
\begin{align}
\| \hat{z} -  \mathds{1}_K \|_1 \le \frac{2 \tau }{ 1- \rho } \| A (\hat{z} -  x_\pm) \|_2 .
\end{align}
Finally, noticing $\| A (\hat{z} - x_\pm) \|_2 \le 2 \eta$ provides the desired error bound.
\end{proof}

\subsection{Proof of Lemma \ref{lem:subUF}}\label{subsecAStat1}
We will only prove Equation \eqref{eq:subUF}, the remaining is straightforward.
First notice that the subdifferential of $f$ defined in Equation \eqref{fUF} simplifies to
\begin{align}
\partial f(\hat{x})=\{s\in \R^N: \|y\|_1\ge \|\hat{x}\|_1 + \langle s, y-\hat{x} \rangle \text{ for all } y\in [0,L]^N\},
\end{align}
since for given $y\notin [0,L]^N$, the inequality $f(y)\ge f(\hat{x}) + \langle s, y-\hat{x}\rangle$ is vacuously true for all $s\in\R^N$.
Thus
\begin{align}
\partial f(\hat{x}) &=\{s\in \R^N: \sum_{i=1}^N (1-s_i) y_i + \sum_{i\in K}(s_i-1) \hat{x}_i\geq0,  y\in [0,L]^N \}.
\end{align}
Next, for the sake of brevity, define
\[
S := \{s\in \R^N: s_i \ge 1 \text{ for } i\in K_L,s_i=1 \text{ for } i\in K\setminus K_L \text{ and } s_i\le 1 \text{ for } i\in K^C \}.
\]
To prove $S \subseteq \partial f(\hat{x})$, let $s\in S$ and observe
\begin{align}
\sum_{i=1}^N (1-s_i) y_i + \sum_{i\in K}(s_i-1)\hat{x}_i
&=\sum_{i\in K_L}(1-s_i)y_i + \sum_{i\in K^C}(1-s_i)y_i+\sum_{i\in K_L}(s_i-1)L \\
&= \sum_{i\in K_L}(1-s_i)(y_i-L) + \sum_{i\in K^C}(1-s_i)y_i \\
&\ge 0 +\sum_{i\in K^C}(1-s_i)y_i\\
&\ge0
\end{align}
where we used in the first step that $(1-s_i)= 0$ for $i\in K\setminus K_L$ and $\hat{x}_i=l$ for $i\in K_L$, in the third step that $L\ge y_i$ and $s_i\ge1$ for $i\in K_L$ and in the last step that $s_i\le 1$ for $i\in K^C$.
Thus, indeed, $S \subseteq \partial f(\mathds{1}_K)$.

To prove  the other inclusion, i.e., $\partial f(\hat{x}) \subseteq S$, we consider the following points $y^j\in [0,L]^N$ for $j\in [N]$:
\begin{align}
y^j= \left\{
   \begin{array}{ll} \hat{x} - \hat{x}_je_j:  \quad &  j \in K,\\[3pt]
                     \hat{x} + e_j: \quad & j\in K^C
   \end{array}
   \right.
\end{align}
and for $j\in K\setminus K_L$ additionally the points:
\begin{align}
z^j= \hat{x}+(L-\hat{x}_j)e_j.
\end{align}
Since $y^j\in [0,L]^N$ for all $j$, every $s \in \partial f(\hat{x})$ must satisfy
\begin{align}\label{subUF}
0\le \sum_{i=1}^N (1-s_i) y^j_i + \sum_{i\in K}(s_i-1)\hat{x}_i, \quad j=1,\dots, N.
\end{align}
For $j\in K$, this inequality is equivalent to
\begin{align}
0\le (s_j-1)\hat{x}_j,
\end{align}
which yields equivalently $s_j \geq 1$. And for $j\in K^C$, we have
\begin{align}
0\le (1-s_j),
\end{align}
which is equivalent to $s_j \le 1$. 
Additionally for $j\in K\setminus K_L$ replacing $y^j$ with $z^j\in [0,L]^N$ in \eqref{subUF} yields
\begin{align}
 0\le(1-s_j)(L-\hat{x}_j),
\end{align}
and, hence, $s_j\le1$, because $\hat{x}_j<L$ for $j\in K\setminus K_L$. But we have already seen that $s_j\ge1$ for $j\in K$ and, hence, it needs to hold $s_j=1$ for $j\in K\setminus K_L$.
This implies $\partial f(\hat{x}) \subseteq S$. The claim is proven.

\subsection{Proof of Lemma \ref{lem:subBF}}\label{subsecAStat2}
To  prove Lemma \ref{lem:subBF} we first observe that the subdifferential simplifies to
\begin{align}
\partial f(x_0)=\{s\in \R^N: \|y\|_1\ge \|x_0\|_1 + \langle s, y-x_0 \rangle \text{ for all } y\in [-L_1,L_2]^N\},
\end{align}
where $x_0$ is defined as in Equation \eqref{Not:1} and $f$ as in Equation \eqref{fBF}.
Thus
\begin{align}
\partial f(x_0) &=\{s\in \R^N: \sum_{i=1}^N|y_i|-\sum_{i=1}^N s_i y_i + \sum_{i\in K}s_ix_{0,i}- \sum_{i\in K}|x_{0,i}|\geq0,  y\in [-L_1,L_2]^N \}.
\end{align}
Next, for the sake of brevity, define
\begin{align}
S := \left\{s\in \R^N: s_i \ge 1 \text{ for } i\in K_{L_2}, s_i=1 \text{ for } i\in \hat{K}^+,
|s_i|\le 1 \text{ for } i\in K^C, \right.\\ \left. s_i=-1 \text{ for }i\in \hat{K}^- \text{ and } s_i\le -1 \text{ for } i\in K_{-L_1}\right\},
\end{align}
where $\hat{K}^+=\bigcup_{i=1}^{L_2-1}K_i$ and $\hat{K}^-=\bigcup_{i=-L_1+1}^{-1}K_i$.
To prove $S \subseteq \partial f(x_0)$, let $s\in S$ and show that for all $y\in [-L_1,L_2]^N$ it holds
\begin{align}\label{sumBF}
0\le \sum_{i=1}^N|y_i|-\sum_{i=1}^N s_i y_i + \sum_{i\in K}s_ix_{0,i}- \sum_{i\in K}|x_{0,i}| =\sum_{i=1}^Na_i,
\end{align}
where $a_i:=|y_i|-s_i y_i + s_ix_{0,i}-|x_{0,i}|$ for $i\in [N]$. We will prove that $a_i\ge 0$ for all $i\in [N]$.

For $i\in K_{L_2}$ and $y_i\le 0$ we have
    \begin{align}
     a_i=|y_i|-s_iy_i+(s_i-1)L_2=-(1+s_i)y_i+(s_i-1)L_2\ge 0,
    \end{align}
because $s_i\ge 1$. For $i\in K_{L_2}$ and $y_i\ge 0$ it follows
\begin{align}
a_i=(s_i-1)(L_2-y_i)\ge0, 
\end{align}
because $s_i\ge 1$ and $y_i\le L_2$.
For $i\in\hat{K}^+$ we have
\begin{align}
 a_i=|y_i|-y_i+x_{0,i}- x_{0,i}=|y_i|-y_i\ge0,
\end{align}
where we used that $s_i=1$ and $x_{0,i}\ge0$.
For $i\in K^C$ we have
\begin{align}
 a_i=|y_i|-s_iy_i\ge0,
\end{align}
because $x_{0,i}=0$ and $|s_i|\le1$.
For $i\in\hat{K}^-$ we have
\begin{align}
 a_i=|y_i|+y_i-x_{0,i}- x_{0,i}=|y_i|+y_i\ge0,
\end{align}
because $s_i=-1$ and $x_{0,i}\le0$.
For $i\in K_{-L_1}$ and $y_i\le0$ we have
\begin{align}
     a_i=|y_i|-s_iy_i-s_iL_1-L_1=-(1+s_i)y_i-(s_i+1)L_1=-(1+s_i)(y_i+L_1)\ge0,
    \end{align}
because $s_i\le -1$ and $y_i\ge -L_1$. For $i\in K_{-L_1}$ and $y_i\ge 0$ it follows
\begin{align}
a_i=(1-s_i)y_i-(s_i+1)L_1\ge0,
\end{align}
because $s_i\le -1$.
Thus it is proven that $\sum_{i=1}^Na_i\ge 0$ and therefore $s\in \partial f(x_0)$.

To prove  the other inclusion, i.e., $\partial f(x_0) \subseteq S$, we consider the following points $y^j,zj\in [-L_1,L_2]^N$ for $j\in [N]$ 
\begin{align}
y^j= \left\{
   \begin{array}{ll} x_0 - x_{0,j}e_j  \quad & \text{for } j\in K,\\[3pt]
                     x_0 + e_j \quad &\text{for }j\in K^C,
   \end{array}
   \right. \quad \text{and} \quad
z^j= \left\{
   \begin{array}{ll} x_0 - (x_{0,j}- L_2)e_j  \quad & \text{for } j\in K^+,\\[3pt]
                     x_0 - e_j \quad &\text{for }j\in K^C,\\[3pt]
                     x_0 - (x_{0,j}+ L_1)e_j  \quad &\text{for }  j\in K^-,
   \end{array}
   \right.
\end{align}
where $K^+=\hat{K}^+\cup K_{L_{2}}$ and $K^-=\hat{K}^-\cup K_{-L_{1}}$. Note that $z^j=x_0$ for $j\in K_{L_2}\cup K_{-L_1}$.
We now substitute $y$ in Equation \eqref{sumBF} with these vectors. It is then easy to see that $s\in S$.
This implies $\partial f(\hat{x}) \subseteq S$ and the claim is proven.

\end{appendix}

\nocite{*}
\bibliographystyle{amsplain}
\bibliography{D-NSP_v3.4.bib}

\end{document}